\documentclass[12pt]{amsart}
\usepackage{amsmath,amssymb,amsthm,amscd}

\usepackage[all]{xy}

\title[An excision theorem 
 for the K-theory of $C^{*}$-algebras]{An excision theorem 
 for the K-theory of $C^{*}$-algebras, with 
 applications to groupoid $C^{*}$-algebras}
 
\author{Ian F. Putnam}
 
 \address{Department of Mathematics and Statistics,\\
 University of Victoria, \\
 Victoria, B.C. \\
 Canada}
 \email{ifputnam at uvic.ca}
 
 \date{August 21, 2020}
\thanks{Supported by a Discovery Grant
 from the Natural Sciences and Engineering Research Council, Canada}

\newcommand{\N}{\mathbb{N}}
\newcommand{\Z}{\mathbb{Z}}

\newcommand{\R}{\mathbb{R}}
\newcommand{\C}{\mathbb{C}}

\newcommand{\tA}{\widetilde{A}}
\newcommand{\ta}{\widetilde{A'}}
\newcommand{\tB}{\widetilde{B}}

\newtheorem{defn}{Definition}[section]
\newtheorem{thm}[defn]{Theorem}
\newtheorem{ex}[defn]{Example}
\newtheorem{prop}[defn]{Proposition}
\newtheorem{cor}[defn]{Corollary}
\newtheorem{lemma}[defn]{Lemma}

\newtheorem{rmk}[defn]{Remark}

\begin{document}

\begin{abstract}
We discuss the relative K-theory for a $C^{*}$-algebra, $A$, together 
with a $C^{*}$-subalgebra, $A' \subseteq A$. The relative group 
is denoted $K_{i}(A';A), i = 0, 1$, and is due to Karoubi. 
 We present a situation of two pairs $A' \subseteq A$ and $B' \subseteq B$
 are related so that there is a natural isomorphism 
 between their respective relative K-theories. We  also discuss 
 applications to the case where $A$ and $B$ are $C^{*}$-algebras
 of a pair of locally compact, Hausdorff topological groupoids, with 
 Haar systems. 
\end{abstract}

\maketitle

\section{Introduction}
\label{1}

The  goal of this paper is the computation of K-theory groups
of the reduced $C^{*}$-algebras of groupoids, meaning locally compact, Hausdorff groupoids with a Haar system. 
 To be more specific, we will be concerned with a pair of 
 groupoids which are related in some way so that one 
 reduced $C^{*}$-algebra is a $C^{*}$-subalgebra of the other. 
 Our results  allow
 computation of the relative K-theory  of this pair.
 
 Results along this line have already been obtained in \cite{Put:exc} 
 and \cite{Put:grpd}, but under very restrictive hypotheses.
 In particular, the groupoids there are principal and \'{e}tale.
 Moreover, the relation between the pair of groupoids is very limited.
 Our aim here is to extend the generality of these results.
 At the same, we give a much simpler, more conceptual 
 description of the isomorphism between relative groups that
 is our main objective.
 
 As the theory of groupoid $C^{*}$-algebras becomes rather technical
 quite quickly, we will devote this section to a discussion of the 
 the principal ingredients in the paper along with a couple
 of rather simple examples which 
 nicely illustrate some of the main ideas.
 
 Our first key ingredient is the notion of \emph{relative}
 K-theory. Given a $C^{*}$-algebra, $A$, and a $C^{*}$-subalgebra, 
 $A'$, Karoubi \cite{Ka:book}
 defined relative groups $K_{i}(A';A), i = 0,1$. We will 
 review his definition (at least for $i=0$) in the next section.
 (This can be defined for any 
 $*$-homomorphism $\alpha: A' \rightarrow A$. We will not need the definition
 in this generality, but
  we refer the reader to \cite{Ha:relK} for
 more details.)
 A key consequence is the existence of a six-term exact sequence:
 
 \vspace{.5cm}
\hspace{2cm}

\xymatrix{ 
  K_{0}(A';A)  \ar[r]^{\nu} & 
  K_{0}(A')  \ar[r]^{i_{*}}   
     & K_{0}(A)  \ar[d]^{\mu} \\
 K_{1}(A) \ar[u]^{\mu} &  
  K_{1}( A')  \ar[l]^{i_{*}} &
  K_{1}(A';A)  \ar[l]^{\nu} }
 
  \vspace{.5cm}
  where $i$ denotes the inclusion map.
  
 Our main results may be described as \emph{excision}, meaning 
 that the relative K-theory of a pair $A' \subseteq A$ depends 
 only on  $A - A'$. Of course, if $A', A$ are topological spaces
 $A - A'$ makes perfect sense as a topological space, but this doesn't
 make so much sense for $C^{*}$-algebras. In that setting, a 
 nice first example of excision is the following: suppose that 
 $A'$ is a closed two-sided ideal in $A$, then 
 $K_{i}(A';A) \cong K_{i+1}(A/A')$, where $A/A'$ is the usual quotient 
 $C^{*}$-algebra.
 
 Let us re-state that result in a way
  which will be useful
 for comparison later. Suppose that $A, B, C$ are 
 $C^{*}$-algebras and $\alpha: A \rightarrow C, \beta: B \rightarrow C$
 are $*$-homomorphisms. If $\alpha(A) = \beta(B)$, then the fact that
 the kernels are ideals implies that
 \[
 K_{*}(\ker( \alpha); A) \cong K_{*}(\alpha(A) ) = 
 K_{*}(\beta(B) )  \cong K_{*}(\ker( \beta); B).
 \]
 Our main results will be concerned with replacing $*$-homomorphisms
 in this statement with bounded $*$-derivations. At this point, we
 merely note that the kernel of a bounded $*$-derivation is
 a $C^{*}$-subalgebra, although not an ideal. 
 
 Let us give a very easy example using commutative $C^{*}$-algebras.
 Let $X$ be any compact, Hausdorff space. Choose two 
 distinct points, $y_{1}, y_{2}$, in $X$ and let $X'$ be the quotient space
 obtained by identifying them and $\pi:X \rightarrow X'$ be 
 the quotient map. This means that $\pi$ induces 
 an injection of $C(X')$ in $C(X)$. Alternately, 
 we can write
 \[
 C(X') = \{ f \in C(X) \mid f(y_{1}) = f(y_{2}) \}.
 \]
 
 These two algebras differ only at the points $Y = \{ y_{1}, y_{2} \}$
 in $X$, or at $Y' = \{ [y_{1}] \}$ in $X'$. We have the diagram:
 \vspace{.5cm}
 
 \hspace{1cm}
 \xymatrix{
 0 \ar[r] & C_{0}(X-Y) \ar[r] & C(X) \ar[r] & C(Y) \ar[r] & 0 \\
 0 \ar[r] & C_{0}(X'-Y') \ar[r]  \ar[u]^{=} &
  C(X') \ar[r]  \ar[u]^{\subseteq}  & C(Y') \ar[r] \ar[u]^{\subseteq} & 0 }
 \vspace{.5cm}
 
 \noindent
 which is commutative and has  exact rows.
There is an associated six-term exact sequence involving the three
relative K-groups, and since in one of them, the inclusion
is actually an equality, that relative group is zero. We conclude that
there is an isomorphism
\[
K_{i}(C(X');C(X)) \cong  K_{i}(C(Y');C(Y)), i = 0,1.
\]
The latter group is rather easy to compute using the exact sequence above.

This computation can also be regarded as an excision result. While considering
 this example, let us see how the notion of 
$*$-derivation can be useful. Assuming that $X$ is separable, 
let  us choose a countable dense
subset, $Y \subset Z \subset X$. There is an obvious
representation of $C(X)$ on $\ell^{2}(Z)$ and a slightly
less obvious one of $C(Y)$, which is zero on $\ell^{2}(Z - Y)$.
Let $F$ be the self-adjoint operator which is the identity
on $\ell^{2}(Z - Y)$ and such that $F \xi(y_{i}) = \xi(y_{3-i})$, 
for $i=1,2$ and $\xi$ in $\ell^{2}(Z)$. Also, let 
$\delta(a) = i [F, a]$, for any bounded operator $a$ on $\ell^{2}(Z)$.
Observe that $\delta$ is a bounded $*$-derivation with
\[
\begin{array}{rcl}
\ker(\delta) \cap C(X)&  = &  C(X') \\
\ker(\delta) \cap C(Y)&  = &  C(Y')  \\
\delta(C(X)) & = & \delta(C(Y)).
\end{array}
\]
Put in this way, the excision result above
 now looks similar to the earlier
result on the kernels of $*$-homomorphisms.

We now look at another example, which has many similarities with 
the last, but displays some important new features.
Let $X = \{ 0, 1 \}^{\N}$, $ X' = [0, 1]$ and $\pi:X \rightarrow X'$
be defined by 
\[
\pi(x) = \sum_{n=1}^{\infty} x_{n}2^{-n}, 
\]
for $x = (x_{n})_{n=1}^{\infty}$ in $X = \{ 0, 1 \}^{\N}$.
This can also be described as the restriction of the devil's staircase
to the Cantor ternary set. It is also known less formally as
 base $2$ expansion of real numbers.

We let $Y' = \{ k2^{-n} \mid n \geq 1, 0 < k < 2^{n} \}$ and 
$Y = \pi^{-1}(Y')$. It is a rather simple matter to check that 
$\pi$ is one-to-one on $X-Y$ and is two-to-one on $Y$. In fact, for $y'$ in 
$Y'$, $\pi^{-1}\{ y' \}$ consists of two points,
$(x_{1}, x_{2}, \ldots, x_{n}, 1, 0, 0, \ldots)$ and 
$(x_{1}, x_{2}, \ldots, x_{n}, 0, 1, 1, \ldots)$. At this point, 
the situation is very much like our last example. The significant difference
is that $Y$ and $Y'$ are no longer closed and the diagram we had 
above is no longer available.

The solution here is to introduce new topologies on $Y \subseteq X$ 
and $Y' \subseteq [0,1]$ which are finer
than the relative topologies from $X$ and $[0,1]$, respectively, in which
they are locally compact (and still Hausdorff). In this case, the 
obvious choice is the discrete topologies.  Let us continue the 
development with derivations we had in the earlier case. We represent
$C(X)$ and $C_{0}(Y)$ as multiplication operators on $\ell^{2}(Y)$
(noting that $Y$ is conveniently dense in $X$). Observe that $C(X)$ acts as 
multipliers of $C_{0}(Y)$. This is a result of the fact that for 
any $f$ in $C(X)$, its restriction to $Y$ will remain continuous
in any finer topology. Further, define $F$ to be the operator 
$F \xi(y) = \xi(\bar{y})$, where $\bar{y}$ is the unique point in $Y$ 
with $\pi(\bar{y}) = \pi(y)$ and $\bar{y} \neq y$. Again define 
$\delta(a) = i[F, a]$ for any bounded operator on $\ell^{2}(Y)$. 
It is a simple matter to check that 
\[
\begin{array}{rcl}
\ker(\delta) \cap C(X)&  = &  C(X') \\
\ker(\delta) \cap C_{0}(Y)&  = &  C_{0}(Y')  \\
\delta(C(X)) & = & \delta(C_{0}(Y)).
\end{array}
\]
In this situation, the conclusion that
\[
K_{i}(C(X');C(X)) \cong  K_{i}(C_{0}(Y');C_{0}(Y)), i = 0,1.
\]
follows from our main result, Theorem \ref{3:70}.

The key feature in this last example, which differs from the first we gave, 
is the idea that the subset where the two algebras of functions differ
must be endowed with a new, finer topology. At the same time, we are 
interested in groupoid $C^{*}$-algebras and the issue of endowing 
a subgroupoid with a new finer topology so that the original algebra
acts as multipliers of the smaller one is a considerable technical one.

The two examples we have listed above are part of a general class
which we refer to as 'factor groupoids'. We develop the theory in some generality
in section 7. The idea, avoiding many technical issues,
 is to take a surjective morphism
of groupoids $\pi: G \rightarrow G'$. Under some  hypotheses, we
show that this induces an inclusion
$C^{*}_{r}(G') \subseteq C^{*}_{r}(G)$. We then consider $H \subseteq G$ 
and $H' \subseteq G'$ to be the subgroupoids where the map $\pi$ fails to be 
one-to-one. Under a number of technical hypotheses,
 we first show that $H$ and $H'$ 
may be given new, finer topologies and prove that we have 
\[
K_{*}(C^{*}_{r}(G'); C^{*}_{r}(G)) \cong K_{*}(C^{*}_{r}(H'); C^{*}_{r}(H)),
\]
the latter being significantly simpler to compute in many examples.

The other situation which is considered in section 6 is one we refer to
an 'subgroupoids'. Here, we suppose that $G$ is a groupoid 
and $G^{0} \subseteq  G' \subseteq G$ is an open subgroupoid. Again we have 
an inclusion $C^{*}_{r}(G') \subseteq C^{*}(G)$. In this situation, we 
introduce $H \subseteq G$ 
and $H' \subseteq G'$ as the subgroupoids where the groupoids $G$ and $G'$ differ.
 Again, under a number of technical hypotheses we first show that $H$ and $H'$ 
may be given new, finer topologies and prove that we have 
\[
K_{*}(C^{*}_{r}(G'); C^{*}_{r}(G)) \cong K_{*}(C^{*}_{r}(H'); C^{*}_{r}(H)),
\]
the latter again 
being significantly simpler to compute in many examples.

Let us mention that, if the groupoids are all amenable, then 
the Baum-Connes conjecture holds \cite{Tu:BC}. It seems likely
that a proof of our results could be given by using this and conventional
excision results in topology. We believe there is some virtue
in working with the $C^{*}$-algebras themselves. 
In particular, this is preferable for doing the computations in 
most applications.

 Let us briefly mention some  applications of the results.
In some cases, these will follow from the earlier paper \cite{Put:grpd}.

For the subgroupoid situation, the simplest 
  example of this
  are the so-called orbit-breaking subalgebras,
   $A_{Y} \subseteq C(X) \ltimes \Z$
  first introduced in \cite{Put:ZCan}.  Indeed, we give a 
  considerable 
  generalization of this construction at the end of the section 
  6 in 
  Theorem \ref{6:250} and Corollary \ref{6:260}.
  
  Another application was given in \cite{Put:K}.  The main question is, 
  given some $K$-theory data, can one construct an \'{e}tale
   groupoid whose associated $C^{*}$-algebra falls in the 
   Elliott classification scheme and has the given $K$-theory groups.
   In this case, assuming the $K$-zero group is a simple , 
   acyclic dimension group
   $K$-one is torsion free,
   one begins with $G'$ as the AF-equivalence relation with that $K$-zero group
   and constructs $G' \subseteq G$ so that $K_{0}(G^{*}(G))$ remains the 
   same, while $K_{1}(C^{*}(G))$ becomes the desired $K$-one group.
   
The subgroupoid results are also used in \cite{DPS:2} for similar purposes, 
including the case of 
non-zero real rank $C^{*}$-algebras.

In \cite{DPS:NonH}, examples were given of non-homogeneous extensions of
minimal Cantor $\Z$-actions. The $K$-theory of these extensions can 
be computed in specific examples \cite{Ha:fac}, using the 
factor groupoid situation.
Additionally, \cite{Ha:fac} considers quotients which may be 
constructed rather analogously to the extensions given in \cite{Put:K}.

Finally, we mention work in progress with Rodrigo Trevi\~{n}o. 
This is based on work of Lindsey and Trevi\~{n}o \cite{LT:BD}
which begins with a bi-infinite ordered Bratteli diagram 
and constructs from it a flat surface with vertical foliation.
Typically, the surface is infinite genus. The foliation 
$C^{*}$-algebra is actually a subalgebra of the AF-algebra 
associated with the Bratteli diagram. In fact, their groupoids
can be related by a two-step process through a third 
groupoid. The first step is that the intermediate 
groupoid is a factor of the AF-equivalence relation. 
The second is that the foliation groupoid is a subgroupoid 
of the intermediate 
groupoid. An interesting consequence of these $K$-theory 
computations is that, if the $K$-zero group of the Bratteli diagram
is not finitely-generated, then the surface is necessarily
infinite genus.

The paper is organized as follows. The next section outlines 
basic facts about relative K-theory for $C^{*}$-algebras.
 The third section is our 
 excision result.
It is stated in considerable generality
 for derivations between $C^{*}$-algebras.
  Its proof is rather long and technical, 
  so it appears separately in section 4.
  
  In section 5, we turn to the rather general question: given 
  a groupoid, $G$, and  a subgroupoid, $H \subseteq G$, endowed with a finer 
  topology, what conditions ensure that the reduced groupoid 
  $C^{*}$-algebra of $G$ acts as multipliers of that of $H$? 
  
  In section 6, we combine the excision results of section three 
  and those of section
  five to consider the situation of an 
  open subgroupoid.
  In section 7, we do the same for the situation of a factor groupoid. 
  
  \vspace{1cm}
  
 \noindent
\textbf{Acknowledgements}

I am particularly grateful to Jonathan Rosenberg for directing me to Karoubi's 
book for relative K-theory. I am also grateful to Jean Renault and 
Dana Williams for  advice and insights on groupoids.
  I have also benefited 
from helpful conversations with Chris Bruce, Anna Duwenig and Mitch Haslehurst.

\section{Relative $K$-theory}
\label{2}

In this section, we discuss a relative K-theory for 
$C^{*}$-algebras introduced by Karoubi \cite{Ka:book}.
Most of the basic ideas are already in \cite{Ka:book}, but we 
will add to them slightly.

The  idea is to consider a $C^{*}$-algebra, $A$, together 
with a \newline
$C^{*}$-subalgebra, $A' \subseteq A$, and to define a 
relative group for the pair, denoted $K_{0}(A';A)$.
Let us mention that this definition has a 
generalization to the case where $\varphi: A' \rightarrow A$ is a 
$*$-homomorphism, but we need only consider the case when $\varphi$ is the 
inclusion map. More information can be found 
in \cite{Ka:book} or \cite{Ha:relK}.

 We also remark that there is a definition of a second
relative group $K_{1}(A';A)$ (perhaps it would be more accurate to say 
a sequence of relative groups satisfying Bott periodicity).  
Our results hold for these groups as well, but we will not need that 
for our applications. Again, we refer the reader to \cite{Ha:relK}
for more information.

Before beginning, we also remark that there is another
natural  description of the relative
K-theory by simply defining $K_{i}(A';A)$ to be $K_{i}(C(A';A))$, for 
$i=0,1$, where 
$C(A';A)$ is the mapping cone
\[
C(A';A) = \{ f : [0,1] \rightarrow A \mid f \text{ continuous }, f(0)=0, f(1) \in A' \}.
\]
In fact, this was the definition given in \cite{Put:exc}. The advantage
of Karoubi's definition is that the classes are presented in a much more 
convenient fashion. In \cite{Put:exc}, the first task was to give an equivalent
formulation in what are essentially Karoubi's terms. While the 
definition of our key map (Theorem \ref{3:10}) is given in fairly 
abstract terms, the proof of Theorem \ref{3:70} relies quite heavily 
on the presentation of the classes given by Karoubi. In addition, 
in most examples, it seems much simpler to use Karoubi's presentation.
In the end, this seems a much more natural, if somewhat longer,
 definition of relative K-theory.

Let $A$ be a unital $C^{*}$-algebra. As usual, for $n \geq 1$, 
we let $M_{n}(A)$
 be the $n \times n$-matrices over $A$, regarded as a $C^{*}$-algebra.
 We use the usual (non-unital) inclusions $M_{n}(A) \subseteq M_{n+1}(A)$, 
 for all $n \geq 1$ and let $\mathcal{M}(A)$ denote the union, regarded as a
 normed $*$-algebra. It is convenient to regard the
  elements of $\mathcal{A}$
 as matrices, indexed by the positive integers, with
  only finitely many non-zero entries.
 We also let $\mathcal{P}(A)$ denote the set 
 of all projections (self-adjoint idempotents)
  in $\mathcal{M}(A)$.
  
We consider the category
whose objects are the elements of $\mathcal{P}(A)$.
If $p$  and $q$ are in $\mathcal{P}(A)$, then 
the morphisms from $p$ to $q$ are the elements of  $q\mathcal{M}(A)p$.
Composition of morphisms is given by their product and the 
element $p$ is the identity morphism from $p$ to itself.
We denote this category by $\mathcal{P}(A)$. It is an additive category in 
an obvious sense. Moreover, each set of morphisms is actually a Banach
space in an obvious way and  $\mathcal{P}(A)$ is a quasi-surjective 
Banach category (II.2.1 and II.2.6 of \cite{Ka:book}). Then $K_{0}(A)$ 
is defined to be the $K$-theory of this category, as in II.1 of
\cite{Ka:book}. (Some caution must be used: a homotopy of 
morphisms in this category takes place inside a single
$q\mathcal{M}(A)p$, which is slightly different from a homotopy inside
$\mathcal{M}(A)$.)
 
 We let $\tA$ denote the unitization of $A$.
In the case that the $C^{*}$-algebra $A$ is not unital,  
$K_{0}(A)$ is defined as the kernel of the map induced from the usual
homomorphism from $\tA$ to $\C$. Conveniently, this 
conclusion also holds for unital $C^{*}$-algebras.

Every element of $\tA$ can be written as a sum of a
 complex multiple of the unit 
and an element of $A$. If $a$ is in $\tA$, we let 
$\dot{a}$ denote the 
complex number involved. We extend this notation
 to elements $a$
in matrices over $\tA$, so that $\dot{a}$ is a
 complex matrix of the same size.

We now suppose that $A$ is a $C^{*}$-algebra and 
$A'$ is a $C^{*}$-subalgebra. To define a relative group, we follow
the ideas of \cite{Ka:book}, using 
$\varphi: \mathcal{P}(A') \rightarrow
\mathcal{P}(A)$ being the inclusion map, but make some minor
alterations. First, we would like to include non-unital $A$ and $A'$, 
so we consider the obvious unital inclusion of $\ta $ in 
$\tA$. Second, we will suppress this map in our notation.
We consider triples $(p,q, a)$, where $p$ and $q$ are objects in 
$\mathcal{P}(\ta)$ and $a$ is an invertible morphism
 from $p$ to $q$ in $\mathcal{P}(\tA)$. 
 Specifically, if 
 $p$ is in $\mathcal{M}(\ta)$ and 
 $q$ is in $\mathcal{M}(\ta)$, then $a$ is in  
 $q\mathcal{M}(\tA)p$   and there is $b$ in 
$p\mathcal{M}(\tA)q$ such that $ab=q$ and $ba=p$.
We let $\Gamma(A';A)$ denote the set of all such triples.
Although it is likely to raise a storm of controversy, we
note that $0$ is an invertible morphism from $\{ 0 \}$ to $\{ 0 \}$, so 
$(0,0,0)$ is in   $\Gamma(A';A)$.

We say two such triples $(p,q,a)$ and $(p',q',a')$ are \emph{isomorphic}
if there are isomorphisms  $c$ from $p$ to $p'$ and $d$ from $q$ to $q'$
in
$\mathcal{P}(\ta)$ such that $da=a'c$. In this situation, we also say
they are \emph{isomorphic via} $c, d$. In particular, 
if $(p, q, a)$ is in $\Gamma(A';A)$ and $a$ lies 
in $\mathcal{M}(\ta)$, then $(p, q, a) $ is isomorphic 
to $(q, q, q)$ via the pair $a, q$.

Let us briefly mention that there are some small difficulties
in taking direct sums of elements of $\mathcal{M}(A)$:
if the elements are regarded as $a \in M_{m}(A)$ and $b \in M_{n}(A)$, 
then 
$a \oplus b  \in M_{m+n}(A)$. This is not quite 
consistent with the identification of elements
 of $M_{n}(A)$ with those
 in $M_{n+1}(A)$.
On the other hand, the result $a \oplus b$ is well-defined up to
isomorphism as above and this ambiguity will not cause any confusion.

A triple $(p, q, a)$ is 
\emph{elementary}
 if $p=q$ and $a$ is homotopic to $p$
  within the automorphisms 
of $p$ in $\mathcal{P}(\tA)$. If, in addition, 
 $a$ is actually a unitary in 
$p \mathcal{M}(\tA)p$, then there exists a homotopy  from $p$ to $a$ 
within the unitaries.  We also make the observation that $(p, p, a)$
is elementary if and only if $a$ is obtained
as an invertible element of the $C^{*}$-algebra
\[
\{ f \in C( [0, 1], p \mathcal{M}(\tilde{A})p) \mid f(0) \in \C p \},
\]
when evaluated at $1$.

Finally, we introduce an equivalence relation $\sim$ on 
$\Gamma(A';A)$ as follows. Two triples $(p,q,a) \sim (p',q',a')$ if 
there are elementary triples $(p_{1}, p_{1}, a_{1})$ and 
$(p_{2}, p_{2}, a_{2})$ such that 
\[
(p,q,a) \oplus (p_{1}, p_{1}, a_{1}) = 
(p \oplus p_{1}, q \oplus p_{1}, a \oplus a_{1} ) 
\]
is isomorphic to $(p', q', a') \oplus (p_{2}, p_{2}, a_{2})$.
Clearly, isomorphic triples are equivalent and any elementary triple
is equivalent to $(0,0,0)$.

We define $K_{0}(A';A)$ as the set of equivalence $\sim$-class
of the elements of  $\Gamma(A';A)$ 
by this equivalence relation. We denote the 
equivalence class of $(p, q, a)$ by
$[p, q, a]$.  It is a simple matter to check that 
\[
[p, q, a] + [p', q', a'] = [p + p', q + q', a + a'],
\]
for $(p, q, a), (p', q', a')$ in $\Gamma(A';A)$ with 
$pp'=qq'=0$, 
 is a well-defined
binary operation. 
Alternately, we could define
\[
[p, q, a] + [p', q', a'] = [p \oplus p', q \oplus q', a \oplus a'].
\]
The  element $[0,0,0]$ is the identity and any element 
$[p,q,a]$ has inverse $[q, p, b]$, where $b$ satisfies $ab=q, ba =p$.
Hence, $K_{0}(A';A)$ is a group.

Suppose that 
 $\pi: A \rightarrow B$ is any $*$-homomorphism between 
two $C^{*}$-algebras and $A' \subseteq A, B' \subseteq B$ are 
two subalgebras satisfying $\pi(A') \subseteq B'$, then we may extend 
$\pi$ to a unital map from $\tilde{A}$ to $\tilde{B}$
and to matrices over $\tilde{A}$ and it follows that 
$\pi$ induces a group homomorphism
$\pi_{*}: K_{0}(A';A) \rightarrow K_{0}(B';B)$.

We will not give a proof of the following result, but refer the reader to
\cite{Ha:relK}.

\begin{thm}
\label{2:5}
Let $A$ be a $C^{*}$-algebra and let $A'$ be a subalgebra.
There is an exact sequence
\vspace{.5cm}

\hspace{1cm}
\xymatrix{  K_{1}( A')  \ar[r]^{i_{*}} &
 K_{1}(A) \ar[r]^{\mu} & 
  K_{0}(A';A)  \ar[r]^{\nu} & 
  K_{0}(A')  \ar[r]^{i_{*}}  
     & K_{0}(A)}
    \vspace{.5cm}
     
\noindent where $i: A' \rightarrow A$ denotes the inclusion map,
 $\mu([u]_{1}) = [1_{n}, 1_{n}, u] $, for any unitary $u$ in 
 $M_{n}(\tA)$  and $\nu[p, q, a] = [p]_{0} - [q]_{0}$, for any 
 $(p,q,a)$ in $\Gamma(A'; A)$.
\end{thm}

\begin{ex}
\label{2:100}
Let $\mathcal{H}$ be a Hilbert space and let $\mathcal{N}$ 
be a closed subspace with $\mathcal{N}^{\perp}$
its orthogonal complement.
Assume that $\mathcal{N} \neq 0, \mathcal{H}$.
We consider $A = \mathcal{K}(\mathcal{H})$, the $C^{*}$-algebra
of  compact 
operators on $\mathcal{H}$, and 
$A' = \mathcal{K}(\mathcal{N}) \oplus \mathcal{K}(\mathcal{N}^{\perp})$.
 From the short
exact sequence in Theorem \ref{2:5} and the well-known result 
that $K_{0}(\mathcal{K}(\mathcal{H})) \cong \Z$, via the usual trace, and 
$K_{1}(\mathcal{K}(\mathcal{H})) \cong 0$, we see that 
 $K_{0}(A';A) $ is isomorphic to the kernel of the map $i_{*}$, which consists
 of pairs $([p]_{0} - [p']_{0}, [q]_{0} - [q']_{0})$, 
 where $p,p'$ are finite rank 
 projections on subspaces of $\mathcal{N}$, $q, q'$ are
  finite rank projections on 
 subspaces of $\mathcal{N}^{\perp}$ and 
 $Rank(p) - Rank(p') = Rank(q') - Rank(q)$.
 Associating to such an element $Rank(p) - Rank(p')$
  is an isomorphism from this 
 subgroup to $\Z$.

We can give a useful classification of many elements in the relative group
 as follows. Suppose  $S, S'$ are finite rank operators on $\mathcal{H}$ with
$S(\mathcal{N}) \subseteq \mathcal{N}^{\perp}$, 
$ S(\mathcal{N}^{\perp}) =0$ and 
and $S'(\mathcal{N})=0, S'(\mathcal{N}^{\perp})  \subseteq \mathcal{N}$. Define 
$P, Q, P', Q'$ as the 
projections onto the initial space of $S$, the range of $S$, 
the initial space of $S'$ and the range of $S'$, respectively.
Then $(P+P', Q+Q', S+S')$ determines an element of
$\Gamma(A';A)$.
The composition of the map $\nu$ with
 the isomorphism described above sends the class of
this element to
\[
Rank(P) - Rank(P') = Rank(S) - Rank(S').
\]
\end{ex}

\begin{ex}
\label{2:110}
Let $D$ denote any $C^{*}$-algebra. We 
consider $A= M_{2}(D)$, $ A' = D \oplus D$, with 
the obvious inclusion as diagonal matrices.
 Using the exact sequence from Theorem \ref{2:5}
it is easy to check that \newline
 $K_{*}(A) \cong K_{*}(D), K_{*}(A') \cong K_{*}(D) \oplus K_{*}(D)$
and, with these identifications, $i_{*}(x,y) = x+y, x,y \in K_{*}(D)$.
It follows at once that $\ker(i_{*}) \cong K_{*}(D)$.

If $n \geq 1$ and $p, q$ are two 
projections in $M_{n}(\tilde{D})$ with $\dot{p} = \dot{q}$, 
the triple
\[
\left(  \left[ \begin{array}{cc} p & 0 \\ 0 & q  \end{array} \right],
        \left[ \begin{array}{cc} q & 0 \\ 0 & p  \end{array} \right],
     \left[ \begin{array}{cc} 0 & q \\ p & 0  \end{array} \right], \right)
        \]
        is in $\Gamma(A';A)$. Its image 
        under $\nu$ is $([p]_{0} - [q]_{0}, [q]_{0} - [p]_{0})$
        in $ K_{0}(D) \oplus K_{0}(D)$ 
        and $\nu$ is an isomorphism between $K_{0}(A';A)$ and 
        $\ker(i_{*})$. We conclude in this situation that
        \[
        K_{0}(D \oplus D; M_{2}(D))   \cong K_{0}(D).
        \]
        
\end{ex}

\begin{ex}
\label{2:120}
Let $D$ denote any $C^{*}$-algebra. We consider $A=  D \oplus D$ and 
$ A' = \{ (d,d ) \in D \oplus D \mid d \in D \} \cong D$. Here, it is clear that
$K_{*}(A) = K_{*}(D) \oplus K_{*}(D), K_{*}(A') = K_{*}(D)$ and with
 these identifications, $i_{*}(x) = (x,x), x \in K_{*}(D)$. 
 If $n \geq 1$ and 
 $u$ is a unitary in $M_{n}(\tilde{D})$, then the triple
 $(1_{n} \oplus 1_{n}, 1_{n} \oplus 1_{n}, u \oplus 1_{n})$ is
  in $\Gamma(A';A)$ and the map sending $[u]_{1}$
 in $K_{1}(D)$ 
 to $[1_{n} \oplus 1_{n}, 1_{n}\oplus 1_{n}, u \oplus 1_{n}]$ 
 in $K_{0}(A';A)$ is an isomorphism.
  We conclude in this situation that
        \[
        K_{0}(D; D \oplus D)   \cong K_{1}(D).
        \]
\end{ex}

We want to establish a few simple properties of the triples 
under consideration.

\begin{prop}
\label{2:10}
Let $A$ be a $C^{*}$-algebra and $A'$ be a $C^{*}$-subalgebra of $A$.
\begin{enumerate}
\item
If $(p, q, a)$ is in $\Gamma(A';A)$ and $a$ lies 
in $\mathcal{M}(\ta)$, then $[p, q, a]=0 $. 
\item If $a(t), 0 \leq t \leq 1$ is a continuous path 
of invertible elements in $q\mathcal{M}(\tA)p$ with $p, q$ in 
$\mathcal{M}(\ta)$, for all $0 \leq t \leq 1$,
then $[p,q,a(0)] = [p,q,a(1)]$.
\item 
Every element of $\Gamma(A';A)$ 
is equivalent to one of the form \newline
$(1_{m}, q, a)$, where $a$ is a partial isometry with $a^{*}a=1_{m}$,
 $aa^{*} = q$ and $ \dot{a} = \dot{q} = 1_{m}$. 
We refer to such an element as being in standard form. 
 
\item 
For $(1_{m}, q, a)$  in $\Gamma(A';A)$ is in standard form,
 $[1_{m}, q, a] =0$ if and only if
there exists $n \geq 1$, and elementary  triples
$(1_{n}, 1_{n}, a_{1})$ and \newline 
$(1_{m+n}, 1_{m+n}, a_{2})$ 
in standard form 
such that $a_{2}(a \oplus a_{1})^{*}$ is in $\mathcal{M}(\ta)$.

\item 
If $(p,q,a), (p', q', b)$ are in $\Gamma(A';A)$ and satisfy $q=p'$, then 
\[
[p,q,a] + [p', q', b] = [p, q', ba].
\]
\item 
If $(p, q, a)$ is in $\Gamma(A', A)$, then $-[p, q, a] = [q, p, a^{*}]$.
\end{enumerate}
\end{prop}

\begin{proof}
The first part follows from the fact that, if $a$ is 
in $\mathcal{M}(\ta)$, then $(p,q,a)$ is 
isomorphic to $(q,q,q)$ via $a, q$ and the latter is elementary.

For the second part,  we consider
\[
\left[ \begin{array}{cc} a(0)^{-1}  & 0 \\ 0 & q \end{array} \right] 
\left[ \begin{array}{cc} \cos(t \pi) q  &  \sin(t \pi) q\\ 
-\sin(t \pi) q  & \cos(t \pi) q   \end{array} \right] 
\left[ \begin{array}{cc} a(t)  & 0 \\ 0 & q \end{array} \right] 
\]
is a continuous path of invertibles in 
$(p \oplus q)  \mathcal{M}(\tA)(p \oplus q)$ from $p \oplus q$ to
\[
 b  = \left[ \begin{array}{cc} 0 & a(0)^{-1}  \\ a(1)  & 0 \end{array} \right],
\]
implying that $(p \oplus q, p \oplus q, b)$ 
is elementary. In a similar way, if we let 
\[
  c= \left[ \begin{array}{cc} 0 & a(0)^{-1}  \\ a(0)  & 0 \end{array} \right],
\]
then $(p \oplus q, p \oplus q, c)$ is also elementary. It is then a 
simple computation to check that 
$(p,q, a(0)) \oplus (p \oplus q, p \oplus q, b)$
is isomorphic to $(p,q, a(1)) \oplus (p \oplus q, p \oplus q, c)$ via
\[
\left[ \begin{array}{ccc} 0 & p & 0 \\ p & 0 & 0 \\ 0 & 0 & q \end{array} \right],
\hspace{1cm}
\left[ \begin{array}{ccc} 0 & 0 & q \\ 0 & p & 0 \\ q & 0 & 0 \end{array} \right]
\]
The conclusion follows.

For the third part, we will first show that we can assume that $p=1_{m}$, for
some $m$. Suppose that $p$ lies in $M_{m}(\ta)$, for some $n$.
Consider the partial isometry
\[
b = \left[ \begin{array}{c} 0 \\ 1_{m}-p   \end{array} \right],
\]
which lies in $M_{2m,m}(\ta)$. It is clear that $b^{*}b= 1_{m}-p$ and 
$[1_{m}-p, bb^{*}, b] = 0$, from part one. We also have 
\[
[p,q,a] + [1_{m}-p, bb^{*}, b]  = [1_{m}, q + bb^{*}, a + b].
\]

We now simply write our element as $(1_{m}, q, a)$ and show that $a$ can be replaced
by a partial isometry. It is a standard fact that $a$ being invertible implies
that $a^{*}a$ is positive and invertible in $p\mathcal{M}(\tA)p$ and so 
$a(a^{*}a)^{2^{-1}(t-1)}, 0 \leq t \leq 1$ is a path of invertibles from
$a$ to $a (a^{*}a)^{-1/2}$ which is a 
 partial isometry in $ \mathcal{M}(\tA)$ and it follows from part 2 
 that $[1_{m}, q, a] = [1_{m}, q, a (a^{*}a)^{-1/2}]$.

Finally, we consider $(1_{m}, q, a)$ where $a$ is a partial 
isometry. It follows that $\dot{a}$ is a partial
 isometry in $\mathcal{M}(\C)$
with $\dot{a}^{*}\dot{a} = 1_{m}$. We
 may find a path $b(t), 0 \leq t \leq 1$, of partial 
isometries in  $\mathcal{M}(\C)$ from  $1_{m}$ to $\dot{a}$. Then 
$b(t)^{*}a$ is a path of partial isometries from $a$ to $c= \dot{a}^{*}a$
which satisfies $\dot{c} = 1_{m}$. The conclusion follows 
from part 2.

For the fourth part, the 'if' statement is clear. 
It follows from the definition that, 
if  $[1_{m},q,a]=0$, then we may find elementary triples
 $(p_{i}, p_{i}, a_{i}), i = 1,2 $, such that 
 $(1_{m},q,a) \oplus (p_{1}, p_{1}, a_{1})$ and 
$(p_{2}, p_{2}, a_{2})$ are isomorphic. Choose $n$ sufficiently large so that 
$p_{1}$ is in $M_{n}(\ta)$. Simply adding the elementary triple
$(1_{n}-p_{1}, 1_{n} -p_{1}, 1_{n}-p_{1})$ to both sides, we have 
 $(1_{m},q,a) \oplus (1_{n}, 1_{n}, a_{1} + 1_{n} - p_{1})$ is 
 isomorphic to an elementary triple, which we will still write as 
 $(p_{2}, p_{2}, a_{2})$. We also write $a_{1} + 1_{n} - p_{1}$ as simply
 $a_{1}$.
 
  Let us say that the isomorphism is via $b,c$.
 This means that $a_{2} b = c (a \oplus a_{1})$. It follows from the fact 
 that $a, a_{1}$ and $a_{2}$ are partial isometries that 
 $c^{*}c = a_{2}^{*}b^{*}b a_{2}$. It follows that 
 $(c^{*}c)^{-1/2} = a_{2}^{*}(b^{*}b )^{-1/2} a_{2}$ and hence
 the pair is also isomorphic via
 $b(b^{*}b )^{-1/2}, c(c^{*}c)^{-1/2} $, which are partial
  isometries. Henceforth, we
 write them as $b,c$ and assume they are partial isometries.
 If we simply consider $(1_{m+n}, 1_{m+n}, b^{*}a_{2}b)$, then we have 
\[
 b^{*}a_{2}b (a \oplus a_{1})^{*} = b^{*} c(a \oplus a_{1})(a \oplus a_{1})^{*}
     =b^{*}c
     \]
     which lies in $\mathcal{M}(\ta)$, as desired.

For the fifth part, 
we assume that $p,p',q,a,b$ are all contained in 
$M_{n}(\tA)$, for some $n \geq 1$. Consider the following product
\[
\left[ \begin{array}{ccc} 
p & 0 & 0 \\ 0 & b^{*} & 0 \\ 0 & 0 & q
\end{array} \right]
\left[ \begin{array}{ccc} 
1 & 0 & 0 \\ 0 & 0 & 1 \\ 0 & 1 & 0
\end{array} \right]
\left[ \begin{array}{ccc} 
p & 0 & 0 \\ 0 & ba & 0 \\ 0 & 0 & q
\end{array} \right]
\left[ \begin{array}{ccc} 
0 & 1 & 0 \\ 1 & 0 & 0 \\ 0 & 0 & 1
\end{array} \right]
\left[ \begin{array}{ccc} 
p & 0 & 0 \\ 0 & a^{*} & 0 \\ 0 & 0 & q
\end{array} \right]
\]
A direct computation shows the product is
\[
\left[ \begin{array}{ccc} 
0 & a^{*} & 0 \\ 0 &  0 & b^{*} \\ ba & 0 & 0
\end{array} \right].
\]
This is an invertible element of 
$(p \oplus p' \oplus q)M_{3n}(\tA)(p \oplus p' \oplus q)$.
On the other hand, the two matrices with complex entries 
are homotopic through such matrices to the identity 
and so this element is 
homotopic through 
invertibles in \newline
$(p \oplus p' \oplus q)M_{3n}(\tA)(p \oplus p' \oplus q) $  to
the product 
 \[
\left[ \begin{array}{ccc} 
p & 0 & 0 \\ 0 & b^{*} & 0 \\ 0 & 0 & q
\end{array} \right]
\left[ \begin{array}{ccc} 
p & 0 & 0 \\ 0 & ba & 0 \\ 0 & 0 & q
\end{array} \right]
\left[ \begin{array}{ccc} 
p & 0 & 0 \\ 0 & a^{*} & 0 \\ 0 & 0 & q
\end{array} \right] = 
\left[ \begin{array}{ccc} 
p & 0 & 0  \\ 0 & b^{*}baa^{*} & 0 \\ 0 & 0 & q
\end{array} \right].
\]
This  is 
 homotopic to $p \oplus p' \oplus q$ by using the homotopy
 $(b^{*}b)^{t}(aa^{*})^{t}, 0 \leq t \leq 1$.
In a similar way, the element 
\[
\left[ \begin{array}{cccc} 
0 & a & 0 & 0 \\ a^{*} & 0 & 0 & 0 \\ 0 & 0 & 0 & b \\ 0 & 0 & b^{*} & 0
\end{array} \right] 
\]
is also elementary.
We add the former to $a \oplus b$ and the
 latter to $ab $ and the results  
 are clearly isomorphic. The conclusion 
 follows.

For the final part, we use part 5 to compute
\[
[p,q,a] + [q,p,a^{*}] = [p,p,a^{*}a]
\]
and $(a^{*}a)^{t}, 0 \leq t \leq 1$ is a homotopy between $a^{*}a$ and $p$, 
so the last triple is elementary.
\end{proof}

As a final item for this section, we need a result that 
links relative groups for a pair
of short exact sequences. This will be used in a key way in the next section
in defining our excision map.

    \begin{lemma}
    \label{2:60}
Let
\vspace{.5cm}

\hspace{3cm}
\xymatrix{ 0 \ar[r]  &  A \ar[r]^{\pi} & C \ar[r]^{\rho}  &  D  \ar[r] &  0 \\
   0 \ar[r]  & A' \ar@{}[u]|-*[@]{\subseteq} \ar[r]^{\pi}   & C' \ar[r]^{\rho}
   \ar@{}[u]|-*[@]{\subseteq}
     & D \ar[r] \ar@{}[u]|-*[@]{=} & 0 }
   \vspace{.5cm}
  
 \noindent    be a commutative diagram with exact rows. 

If $c$ is a partial isometry in $M_{n}(\tilde{C})$ such that 
$p = c^{*}c$ is in $M_{n}(\ta)$,
then there is  a partial isometry $a$ in $M_{2n}(\tA)$
 with $a^{*}a = p$ and  $ac^{*}$ is in $\mathcal{M}(\tilde{C'})$.
 Hence,  $(p, q, c)$ is isomorphic to
 $(p, aa^{*}, a)$ via $p, ac^{*}$ in $\Gamma(C';C)$.
 In particular, if $(1_{m}, q, c)$ is in standard form, 
 then it is isomorphic
 to $(1_{m}, aa^{*}, a)$ which is also
  in standard form, via $1_{m}, ac^{*}$.
\end{lemma}

\begin{proof}
As $\rho(C') = D = \rho(C)$, we may find $c'$ in $\mathcal{M}(\tilde{C'})$
such that $\rho(c') = \rho(c)$ and 
$\Vert c' \Vert \leq \Vert c \Vert \leq 1$.
Let 
\[
c'' = 
\left[ \begin{array}{cc} qc' & (q - qc'c'^{*}q)^{1/2} \end{array} \right].
\]
It is immediate that $c''$ is in $\mathcal{M}(\tilde{C'})$, 
$c''c''^{*} = q$ and $\rho(c'') = \rho(c)$. Hence, $a = c''^{*}c$
is a partial isometry and 
\[
\rho(a) = \rho(c'')^{*}\rho(c) = 
\rho(c)^{*} \rho(c) = \rho(c^{*}c) = \rho(p) =\dot{p},
\]
as $p$ is in $\mathcal{M}(\ta)$. This means that
$a$ is in $\mathcal{M}(\tilde{A})$. In addition, we have 
$a^{*}a = c^{*}c''c''^{*}c= c^{*}c=p$.
This completes the proof.
\end{proof}

    \begin{thm}
    \label{2:70}
Let
\vspace{.5cm}

\hspace{3cm}
\xymatrix{ 0 \ar[r]  &  A \ar[r]^{\pi} & C \ar[r]^{\rho}  &  D  \ar[r] &  0 \\
   0 \ar[r]  & A' \ar@{}[u]|-*[@]{\subseteq} \ar[r]^{\pi}   & C' \ar[r]^{\rho}
   \ar@{}[u]|-*[@]{\subseteq}
     & D \ar[r] \ar@{}[u]|-*[@]{=} & 0 }
   \vspace{.5cm}
  
 \noindent    be a commutative diagram with exact rows. Then the natural map
  \newline 
   $\pi_{*} : K_{0}(A';A) \rightarrow K_{0}(C';C)$ is an isomorphism.
   \end{thm}
 
 \begin{proof}
By 
 Proposition \ref{2:10}, every class in $K_{0}(C';C)$ may 
 be represented by a triple in standard form. By Lemma 
 \ref{2:60}, such a triple is isomorphic to one from 
 $\Gamma(A';A)$ and isomorphism implies 
 equivalence and it follows that  $\pi_{*}$ is surjective.

 To prove that $\pi_{*}$ is injective, we begin with 
 $(1_{m}, q, a)$ in $\Gamma(A';A)$, which we may 
 assume is in standard form, and $[1_{m}, q, a]=0$
 in $\Gamma(C';C)$. It follows from Proposition \ref{2:10}
 that we may find $(1_{n},1_{n},c_{1})$ and  \newline
 $(1_{m+n}, 1_{m+n}, c_{2})$ is $\Gamma(C';C)$ which are 
 also in standard form and such that
 $(1_{m}, q, a) \oplus (1_{n},1_{n},c_{1})$ and $(1_{m+n}, 1_{m+n}, c_{2})$
 are isomorphic via $1_{m+n}, c_{2}(a \oplus a_{1})^{*}$.
 
 We consider the homotopies $c_{1}(t), 0 \leq t \leq 1$
  from $1_{m}$ to $c_{1}$ and 
 $c_{2}(t), 0 \leq t \leq 1$
  from $1_{m}$ to $c_{2}$ as elements
   in $\mathcal{M}((C_{0}(0,1] \otimes C)^{\sim})$.
   We can assume that the classes $(1_{n},1_{n}, c_{1}(t))$ and 
   $(1_{m+n}, 1_{m+n}, c_{2}(t))$ are in standard form. Hence, we can 
   apply Lemma \ref{2:60} to the diagram after tensoring with $C_{0}(0,1]$
   to obtain $a_{1}(t)$ and $a_{2}(t)$. Putting everything together, 
   $(1_{m+n}, q\oplus 1_{n}, a \oplus a_{1}(t))$ 
   is isomorphic to $(1_{m+n}, 1_{m+n}, a_{2}(t))$ via 
  $ a_{2}(t)(a \oplus a_{1}(t))^{*}$.  It follows that 
   $(1_{n}, 1_{n},  a_{1}(1))$ and $(1_{m+n}, 1_{m+n}, a_{2}(1))$
   are elementary in $\Gamma(A';A)$  and 
  $(1_{m+n}, q\oplus 1_{n}, a \oplus a_{1}(1))$ 
   is isomorphic to $(1_{m+n}, 1_{m+n}, a_{2}(1))$ 
     Thus we conclude that 
   $[1_{m}, q, a] = 0$ in $K_{0}(A';A)$.    
  \end{proof}     
  
  We remark that this is actually a special case of a more general 
  result that, in the case of a commutative diagram with 
  exact rows,
  \vspace{1cm}
  
  \hspace{3cm}
\xymatrix{ 0 \ar[r]  &  A \ar[r]^{\pi} & C \ar[r]^{\rho}  &  D  \ar[r] &  0 \\
   0 \ar[r]  & A' \ar@{}[u]|-*[@]{\subseteq} \ar[r]^{\pi}   & C' \ar[r]^{\rho}
   \ar@{}[u]|-*[@]{\subseteq}
     & D' \ar[r] \ar@{}[u]|-*[@]{\subseteq} & 0 }
   \vspace{.5cm}
   
   \noindent there is an associated six-term exact
    sequence of relative groups. We refer the reader to \cite{Ha:relK} for
    details, but we included a proof here of the special case
    for completeness.
    
    \begin{rmk}
    \label{2:150}
    Let us make a few concluding remarks on the subject of relative K-theory.
    First of all, there is also a relative group $K_{1}(A';A)$. Without going into
    many details, one considers triples $(p,a,g)$, where $p$ is a 
    projection in $\mathcal{\ta}$, $a$ is 
    an invertible in $p \mathcal{M}(\ta) p$ and $g(t), 0 \leq t \leq 1,$
    is a continuous path of invertibles in 
    $p \mathcal{M}(\tA)p$ with $b(0) = p, b(1) =a$, 
    We refer
    the reader to \cite{Ha:relK} for more details. The only 
    important result for us here is that
    \[
    K_{1}(A';A) \cong K_{0}(C_{0}(\R) \otimes A'; C_{0}(\R) \otimes A),
    \]
    in a natural way so that all of our results here for the relative 
    $K_{0}$-group pass to the relative 
    $K_{1}$-group as well. Of course, the 
    exact sequence of Theorem \ref{2:5} actually becomes a six-term exact
    sequence.
    
    Furthermore, the setting of relative K-theory in \cite{Ha:relK}
    considers a \newline
    $*$-homomorphism $\varphi: A' \rightarrow A$ and defines
      a relative group for the map, $K_{0}(\varphi)$. Our situation 
      is simply the special case of the inclusion map.
    \end{rmk}

\section{Excision: the main result}
\label{3}

We begin with the following result, which
 is a minor variant of a well-known fact. The important part of 
 the set-up is that we do not suppose that either $C^{*}$-algebra is acting
 non-degenerately.

\begin{prop}
\label{3:5}
Let $\mathcal{H}$ be a Hilbert space.
If $A,  B \subseteq \mathcal{B}(\mathcal{H})$ are
 $C^{*}$-algebras and $A B \subseteq A$, then $A + B$ is a $C^{*}$-algebra, 
$A$ is a closed two-sided ideal in $A +B$ and the quotient 
is isomorphic to $B/A \cap B$.
\end{prop}

\begin{proof}
 It is clear that $A + B$ is a $*$-subalgebra  of 
 $\mathcal{B}(\mathcal{H})$
 and that $A$ is 
 a closed, two-sided ideal. We will show that $A + B$ is closed and the 
 proof will be complete.

Let $ \mathcal{N} $ be the closure of $A\mathcal{H}$. The condition that
$AB \subseteq A$ means that $\mathcal{N}$ is invariant under $B$. 
Hence, writing $\mathcal{H} = \mathcal{N} \oplus \mathcal{N}^{\perp}$,
each element $b$ of $B$ may be written as 
$b|_{\mathcal{N}} \oplus b|_{\mathcal{N}^{\perp}}$. As $A$ acts 
non-degenerately on $\mathcal{N}$, $b|_{\mathcal{N}} $ lies in $M(A)$, 
the multiplier algebra of $A$, for every $b$ in $B$.

Let $q: M(A) \rightarrow M(A)/A$ and 
$ \theta: B \rightarrow   M(A)/A \oplus \mathcal{B}(\mathcal{N}^{\perp})$
 be the  map sending $b$ to 
 $q(b|_{\mathcal{N}}) \oplus  b|_{\mathcal{N}^{\perp}} $.
Since this is a $*$-homomorphism, it is continuous and 
$\theta(B)$ is closed. 
We note that $A + B = \theta^{-1}( \theta(B))$ is closed also.
\end{proof}

We will assume throughout this section 
that $A$ and $B$ are related as above and we 
let $i:A \rightarrow A + B$ and $j: B \rightarrow A+B$ denote the 
two inclusion maps.

\begin{thm}
\label{3:10}
Let $A,  B \subseteq \mathcal{B}(\mathcal{H})$ be
 $C^{*}$-algebras with $A B \subseteq A$.

 Let $E$ be a Banach
 $A + B$-bimodule which is 
also a $C^{*}$-algebra and let $\delta: A + B \rightarrow E$
be a  bounded
 $*$-derivation satisfying
  $ \delta(B) \subseteq \delta(A)$.
  
  Then 
 
 \xymatrix{ 0 \ar[r]  &  A \ar[r]^{i} & A+B \ar[r]  &  (A+ B) / A  \ar[r] &  0 \\
   0 \ar[r]  &  \ker(\delta) \cap A \ar@{}[u]|-*[@]{\subseteq} \ar[r]^{i}   & 
   \ker(\delta) \ar[r]
   \ar@{}[u]|-*[@]{\subseteq}
     & (A+ B)/A \ar[r] \ar@{}[u]|-*[@]{=} & 0 }
     
 \noindent     is a commutative diagram with exact rows.      
In consequence, 
\[
i_{*}:K( \ker(\delta) \cap A; A) \rightarrow K(\ker(\delta); A+B)
\]
 is an isomorphism and 
\[
\alpha = (i_{*})^{-1} \circ j_{*}: K(\ker(\delta) \cap B; B) \rightarrow K(\ker(\delta) \cap A; A).
\]
is a homomorphism.
\end{thm}

\begin{proof}
Exactness of the top row follows from Proposition \ref{3:5} and 
commutativity is obvious. Exactness of the bottom row is 
easy except in showing that the map from $\ker(\delta)$ is surjective. 
It is clear that any element of $(A+B)/A$ can 
be represented as $b + A$, with $b$ in $B$.
By hypothesis, there exists $a$ in $A$ with $\delta(a) = \delta(b)$. 
It follows that $b-a$ is in $\ker(\delta)$ and its image 
in $(A+B)/A$ is simply $b + A $.

The rest follows from Theorem \ref{2:100}.
\end{proof}

Let us consider a special case of interest, just
 to see that the hypotheses are quite general and that
the map $\alpha$
is non-trivial. 
Let $B$ be any $C^{*}$-algebra and suppose that 
$(\pi, \mathcal{H}, F)$
is a Fredholm module for $B$: that is, $\mathcal{H}$ is a separable
Hilbert space, 
$\pi: B \rightarrow \mathcal{B}(\mathcal{H})$ is a representation of
$B$ on $\mathcal{H}$ and $F$ is a bounded linear operator 
on $\mathcal{H}$ such that 
\[
\pi(b)(F - F^{*}), \pi(b)(F^{2}-1), [\pi(b), F] = \pi(b)F - F \pi(b),
\]
are all compact, for any $b$ in $B$. 
We assume the slightly stronger conditions
that $F=F^{*}, F^{2}=1$ and, for simplicity, that
$B$ and $\pi$ are unital.
 Let $P = \frac{1}{2}(F+1)$, which is a 
self-adjoint projection.

Such a Fredholm module induces a natural homomorphism
from $K_{1}(B)$ to the integers, which sends $[u]_{1}$, where
 $u$ is 
 a unitary in
  $M_{n}(B)$, to
\[
Ind((1_{n} \otimes P) (id_{n} \otimes \pi)(u) 
\vert_{\C^{n} \otimes P\mathcal{H}}),
\]
where $Ind$  denotes Fredholm index. We
 denote this map by $Ind$ for simplicity. 

We can apply our Theorem \ref{3:10} in this situation
as follows.
Let $A= \mathcal{K}(\mathcal{H})$,
 use $\pi(B)$ as the $C^{*}$-algebra $B$ in \ref{3:10},
 $ E =  \mathcal{B}(\mathcal{H})$, 
which is obviously an $A + \pi(B)$ bimodule 
and define
\[
\delta(x) = i[x, F], x \in A + \pi(B).
\]
It is a simple matter to see that  
\[
\ker(\delta) \cap A = \mathcal{K}(P\mathcal{H}) \oplus 
\mathcal{K}((I-P)\mathcal{H})
\]
and  the relative group
$K_{0}(\ker(\delta) \cap A   ;A) \cong \Z$ has already 
been computed in
Example \ref{2:100}.

\begin{thm}
\label{3:30}
Let  $B$ be a unital 
$C^{*}$-algebra and $(\pi, \mathcal{H}, F)$ be a 
Fredholm module for $B$ with $\mathcal{H}$ separable,
 $F=F^{*}, F^{2}=I, F \neq \pm I$ and  $\pi$ unital.
With $E, \delta, A$ as above, the following 
diagram is commutative 
\vspace{.5cm}
\hspace{2cm}

\xymatrix{
K_{1}(B) \ar[d]^{\mu \circ \pi} \ar[r]_{Ind} & \Z  
\\
 K_{0}(\ker(\delta) \cap \pi(B); \pi(B)) \ar[r]^{\alpha} & 
 K_{0}(\ker(\delta) \cap A; A) \ar[u]^{\cong} }
 
 \vspace{1cm}

\noindent where $\mu$ is as in Theorem \ref{2:5},
  $\alpha$ is as in Theorem \ref{3:10}
  and
 the vertical arrow on the right is the isomorphism
 described in Example \ref{2:100}.
\end{thm}

\begin{proof}
For simplicity, we consider a unitary $u$ in $B$, instead of 
$M_{n}(B)$, for some $n$. Let $P_{1}, P_{2}, P_{3}, P_{4}$ 
be the orthogonal projections of 
$\mathcal{H}$ onto each of the following four subspaces:
\[
\begin{array}{cc}P\mathcal{H} \cap \pi(u)^{*}P\mathcal{H}, &
P\mathcal{H} \cap \pi(u)^{*}(I-P)\mathcal{H}, \\
(I-P)\mathcal{H} \cap \pi(u)^{*}(I-P)\mathcal{H},  & 
(I-P)\mathcal{H} \cap \pi(u)^{*}P\mathcal{H}. \end{array}
\]
As $\pi(u)$ is unitary, $P_{1} + P_{2} +P_{3} +P_{4} =I$.
The projections $P_{2}$ and $P_{4}$ are both finite rank
from the Fredholm module condition and the index of 
$P \pi(u) \vert_{P\mathcal{H}}$ is simply
$\dim(P_{2}\mathcal{H}) - \dim(P_{4}\mathcal{H})$.
Also let $Q_{i} = \pi(u) P_{i} \pi(u)^{*}, 1 \leq i \leq 4$.

Letting $j_{*}$ be as in Theorem \ref{3:10}, 
$j_{*} \circ \mu \circ \pi[u]_{1}$ is represented by the class of 
$(I, I, \pi(u))$ in $\Gamma(\ker(\delta), \pi(B) +A)$. It 
follows  that 
\begin{eqnarray*}
[I, I, \pi(u)] & = & [ P_{1} + P_{3}, Q_{1} + Q_{3}, \pi(u)(P_{1} + P_{3})] \\
   &   &   +  [ P_{2} + P_{4}, Q_{2} + Q_{4}, \pi(u)(P_{2} + P_{4})].
   \end{eqnarray*}
On the other hand, $ \pi(u)(P_{1} + P_{3})$ commutes with $P$ and, hence,
 is in  $\ker(\delta)$, so 
 $[ P_{1} + P_{3}, Q_{1} + Q_{3}, \pi(b)(P_{1} + P_{3})] = 0$.
 We have 
 \[
 \alpha \circ \mu \circ \pi([u]_{1}) = 
 i_{*}^{-1} \circ j \circ \mu \circ \pi([u]_{1})
   =  [ P_{2} + P_{4}, Q_{2} + Q_{4}, \pi(u)(P_{2} + P_{4})]
   \]
   and  the isomorphism of Example \ref{2:100} (using
   $\mathcal{H}^{+} = P\mathcal{H}$) maps this to
   \[
   rank(\pi(u)P_{2}) - rank(\pi(u)P_{4}) = 
   \dim(P_{2}\mathcal{H}) - \dim(P_{4}\mathcal{H}).
   \]
\end{proof}

We will continue to assume throughout that $A, B, \delta, E$ are 
as in Theorem \ref{3:10}. Our main goal is 
to provide extra conditions under which the  map $\alpha$ is 
actually an isomorphism.

\begin{thm}
\label{3:70}
Let $A, B, E, \delta$ be as in \ref{3:10}. Suppose that $A$ has 
a dense $*$-subalgebra, $\mathcal{A}$ satisfying the following.
  \begin{enumerate}
  \item [C1] There is a constant $K \geq 0$ such that, for every 
  $a$ in $\mathcal{A}$, there is $a' $ in $\ker( \delta) \cap \mathcal{A}$
   such that 
   \[
   \Vert a - a' \Vert \leq K \Vert \delta(a) \Vert.
   \]
   \item [C2] For every  $a_{1}, \ldots, a_{I}$
    in $\mathcal{A}$, there is $0 \leq e \leq 1$ in $M(A)$, 
    the multiplier algebra of $A$ (\cite{Ped:book}) and 
   $b_{1}, \ldots, b_{I}$ in $B$ such that
    such that 
    \begin{enumerate}
    \item $a_{i} = ea_{i} =a_{i}e = eb_{i} = b_{i}e$, 
      and 
    \item $\delta(b_{i}) = \delta(a_{i})$,
    \end{enumerate}
    for all $ 1 \leq i \leq I$.
 \end{enumerate}
  Then $\delta(A) = \delta(B)$ is closed and 
 the map $\alpha$ of Theorem \ref{3:10} is an isomorphism.
\end{thm}

All of our later applications will be to groupoid 
$C^{*}$-algebras. In such cases,  the dense
$*$-subalgebras of continuous compactly-supported
functions on the groupoid will form the
 $*$-subalgebra, $\mathcal{A}$.

 Let us remark that it seems  reasonable  to conjecture
 that one could replace conditions C1 and C2 of Theorem \ref{3:70}
 with the hypothesis that
  $\delta(A) = \delta(B)$ is closed. It certainly makes for a cleaner
 result. On the other hand, if one wants to verify
 this condition 
 in the case of groupoid $C^{*}$-algebras, then employing C1 and C2
 to do this (and using the dense subalgebras of continuous compactly-supported
 functions) is not an unreasonable route.

 We comment that the relation between the statement
  of Theorem \ref{3:70} and its applications  
 is rather similar to that of
 the first isomorphism theorem for groups. That statement is, given 
 a surjective group homomorphism, $\alpha: G \rightarrow H$, the map induces
 an isomorphism between the quotient group $G/ker(\alpha)$ and $H$. 
 In most applications, however, one is not given 
 $\alpha, G$ and $H$, but rather $G$
 and a normal subgroup $N$ and then tries to cook up an $\alpha$ and $H$ 
 such that $N = \ker(\alpha)$, so that one may identify $G/N$ with the 
 (hopefully) more familiar group, $H$.
 Our applications usually involve starting 
 with $B' \subseteq B$ and trying to cook up
 $\delta, A$ and $E$ so as to apply Theorem \ref{3:70} and with
  $B' = B \cap \ker(\delta)$.
  The hope, in this case,
 is that $A \cap \ker(\delta) \subseteq A$ is simpler than 
$B' \subseteq B$.

\section{The proof}
\label{4}

 This section is devoted to a proof of Theorem \ref{3:70}.

We start with expanding on condition C1. 
We remark that, informally, this
 means that \emph{almost} being in $\ker(T)$
  ($Tx$ is small)
implies \emph{nearly} being in  $\ker(T)$ (near an element in 
 $\ker(T)$). The following result is quite standard but we provide a proof 
 for completeness.

\begin{prop}
\label{4:10}
Let $T: X \rightarrow Y$ be a bounded linear map between two Banach spaces.
The following are equivalent.
\begin{enumerate}
\item $T(X)$ is  closed in $Y$.
\item There is a constant $K \geq 1$ such that, for every 
$x$ in $X$, there is $x'$ in $\ker(T)$ such that 
\[
\Vert x - x' \Vert \leq K \Vert Tx \Vert.
\]
\item There is a dense linear subspace  $\mathcal{X} \subseteq X$
and a constant $K \geq 1$ such that, for every 
$x$ in $\mathcal{X}$, there is $x'$ in $\ker(T)$ such that 
\[
\Vert x - x' \Vert \leq K \Vert Tx \Vert.
\]
\end{enumerate}
\end{prop}

\begin{proof}
We give a short sketch of the proof. Assume that $T(X)$ is closed. As 
$T$ is bounded, $\ker(T)$ is a closed subspace and $T$ induces a
bounded linear bijection between $X/ \ker(T)$ and $T(X)$.
 If $T(X)$ is closed, then this map has a continuous inverse (Corollary 
 2.2.5 of \cite{Ped:book}), denoted $S$. If  $x$ is in $X$, we may find
 $x'$ in $\ker(T)$ such that
  \[
  \Vert x - x' \Vert \leq 2 \Vert x + \ker(T) \Vert \leq 
  2 \Vert S \Vert \Vert Tx \Vert.
  \]
  
 The second condition obviously implies the third. Now suppose
  the third is satisfied. Let $x_{n}$ be any sequence in $X$ such that 
 $Tx_{n}$ converges to $y$ in $Y$. As $\mathcal{X}$ is dense in $X$, we 
 can assume without loss of generality that the sequence is actually in 
 $\mathcal{X}$. By passing to a subsequence, we may also assume that 
$ \Vert Tx_{n} - y \Vert < 2^{-n}$. For each $n > 1$, we may find $x'_{n}$
in $\ker(T)$ such that
\[
\Vert x_{n}-x_{n-1} -x'_{n} \Vert \leq \Vert T(x_{n}-x_{n-1}) \Vert
\leq \Vert Tx_{n} - y \Vert + \Vert y - Tx_{n-1} \Vert \leq 2^{2-n}.
\]
It follows easily then that the sequence 
$z_{n} = x_{n} - \sum_{i=1}^{n}x_{n}'$ is Cauchy and hence convergent 
to some $x$  in $X$. Moreover, we have $Tz_{n}= Tx_{n}$ and by 
continuity of $T$, $Tx=y$.
\end{proof}

Observe that our condition C1 is stronger than the third condition above
because C1 requires $a'$ to be in $\ker(\delta) \cap \mathcal{A}$, and 
not simply in $\ker( \delta)$.

An immediate nice consequence is the following.

  \begin{lemma}
  \label{4:15}
  If $\mathcal{A}$ is a dense subalgebra of $A$ satisfying  
   condition C1 of Theorem \ref{3:70} and there is a 
   dense $*$-subalgebra, $\mathcal{B} $ in $B$ 
   such that  
  $\delta(\mathcal{B}) \subseteq \delta(\mathcal{A})$, then 
  we have $\delta(B) \subseteq \delta(A) $.
  \end{lemma}

  \begin{proof}
  It follows from C1 and Proposition \ref{4:10} that the extension
  of $\delta$ to $A$ has closed range. As $\delta$ is bounded and 
 $\mathcal{B}$ is dense in $B$, we are done. 
  \end{proof}

 As we will need to deal with elements of the relative $K$-groups, 
 it will be useful to  have the following, which considers the unitization
 of algebras, matrices over algebras and the cones over algebras.

 \begin{prop}
 \label{4:60}
Suppose that $ A,  B, E, \delta$
  be as in Theorem \ref{3:10} and $\mathcal{A}$
  satisfy conditions C1 and C2 of Theorem \ref{3:70}.
 \begin{enumerate}
 \item
 If we define $\delta: \widetilde{A+B} \rightarrow E$
 by 
 \[
 \delta(\lambda + a + b) = \delta(a+b), a \in A, b \in B, \lambda \in \C,
 \]
 then  $ \delta$ is a bounded $*$-derivation and 
 $\delta(\tilde{B}) \subseteq \delta(\tilde{A})$. 
 \item 
 For any integer $n \geq 1$,  $ M_{n}(A),  M_{n}(B), M_{n}( E), id_{M_{n}} \otimes \delta$
  satisfy the conditions of
  Theorem \ref{3:10} and $M_{n}(\mathcal{A})$
  satisfy conditions C1 (although $K$ may depend on $n$)
  and C2 of Theorem \ref{3:70}. Moreover, the $e$
  of condition C2 
  may be chosen to be of the form $1_{n} \otimes e$, where $e$ is in 
  $M(A)$.
  \item 
    $C_{0}(0,1] \otimes A,  C_{0}(0,1] \otimes B, C_{0}(0,1] \otimes E, id_{C_{0}(0,1]} \otimes \delta$
  satisfy the conditions of
  Theorem \ref{3:10}. In addition, let $C_{c}(0,1] \odot \mathcal{A}$ 
  denote the algebraic tensor product of the continuous, compactly supported functions
  on $(0,1]$ with $\mathcal{A}$. That is, it consists of all functions of the form
  \[
  f(t) = \sum_{i=1}^{I} f_{i}(t) a_{i}, 
  \]
  where $f_{1}. \ldots, f_{I}$ are in $C_{c}(0,1]$ and 
  $a_{1}, \ldots, a_{I}$ are in $\mathcal{A}$.
   Then $C_{c}(0,1] \odot  \mathcal{A}$ is a dense 
   $*$-subalgebra of $C_{0}(0,1] \otimes A$ which 
   satisfies conditions C1 and C2 of  Theorem \ref{3:70}.
 \end{enumerate}
 \end{prop}
 
 \begin{proof}
 The first part is trivial and we omit the proof. For
  the second part and third parts, the proofs of C1 are both trivial.
  
  Let us sketch the proof of C2 for the second part. If we are given
  a collection of $m$ elements of $M_{n}(\mathcal{A})$, their 
  individual entries provide $n^{2}m$ elements of $\mathcal{A}$. If we select
  an appropriate $e$ and $n^{2}m$ elements of $B$, it is fairly easy to check
  that $1_{n} \otimes e$ and the corresponding $m$ elements of 
  $M_{n}(B)$ satisfy the desired conclusion.
  
 We now consider the third part, beginning with C1. Let 
 \[
  f(t) = \sum_{i=1}^{I} f_{i}(t) a_{i}, 
  \]
  be in $C_{c}(0,1] \odot \mathcal{A}$. If $(1 \otimes \delta)(f) = 0$,
   we are done.
  Otherwise, let $0 < \epsilon = \Vert  (1 \otimes \delta)(f) \Vert$.
   We may choose
  $N$ sufficiently large so that  $\Vert f(s) - f(t) \Vert < \epsilon$, whenever
  $\vert s- t \vert < N^{-1}$. Also choose $N$ sufficiently large
  so that $f(t) = 0$, for all $0 \leq t \leq N^{-1}$. For each $1 \leq n \leq N$, 
  let $g_{n}$ be the function which is $0$ on $[0, (n-1)/N] \cup [(n+1)N, 1]$,
  $1$ at $n/N$, and linear on $((n-1)/N, n/N)$ and $(n/N, (n+1)/N)$.
  Let \[
  g(t) = \sum_{n=1}^{N} g_{n}(t) f(n/N), t \in [0,1].
  \]
 It is an easy computation to see that
 $\Vert g - f \Vert < \epsilon$.  For $2 \leq n \leq N$, 
 we may find $a_{n}'$ in 
 $\mathcal{A} \cap \ker(\delta)$ with 
 \[
 \Vert a_{n}' - f(n/N) \Vert \leq K \Vert \delta(f(n/N) ) \Vert \leq K \Vert 1 \otimes \delta(f) \Vert , 1 \leq n \leq N.
 \]
 We note that $f(1/N) = 0$ and so set $a_{1}'=0$.
 We then define $f'(t) = \sum_{n=1}^{N} g_{n}(t) a_{n}'$. It is clear that
 $f'$ is in $\ker( 1 \otimes \delta)$ and  also
 \begin{eqnarray*}
 \Vert f - f' \Vert & \leq & \Vert f - g \Vert + \Vert g - f' \Vert \\
    &  \leq &  \Vert  (1 \otimes \delta)(f) \Vert + \sum_{n=1}^{N} g_{n}(t) \Vert f(n/N) - a_{n}' \Vert \\
     &  \leq &  \Vert  (1 \otimes \delta)(f) \Vert + \sum_{n=1}^{N} g_{n}(t)K  \Vert 1 \otimes \delta (f) \Vert \\
     & \leq & (K+1) \Vert  (1 \otimes \delta)(f) \Vert.
     \end{eqnarray*}
 This completes the proof of C1. 
 
 For C2, if we have a finite collection of functions, $f^{(1)}, \ldots, f^{(I)}$ 
 in $C_{c}(0,1] \odot \mathcal{A}$, we may 
 write all of them in the form \newline
  $f^{(i)}(t) = \sum_{j=1}^{J} f^{(i)}_{j}(t) a_{j}$,
 for all $1 \leq i \leq I$, simply by using all the possible elements of $\mathcal{A}$
 involved in each and using $f^{(i)}_{j}=0$ as needed.
 
 We then choose $b_{j}, 1 \leq j \leq J, e$ as in C2 for $\mathcal{A}$.
 As each $f_{j}^{(i)}$ is compactly supported, we may find
 $0 < \epsilon$ such that each is zero on $[0, \epsilon]$. Let 
 $g(t)$ be the continuous function that is zero on $[0, \epsilon/2]$, $1$ on 
 $[\epsilon, 1]$ linear on $[\epsilon/2, \epsilon]$. It is now easy to that the 
 set of elements $\sum_{j=1}^{J} f_{j}^{(i)} b_{j}$ and $e(t) = g(t)e$ 
 satisfies the desired conclusion.
 \end{proof}

 \begin{lemma}
 \label{4:20}
Suppose $a, b, e$ are elements of a $C^{*}$-algebra satisfying
$e^{*}=e$,
\[
a= ae =ea = be = eb
\]
and $\delta$ is a $*$-derivation on that $C^{*}$-algebra with 
$\delta(a) = \delta(b)$.
 It follows that 
 \begin{enumerate}
 \item $ab= a^{2} =ba, ab^{*} = aa^{*}, b^{*}a= a^{*}a$,
 \item 
\begin{eqnarray*}
 \delta(a) b & =  &  \delta(a) a,  \\
 \delta(a) b^{*} & =  &  \delta(a) a^{*},  \\
 b \delta(a) & =  &  a\delta(a), \\
 b^{*} \delta(a) & =  &  a^{*} \delta(a),
 \end{eqnarray*}
\item 
 for any continuous function $f$ on the positive
  real numbers and complex number $\lambda$, we have
  \begin{eqnarray*}
  a f((\lambda +b)^{*}(\lambda+b)) & = & af((\lambda +a)^{*}(\lambda +a)), \\
   f((\lambda +b)^{*}(\lambda+b))  a  & = &  
   f((\lambda  +a)^{*}(\lambda +a))a, \\
   b f(a^{*}a) & = & a f(a^{*}a), \\
      f(a^{*}a) b & = &  f(a^{*}a) a, \\
  \delta(a)  f((\lambda +b)^{*}(\lambda+b))  & =  
   & \delta(a)f((\lambda +a)^{*}(\lambda +a)) ,  \\
   f((\lambda +b)^{*}(\lambda+b))  \delta(a)& = &  
    f((\lambda +a)^{*}(\lambda +a)) \delta(a),
  \end{eqnarray*}
\item 
 for any continuous function $f$ on the positive
  real numbers and complex numbers $\lambda, \mu$, we have 
  \begin{eqnarray*}
  \delta((\lambda+a) f((\mu +a)^{*}(\mu +a)) )
 &  =  & \delta( (\lambda + a) ( f((\mu +b)^{*}(\mu+b)) ) \\
    &  =  & \delta( (\lambda + b) ( f((\mu +b)^{*}(\mu+b)) ). 
    \end{eqnarray*}
    The same result holds replacing 
   $ (\mu +b)^{*}(\mu+b)$ by $(\mu +b)(\mu+b)^{*}$.
  \end{enumerate}
 \end{lemma}

 \begin{proof}
 For the first part, we compute
 \[
 ab = (ae)b =  a(eb) =a^{2}.
 \]
 The rest is done in a similar way (using $e^{*}=e$).
 
 For the second part, we have 
 \[
  \delta(a)b = \delta(ab) - a \delta(b) 
      = \delta(a^{2}) - a \delta(a)  =  \delta(a)a.
\]     
 The other statements are done in a similar way.
 
 For the third part, observe that if we have any 
 word in $a, a^{*}, b$ and $b^{*}$, the presence of a 
 single $a$ or $a^{*}$ means that the element is the same 
 if we replace all of the $b, b^{*}$ terms with $a, a^{*}$, from
 repeated application of the first part.
  The first two conditions then follow for any polynomial $f$ and 
 subsequently, by realizing any continuous function 
 as a limit of polynomials.
 
 The second part implies that this same observation holds if we
 include terms $\delta(a)$ or $\delta(a)^{*}$ in our words: the presence 
 of a single $a, a^{*}, \delta(a)$ or $\delta(a)^{*}$ means we
 can replace all $b$'s with $a$'s.
 
For the last part, we first consider the case that $f$ is a polynomial.
We expand $\delta(f((\lambda+b)^{*}(\lambda+b)))$ 
using the Leibnitz rule. Each 
term is a word containing $\delta(b)$ or $\delta(b)^{*}$ and 
we use our hypothesis
to replace $\delta(b)$ with $\delta(a)$. The remainder
 of the proof follows 
exactly as in the second part.
  \end{proof}

 \begin{prop}
 \label{4:30}
 If $A, B, \delta, E$ satisfy conditions C1 and C2 
 of Theorem \ref{3:70}, then $\delta(B) = \delta(A)$ is closed.
 \end{prop}
 
 \begin{proof}
 First, we claim  that if $a$ is any element of $\mathcal{A}$, then there is
 $b$ in $B$ with $\delta(b) = \delta(a)$ and
  $\Vert b \Vert \leq \Vert a \Vert$. To see this, first select $b$ 
  in $B$ as in condition C2. Let $f: [0, \infty) \rightarrow \R$
   be the continuous
  function
  \[
  f(t) = \left\{ \begin{array}{cl}
       1 & 0 \leq t \leq \Vert a \Vert \\
       t^{-1/2} \Vert a \Vert^{-1/2} & \Vert a \Vert < t. \end{array} \right.
       \]
       Then 
 \[
 \Vert b f(b^{*}b) \Vert \leq  \Vert f(b^{*}b)b^{*} b f(b^{*}b) \Vert^{1/2}
    \leq \Vert t f(t)^{2} \Vert_{\infty} \leq \Vert a \Vert.
    \]
    and 
    \[
    \delta(b f(b^{*}b)) = \delta(a f(a^{*}a)) = \delta(a),
    \]
    by the last part of Lemma \ref{4:20}.
    
    We know that $\delta(B) \subseteq \delta(A)$. To prove the 
    reverse inclusion, let $a$ be in $A$. With $a_{0} =0$, 
    inductively choose $a_{n}, n \geq 1,$ in $\mathcal{A}$ such that 
    \[
    \Vert (a - a_{0} - \cdots -a_{n-1}) - a_{n} \Vert 
    \leq 2^{-n}\Vert a \Vert
    \]
    and 
    \[
    \Vert a_{n} \Vert \leq \Vert a - a_{0} - \cdots -a_{n-1} \Vert. 
    \]
    For each $n$, we may find $b_{n}$ in $B$ with
     $\delta(b_{n}) = \delta(a_{n})$ and 
     $\Vert b_{n} \Vert \leq \Vert a_{n} \Vert$. The series
     $b = \sum_{n} b_{n}$ is then convergent in $B$ and 
     $\delta(b) = \delta(a)$ by continuity of $\delta$.
 \end{proof}
 
 As a final preliminary step, we observe the following important consequence
 of the fact that $\delta(A)$ is closed.

\begin{lemma}
\label{4:80} Suppose that $\delta: A \rightarrow E$
 is a bounded $*$-derivation with closed range. 
For each  $n \geq 1$, there is $\epsilon_{n} > 0$ such that
if $a$ is a partial isometry in $M_{n}(\tilde{A})$, $a^{*}a=p, 
aa^{*} =q$ with 
$\delta(p) = \delta(q) =0$ and 
$\Vert \delta(a) \Vert < \epsilon_{n}$, then 
there are partial isometries $a_{1}, a_{2}$ in $M_{n}(\tA)$ with
 $ a_{1}^{*} a_{1}= a_{2} a_{2}^{*}= a_{2}^{*}a_{2}=p$, $a= a_{1}a_{2}$, 
 $\delta(a_{2}) = 0$ and $(q,q, a_{1})$ is elementary.
\end{lemma}

\begin{proof}
A standard argument using functional calculus proves that there
is $\epsilon_{0} >0$ such that if $c$ is a partial isometry
any $C^{*}$ algebra $C$ and $d$ is any other element of $C$ with 
$\Vert c -d \Vert < \epsilon_{0}$, then 
$dd^{*}$ is invertible in $cc^{*}C cc^{*}$ and 
\[
\Vert  d (d^{*}d)^{-1/2}- c \Vert < 1.
\]

Let $K$ be the constant associated with
 $\delta:M_{n}(\tilde{A}) \rightarrow M_{n}(E)$ in 
 condition C1 and 
 set $\epsilon_{n} = K^{-1} \epsilon_{0}$, as above.
 
 Given $a, p, q$ as in the statement, we may find $a_{3}$ in 
 $qM_{n}(\tilde{A})p$ with $\delta(a_{3}) = 0$ and 
 $\Vert a_{3} - a \Vert < K \epsilon_{n}$. Let 
 $a_{2} = a_{3}(a_{3}^{*}a_{3})^{-1/2} $ and 
 $a_{1} = aa_{2}^{*}$. All of the conclusions are immediate except that
 $(q, q, a_{1})$ is  elementary. To see this, we first observe that 
\[
\Vert a_{1} - q \Vert = \Vert aa_{2}^{*} - aa^{*} \Vert 
\leq \Vert a_{2} - a \Vert < 1.
\]
It follows that $a_{1}$ is homotopic to $q$ in the invertibles
of $qM_{n}(\tA)q$.
\end{proof}

\begin{lemma}
\label{4:75}
Assume that $A, B, \delta, E$ satisfy conditions C1 and
 C2 of \ref{3:70}. 
Let $(1_{m}, q, a)$ be in standard form in 
$\Gamma(A \cap \ker(\delta);A)$ and 
assume that $a$ is in $M_{n}(\tA)$.
Given $\epsilon > 0$, there exists a partial isometry 
$b$ in $M_{2n}(\tB)$
with $\delta(b^{*}b) =0, bb^{*}= 1_{m}$
 and $\Vert \delta(ab) \Vert < \epsilon$.
\end{lemma}

 \begin{proof} 
  We let 
 $\Vert \delta \Vert$ be the norm of 
 $\delta: M_{n}(A+B) \rightarrow M_{n}(E)$. We also let 
 $K$ be a constant such that C1 holds, for both
 $\delta: M_{2n}(\tilde{A}) \rightarrow M_{2n}(E)$ 
 and  $\delta: M_{2n}(\tilde{B}) \rightarrow M_{2n}(E)$, making use
 of Proposition \ref{4:60},  Theorem \ref{4:30}
 and  Proposition \ref{4:10}.
 
 Let  $\epsilon' > 0$ be such that 
 \[
 \Vert \delta \Vert \left[ (8K \epsilon')^{1/2} + 4K \epsilon' + 7  \epsilon' \right]
 < \epsilon
 \]
 and  so that $K \epsilon' < 1/2$.
 
We find $a_{1}$ in $M_{n}(\mathcal{A})$ such that
 $\Vert  a - 1_{m} -a_{1} \Vert < \epsilon'$ and 
 $\Vert 1_{m} + a_{1} \Vert \leq \Vert a \Vert =1$.
  It follows that 
 \[
 \Vert (1_{m}+a_{1})^{*}(1_{m}+a_{1}) - 1_{m} \Vert,
   \Vert (1_{m}+a_{1})(1_{m}+a_{1})^{*} - q \Vert
  \leq 2 \epsilon'.
  \]
 If we replace $a_{1} $ by $ a_{1} \cdot 1_{m}$, the same
  estimate still holds and so we may assume that $a_{1} 1_{m}=a_{1}$.
 
We use condition C2 to find $b_{1}$ in $M_{n}(B)$ and $e$ in 
$M_{n}(\mathcal{M}(A))$, commuting with $1_{m}$, with 
\[
a_{1} = a_{1}e =ea_{1} =b_{1}e =eb_{1}, \delta(a_{1}) = \delta(b_{1}).
\]
By replacing $b_{1}$ with $b_{1}  1_{m}$, we may assume 
that $b_{1} 1_{m}=b_{1}$.

Let 
\[
f(t) = \left\{ \begin{array}{cl} 1 & t \leq 1, \\
                      t^{-1/2} & t \geq 1 \end{array} \right.
 \]
 and $b_{2} = (1_{m} + b_{1}) f((1_{m}+b_{1})^{*}(1_{m}+b_{1}))$.
 It is clear that $b_{2}$ is in $M_{n}(\tilde{B})$, $b_{2}1_{m}=b_{2}$ and 
 $\Vert c \Vert \leq 1$. It follows that 
 \[
 b_{3} = \left[ b_{2}^{*} \hspace{1cm} \left( 1_{m} - b_{2}^{*}b_{2} \right)^{1/2} \right]
 \]
 is in $M_{2n}(\tilde{B})$ and satisfies $b_{3}b_{3}^{*}=1_{m}$.
 
  We definethe partial
  isometry
 \[
 ab_{3} = \left[ ab_{2}^{*} \hspace{.5cm}
  a\left( 1_{m} - b_{2}^{*}b_{2} \right)^{1/2} \right].
  \]
  
  First, we note 
  \begin{eqnarray*}
 \Vert  \delta(ab_{3})  \Vert & \leq  &  
 \Vert  \delta((1_{m}+a_{1})b_{3})  \Vert 
   +  \Vert \delta \Vert \Vert a - (1_{m}+a_{1}) \Vert \\
     & \leq   &  \Vert  \delta((1_{m}+a_{1})b_{3})  \Vert 
   +  \Vert \delta \Vert \epsilon'.
   \end{eqnarray*}
   
 To deal with the first of the two terms, 
 we make extensive use of Lemma \ref{4:20}.       
We compute
\begin{eqnarray*}
\Vert \delta( (1_{m}+a_{1}) b_{2}^{*} ) \Vert 
& = & \Vert \delta( 1_{m}+a_{1}) b_{2}^{*}  +(1_{m}+a_{1}) 
\delta(  b_{2} )^{*} \Vert \\
 & = & \Vert\delta( 1_{m}+a_{1}) f((1_{m}+b_{1})^{*}(1_{m}+b_{1}))
  (1_{m}+b_{1})^{*}  \\
   &  &   
 +(1_{m}+a_{1}) \delta(  a_{1} )^{*} \Vert \\
 & = & \Vert \delta( 1_{m}+a_{1}) f((1_{m}+a_{1})^{*}(1_{m}+a_{1}))
  (1_{m}+a_{1})^{*} \\
    &  &  
   +(1_{m}+a_{1}) \delta(  a_{1} )^{*} \Vert \\
  & = & \Vert \delta((1_{m} +a_{1})^{*}(1_{m}+a_{1})) \Vert \\
  & \leq & \Vert \delta \Vert 2 \epsilon'.
\end{eqnarray*}
Finally, letting $h(t) = (1 - tf(t)^{2})^{1/2}$, 
 as $ 0 \leq ( 1_{m} +a_{1}^{*}(1_{m}+a_{1}) \leq 1$, 
\[
\Vert h( ( 1_{m} +a_{1})^{*}(1_{m}+a_{1}) ) \Vert = 
\Vert 1_{m} - ( 1_{m} +a_{1})^{*}(1_{m}+a_{1}) \Vert \leq 2 \epsilon'.\]
We also have 
\begin{eqnarray*}
\delta \left( (1_{m}+a_{1} )\left( 1_{m} - b_{2}^{*}b_{2}  \right)^{1/2} \right)
 & = & 
  \delta  (1_{m}+a_{1})\left( 1_{m} - b_{2}^{*}b_{2}   \right)^{1/2}  \\
  &  &  +
  (1_{m}+a_{1})\delta \left( ( 1_{m} - b_{2}^{*}b_{2} )^{1/2} \right)   \\
  & = & 
  \delta \left( (1_{m}+a_{1}) \right) h((1_{m}+b_{1})^{*}(1_{k}+b_{1}))  \\
    &  &  +
  (1_{m}+a_{1})\delta \left( h(1_{m}+b_{1})^{*}(1_{m}+b_{1}))  \right)   \\
   & = & 
  \delta \left( a_{1} \right) h((1_{m}+a_{1})^{*}(1_{m}+a_{1}))  \\
    &  &  +
  (1_{m}+a_{1})\delta \left( h((1_{m}+a_{1})^{*}(1_{m}+a_{1}))  \right)  
\end{eqnarray*}
 Together, we conclude that 
 \[
 \Vert \delta \left( (1_{m}+a_{1} )
 \left( 1_{m} - b_{2}^{*}b_{2}  \right)^{1/2} \right) \Vert 
 \leq \Vert \delta \Vert 4 \epsilon'.
\]
Putting this together implies that
\[
\Vert \delta(ab_{3}) \Vert \leq 
  \Vert \delta \Vert 2 \epsilon' + 
   \Vert \delta \Vert 4 \epsilon' +  \Vert \delta \Vert \epsilon'
   = 7 \Vert \delta \Vert \epsilon'.
\]

At this point, the only property lacking for $b_{3}$ is that
$\delta(b_{3}^{*}b_{3})$ may not be zero. Let $p' = b_{3}^{*}b_{3}$.
We have 
\[
\Vert \delta(p') \Vert = \Vert \delta((ab_{3})^{*}(ab_{3})) \Vert 
\leq 2 \Vert  \delta(ab_{3})\Vert \leq \epsilon'.
\]
So we may find $b_{4}$ in $M_{2n}(\tilde{B})$ with $\delta(b_{4}) =0$ and
$\Vert b_{4} - p' \Vert \leq K \epsilon'$. We may also assume 
that $b_{4}$ is self-adjoint. By Lemma 2.2.3  of \cite{RLL:book}, 
the spectrum of $b_{4}$ is contained in 
$[-K \epsilon', K \epsilon'] \cup [1-K \epsilon', 1+ K \epsilon'] $.
As $K \epsilon' < 1/2$,
$p = \chi_{(1/2, \infty)}(b_{4})$ is a projection
in $M_{2n}(\tilde{B}) \cap \ker(\delta)$ and 
$\Vert p - p' \Vert \leq 2K \epsilon'$. It follows from routine estimates
that we may find a partial isometry $b_{5}$ in $M_{2n}(\tilde{B})$ such that
$b_{5}b_{5}^{*}=p'$ and $b_{5}^{*}b_{5} = p$ and 
\[
\Vert b_{5} - p \Vert \leq (8K \epsilon')^{1/2} + 4K \epsilon'.
\]
From this, it follows that 
\[
\Vert \delta(b_{5}) \Vert \leq 
\Vert \delta \Vert \left[ (8K \epsilon')^{1/2} + 4K \epsilon' \right].
\]
Letting $b = b_{3} b_{5}$ now satisfies all the desired properties.
\end{proof}

We would now like to establish an analogue
of Lemma \ref{2:60}, but replacing $C$ there with $A$ and $A$ with $B$.
We do this at the small price of an elementary factor.

\begin{lemma}
\label{4:90}
Suppose that $(1_{m}, q, a)$ in 
$\Gamma(A \cap \ker(\delta);A)$ is in standard form
and that $\epsilon > 0$.
Then there exists $(1_{m}, p, b)$ in 
$\Gamma(B \cap \ker(\delta);B)$ also in standard form
and $(q, q, \tilde{a})$ in 
$\Gamma(A \cap \ker(\delta);A)$,  elementary with $a^{*}a=aa^{*}=q$
 such that$(1_{m}, q,\tilde{a} a)$ is isomorphic 
 to $(1_{m}, q', b)$ in $\Gamma(\ker(\delta), A+B)$ via
 $1_{m}, b(\tilde{a}a)^{*}$.
\end{lemma}

\begin{proof}
By Lemma  \ref{4:75}, we may find a partial isometry $b$
in $M_{2n}(\tilde{B})$ with $bb^{*}=1_{m}$, $p=b^{*}b$ in 
$\ker(\delta)$ and $\Vert \delta(ab) \Vert < \epsilon$.

We next appeal to Lemma \ref{2:60} using $c=ab$ to find 
a partial isometry $a_{1}$ in 
$M_{4n}(\tA)$ with $a_{1}^{*}a_{1} = p$, $a_{1} a_{1}^{*}=q'$ in 
$\ker(\delta)$ $(p, q, ab)$ isomorphic to
$(p, q', a_{1})$ via $p, a_{1}b^{*}a^{*}$. 
The important part of this is that $ a_{1}b^{*}a^{*}$ is in $\ker(\delta)$.
Hence, we have 
\[
\Vert \delta(a_{1}) \Vert = \Vert \delta(a_{1}(ab)^{*}ab) \Vert 
 = \Vert a_{1}(ab)^{*}\delta(ab) \Vert  < \epsilon_{4n}.
 \]
 
 It now follows from Lemma  \ref{4:80}
 that $a_{1} = a_{2} a_{3}$, with 
 $a_{2}^{*}a_{2}  = a_{3}a_{3}^{*} = 
 a_{3}^{*}a_{3}=p, q_{1} = a_{2} a_{2}^{*}$, $\delta(a_{2})=0$
 and $(p, p, a_{3})$ elementary.
 It follows that
 \[
 a =  aba_{1}^{*}a_{1} b^{*} = (aba_{1}^{*})a_{2} a_{3}b^{*}
 \]
 and hence
 \[
  (aba_{1}^{*}a_{2}) a_{3}^{*} 
 (aba_{1}^{*}a_{2})^{*} a = (aba_{1}^{*}a_{2})b^{*}.
 \]
 Observe that 
 $\tilde{a} = (aba_{1}^{*}a_{2}) a_{3}^{*} (aba_{1}^{*}a_{2})^{*}$
 is also elementary and
 $(1_{m}, q', \tilde{a}a)$ is isomorphic to $(1_{m}, p, b^{*})$
 via $1_{m}, b^{*}a^{*}\tilde{a}^{*} $.
\end{proof}

We are now ready to prove Theorem \ref{3:70}. We begin with surjectivity.
Let $(1_{m}, q, a)$ be in $\Gamma(\ker(\delta) \cap A; A)$ and assume 
it is in standard form and that $a$ is in $M_{n}(\tA)$.
 We apply Lemma \ref{2:60} to the map 
$\delta: A + B \rightarrow E$ to obtain $\epsilon_{2n}$. By Lemma 
\ref{4:75}, we may find a partial isometry $b$ in 
$M_{2n}(\tilde{B})$ such that $\delta(b^{*}b)=0$, $bb^{*}=1_{m}$ an d
$\Vert \delta(ab) \Vert  < \epsilon_{2n}$. Let $p=b^{*}b$. From 
Proposition \ref{2:10}, we know that 
\[
[1_{m}, q, a] + [p,1_{m}, b] = [p, q, ab].
\]
By Lemma \ref{4:80}, we may write $ab = c_{1}c_{2}$, where
$c_{1}, c_{2}$ are partial isometries in $M_{2n}(A+B)$
 with $\delta(c_{2}) =0$ and $(q,q,c_{1})$ elementary.
 It follows that $(p, q, ab)$ is isomorphic to
 $(q,q,c_{1})$ via $c_{2}, q$ and hence
 it represents the zero element of $K_{0}(\ker(\delta);A+B)$. 
 It follows from Theorem  \ref{3:10} that 
 $\alpha[1_{m},p, b^{*}] = [1_{m}, q, a]$.

We now turn to the issue in injectivity.
In view of Theorem \ref{3:10}, it suffices to prove that
the map from $K_{0}(B \cap \ker(\delta); B)$ to $K_{0}(\ker(\delta), A+B)$
is injective. Let $(1_{m}, q, b)$ be an element of 
$\Gamma(\ker(\delta) \cap B; B)$ which is in standard form and 
$b$ is in $M_{n}(\tilde{B})$ and such that its class in
$K_{0}(\ker(\delta), A+B)$ is zero. From Proposition 
\ref{2:10}, there exist $(1_{k}, 1_{k},c_{1})$
 and $(1_{k+m}, 1_{k+m}, c_{2})$
which are elementary and in standard form such that
$c_{2}(b \oplus c_{1})^{*}$ is in $\ker(\delta)$.

As $c_{1}, c_{2}$ are elementary, we may find a partial isometry
$c_{1}(t)$ in \newline 
$M_{2k}((C_{0}(0,1] \otimes (A+B))^{\sim})$ 
such that $(1_{k}, 1_{k}, c_{1}(t))$ is in standard form 
in 
$\Gamma( \ker(1_{C_{0}(0,1]} \otimes \delta) \cap C_{0}(0,1] \otimes(A+B) ;
C_{0}(0,1] \otimes (A+B))$. We may also find $c_{2}(t)$ in an analogous way
in $M_{2(k+m)}(C_{0}(0,1] \otimes (A+B))^{\sim})$.

By part 3 of Proposition \ref{4:60}, we know that conditions C1 and 
C2 also hold for $C_{0}(0,1] \otimes A$ and 
$C_{0}(0,1] \otimes B$. 
We apply Lemma \ref{2:10} to find partial isometries 
$a_{1}(t)$ in $M_{2k}((C_{0}(0,1] \otimes A)^{\sim})$
 and $a_{2}(t)$ in $M_{2(k+m)}((C_{0}(0,1] \otimes A)^{\sim})$ such that
 \begin{eqnarray*}
 \delta(a_{1}c_{1}^{*}) & = &  \delta(a_{2}c_{2}^{*}) = 0,\\
 a_{1}^{*}a_{1} & = & 1_{k}, \\
 a_{2}^{*}a_{2} & = & 1_{k+l}.
 \end{eqnarray*}
 Note that $a_{i}a_{i}^{*} = (a_{i}c_{i}^{*})(a_{i}c_{i}^{*})^{*}$ is in
 $\ker(\delta) $.

With  $\epsilon = \epsilon_{4(k+m)}/2$, we now appeal to Lemma 
\ref{4:75} to find $b_{1}(t) $ in \newline
 $M_{4k}((C_{0}(0,1] \otimes B)^{\sim})$
and $ b_{2}(t)  $ in $M_{4(k+m)}((C_{0}(0,1] \otimes B)^{\sim})$ such that 
\begin{eqnarray*}
(1 \otimes \delta)(b_{i}(t)^{*}b_{i}(t))&  = &  0, \\ 
b_{1}(t)b_{1}(t)^{*} & =  & 1_{k}, \\
b_{2}(t)b_{2}(t)^{*} & = &  1_{k+m}, \\
\Vert (1 \otimes \delta)(a_{i}(t)b_{i}(t)) \Vert & < & \epsilon.
\end{eqnarray*}

Now there is a small problem in that $b_{1}(t)^{*}b_{1}(t)$ 
may not be constant in $t$. But it is a continuous path of projections in
$M_{4k}(\ker(\delta) \cap B)$ which begins
 at $1_{k}$ at $t=0$. Hence we may find
a continuous path of partial isometries
$b_{3}(t)$ in $M_{4k}(\ker(\delta) \cap B)$ with $b_{3}(t)=1_{k}$, 
$b_{3}(t)b_{3}(t)^{*} = b_{1}(t)^{*}b_{1}(t), b_{3}(t)^{*} b_{3}(t)=1_{k}$.
By replacing $b_{1}(t)$ by $b_{1}(t)b_{3}(t)$, we may assume 
$b_{1}(t)^{*}b_{1}(t) = 1_{k}$  for all $t$, without changing the
other properties of $b_{1}(t)$. We do the same 
for $b_{2}(t)$. 

It follows that if we define 
$b_{1} = b_{1}(1), b_{2}=b_{2}(1)$, then 
$(1_{k}, 1_{k}, b_{1})$ and $(1_{k+m}, 1_{k+m}, b_{2})$ are elementary 
in $\Gamma(\ker(\delta) \cap B; B)$. Moreover, we can write
\begin{eqnarray*}
b_{2}^{*}(b \oplus b_{1}^{*})^{*} & = &  b_{2}^{*}(b^{*} \oplus b_{1}) \\
  &  =  & b_{2}^{*}a_{2}^{*}a_{2}c_{2}^{*}c_{2} 
  (b^{*} \oplus c_{1}^{*}c_{1} a_{1}^{*}a_{1} b_{1}) \\
    &  =  &  (a_{2} b_{2})^{*} (a_{2}c_{2}^{*})c_{2} (b^{*} \oplus c_{1}^{*}) 
      ( 1_{m} \oplus (c_{1}a_{1}^{*}) )(a_{1}b_{1})
      \end{eqnarray*}
      
      Observe that the terms are grouped so that 
      each group is in the kernel of $\delta$, except the first and last, 
      $a_{2}b_{2}, a_{1}b_{1}$. It follows that 
 \[
\Vert  \delta(b_{2}^{*}(b \oplus b_{1}^{*})^{*}) \Vert < 2 \epsilon.
\]
By Lemma \ref{4:80}, we may find partial isometries 
$b_{4}, b_{5}$ in $\mathcal{M}(\tilde{B})$ such that 
\[
b_{2}^{*}(b \oplus b_{1}^{*})^{*} =b_{4}b_{5} 
\]
$\delta(b_{5}) = 0$ and $(1_{k+m}, 1_{k+m}, b_{4})$ is elementary.
It follows that $(1_{m} \oplus 1_{k}, q \oplus 1_{k}, b \oplus b_{1})$
is isomorphic to $(1_{k+m}, 1_{k+m}, b_{4}^{*}b_{2}^{*})$ via
$1_{k+m}, b_{5}$ As the latter is elementary, it follows that 
$[1_{m}, q, b]=0$ in $K_{0}(\ker(\delta) \cap B; B)$.

\section{Groupoid $C^{*}$-algebras}
\label{5}

We are interested in applying our excision result
to various groupoid $C^{*}$-algebras.
The first question which arises, even before we get 
to some of the more subtle hypotheses, is when 
the $C^{*}$-algebra of one groupoid lies in the multiplier of another.
This issue is the focus of this section.
We refer the reader to \cite{Ren:LNM} and \cite{Wil:grpbook}
as standard references on groupoids and their $C^{*}$-algebras.

The question for us will reduce to the following: suppose that
$G$ is a locally compact, Hausdorff groupoid
 with a Haar system and $H$ is a subgroupoid of $G$, endowed with 
 its own locally compact, Hausdorff topology which 
 is finer than the relative topology of $G$. 
 Under what circumstances can we conclude that $C^{*}_{r}(G)$ lies 
 in the multiplier algebra of $C^{*}_{r}(H)$.
 
 If this situation does not seem intuitive, let us provide some 
 justification. The first  is that it arises quite naturally as
 we will show in the next two sections. Secondly, let
 us offer the following fairly simple example, which is a special case.
 
 It is well-known that the category of unital, commutative $C^{*}$-algebras
 is isomorphic to that of compact, Hausdorff spaces with continuous maps.
 If we enlarge this to include non-unital commutative $C^{*}$-algebras, 
 we replace 'compact' by 'locally compact' and also require the 
 maps to be proper. On the other hand, there is 
 an alternative (less categorical) possibility. Suppose that 
 $X, Y$ are locally compact, Hausdorff spaces and $\alpha: Y \rightarrow X$
 is simply continuous. Then $\alpha$ induces a $*$-homomorphism, $\rho$, 
 from $C_{0}(X)$ to the algebra of continuous bounded functions on $Y$, which 
 is multiplier algebra of $C_{0}(Y)$. The point is that, for $f$ in $C_{0}(X)$, 
 while $f \circ \alpha$ may not be compactly supported, it is 
 continuous, and its product with any compactly supported function 
 on $Y$ will be compactly supported and continuous.
 
 If we additionally assume that $\alpha$ is injective,
  then we can simply identify
 $\alpha(Y)$ with its image in $X$, which we will simply denote $Y$. 
 But the map $\alpha$ induces a topology on $Y$ which is possibly finer 
 than the relative topology from $X$. 
The map $\rho$ we described above is now simply the 
restriction of functions from $X$ to $Y$. 
 If we consider $X$ and $Y$ with the co-trivial
 groupoid structures (that is, the equivalence relation which
 is equality), then this is exactly the situation we outlined above
 and we have a positive answer to our question. In general, of course, more
 hypotheses are needed.

  We first discuss the structure and relations in the 
  setting of $H \subseteq G$, but at the end of the section
  we will provide some constructions for obtaining $H$ from $G$ 
  which are inspired by the seminal work of Muhly, Renault and 
  Williams \cite{MRW:grpeq}.

Let us begin by setting out some standard notation 
for groupoids
and groupoid $C^{*}$-algebras.

Let  $G$ be a    groupoid. That is, there
there is a  set of composable pairs, $G^{2} \subseteq G \times G$
 with a 
product  from $G^{2}$ to $G$. We assume this is 
associative, as in \cite{Ren:LNM}.
It will sometimes be convenient to
  denote the product map as $\mu_{G}$ from 
$G^{2}$ to $G$. 
The units of $G$ are denoted by $G^{0}$ and the range
 and source maps 
$r_{G}, s_{G} :G \rightarrow G^{0}$, 
$r_{G}(g) = gg^{-1}, s_{G}(g) =g^{-1}g$, for
all $g$ in $G$. When no confusion can arise, we will drop
the subscripts. We let 
\[
G^{u} = r_{G}^{-1}\{ u \}, G_{u} = s_{G}^{-1}\{ u \},
\]
 for each $u$ in $G^{0}$.
 A subset $X$ of $G^{0}$ is called $G$-invariant if, whenever $g$ is in $G$
 with $r_{G}(g) $ in $X$, then $s_{G}(g)$ is also in $X$.
 
We also assume that $G$ is a topological groupoid;
 that is, it 
has a topology, $\mathcal{T}_{G}$, in which the inverse
 and product are continuous. 
Moreover, we assume throughout that it
is second countable, locally compact and  Hausdorff.

Recall the definition of  a Haar system for $G$:
 for each $u$ in $G^{0}$, there is a measure $\nu^{u}$ on $G^{u}$
with full support and such that the function
\[
u \rightarrow \int_{G^{u}} f(x) d\nu^{u}(x)
\]
is continuous, for every $f$ in $C_{c}(G)$, the continuous 
compactly-supported  functions on $G$. We also assume that $\nu$ is 
left-invariant in the sense that, for any $g$ in $G$ and
$f$ in $C_{c}(G)$, we have 
\[
\int_{G^{s(x)}} f(xy) d\nu^{s(x)}(y) = 
\int_{G^{r(x)}} f(y) d\nu^{r(x)}(y).
\]

The following result is well-known, but worth recording here.

\begin{lemma}
\label{5:10}
Let $G$ be a locally compact, Hausdorff groupoid.
The following are equivalent:
\begin{enumerate}
\item The map $r_{G} : G \rightarrow G^{0}$ is open.
\item The map $s_{G} : G \rightarrow G^{0}$ is open.
\item The product map $\mu_{G}: G^{2} \rightarrow G$ is open.
\end{enumerate}
If the groupoid has a Haar system, then all three conditions are satisfied.
\end{lemma}

\begin{proof}
The equivalence of the first two follows from
the facts  that  $x \rightarrow x^{-1}$ is a homeomorphism
and $r(x^{-1}) =s(x)$.

Let us assume the first condition holds and prove the third. 
Let $(x, y)$ be in $G^{2}$, $x \in U, y \in V$, open and 
suppose that $\mu(U \times V \cap G^{2})$ is not open.
Then there exists a sequence $z_{k}$ converging to
$xy$ which is not in $\mu(U \times V \cap G^{2})$. Choose 
a decreasing sequence of open sets, $U_{l}, l \geq 1$, 
contained in $U$ and
with intersection $\{ x \}$. Since $r$ is open, $r(U_{l})$ is open
and contains $r(x)$, which is the limit of the sequence
$r(z_{k})$, as $r$ is continuous. Hence, for each $l$, we may find 
$k_{l}$ and $x_{l} $ in $U_{l}$ such that $r(x_{l}) = r(z_{k_{l}})$.
Hence, the sequence $x_{l}$ is converging to $x$ and 
$x_{l}^{-1}z_{k_{l}}$ is converging to $x^{-1} (xy) = y$.
So for $l$ sufficiently large, $x_{l}^{-1}z_{k_{l}}$ is in $V$
and hence $z_{k_{l}} = x_{l}(x^{-1}_{l} z_{k_{l}})$ is in 
$\mu(U \times V \cap G^{2})$, a contradiction.

Finally, let us assume the third condition and prove 
the first. It is a general fact that if
$f: X \rightarrow Y$ is a continuous open map and
 $Z \subseteq Y$ is closed, then $f|_{f^{-1}(Z)}: f^{-1}(Z) \rightarrow Z$
 is also continuous and open. We apply
  this fact to $\mu: G^{2} \rightarrow G$ and 
 the closed subset $G^{0} \subseteq G$. Here, we have 
 \[
 \mu^{-1}(G^{0}) = \{ (x, x^{-1}) \mid x \in G \} 
 \]
 which is homeomorphic to $G$ via the map $\alpha(x) = (x, x^{-1})$.
 We note that $\mu \circ \alpha=r$ and the proof is complete.
 
 The last statement is Proposition 2.4 of Chapter 2 of \cite{Ren:LNM}.
\end{proof}

We now begin to consider the situation we outlined of $H \subseteq G$. 
The following is an extension result that will be needed shortly.

\begin{lemma}
\label{5:20}
Let $(G, \mathcal{T}_{G})$ and $(H, \mathcal{T}_{H})$
 be  locally compact, Hausdorff 
topological groupoids and that 
 $\nu^{u}, u \in G^{0}$ is a Haar system for $G$.
Assume that $  H$ is a subgroupoid of  $ G$
such that the topology $\mathcal{T}_{H}$ is finer
than the relative topology on $H$ of $\mathcal{T}_{G}$ 
and  that,
for every $u$ in $H^{0}$, 
$H^{u} = G^{u}$ and the relative topologies from 
$\mathcal{T}_{H}$ and $\mathcal{T}_{G}$ agree on this set.

Let $Y \subseteq H$ be a subset which is closed in 
both topologies and such that the two topologies agree
on this set. Let $K \subseteq H^{0}$ be compact, $M \geq 1$ and
suppose that $a: H \rightarrow \C^{M}$ is continuous, compactly
supported and $r(supp(a)) \cup s(supp(a)) \subseteq K$.
Then there exists $b: G \rightarrow \C^{M}$ continuous 
and compactly supported such that
$b(x) = a(x)$, for all $x$ in $Y \cup r^{-1}(K) \cup s^{-1}(K)$.
\end{lemma}

\begin{proof}
Let $U = \{ x \in H \mid a(x) \neq 0$. It follows that 
$\overline{U}$ is compact in $H$, hence also in $G$ and the two 
topologies agree on this set. As $Y$ and $\overline{U}$ are closed in 
both $H$ and $G$, the two topologies agree on $Y \cup \overline{U}$.

Let $X = r^{-1}( K) \cup  s^{-1}( K)$, which is closed in $G$.

We define $b(x) = a(x)$, for all $x$ in $\overline{U} \cup Y$ and we define
$b(x) = 0$, for all $x$ in $X - U$. First, observe that 
this is well-defined, for if $x$ is in $X - U$, then $a(x)=0$.
Secondly, both of these sets are closed in $G$. Thirdly, $b(x)$ 
is continuous on both: on the first because the two relative 
topologies agree there and on the second because it is constant there.
Finally, the support of $b(x)$ is contained in $U$, which is 
pre-compact. 
 Hence, by the Tietze extension theorem (Proposition 1.5.8 of
\cite{Ped:AN}), $b$ may be extended to a continuous function of compact support 
on $G$.

It is clear that $b(x)=a(x)$, for any $x$ in $Y$. Now suppose
that $r(x)$ is in $K$. If $x$ is in $U$, then $b(x)=a(x)$, by definition.
If $x$ is not in $U$, then $a(x)=0$ by the
 hypothesis on $U$, while $b(x)=0$ as $x$ is in $X - U$. 
 The case for $s(x) \in K$ is similar.
\end{proof}

\begin{thm}
\label{5:30}
Let $(G, \mathcal{T}_{G})$ and $(H, \mathcal{T}_{H})$
 be  locally compact, Hausdorff 
topological groupoids and that 
 $\nu^{u}, u \in G^{0}$ is a Haar system for $G$.
Assume that $  H$ is a subgroupoid of  $ G$
such that the topology $\mathcal{T}_{H}$ is finer
than the relative topology on $H$ of $\mathcal{T}_{G}$ 
and  that,
for every $u$ in $H^{0}$, 
$H^{u} = G^{u}$ and the relative topologies from 
$\mathcal{T}_{H}$ and $\mathcal{T}_{G}$ agree on this set.

Then the following are equivalent:
\begin{enumerate}
\item the map $r: H \rightarrow H^{0}$ is open (in $\mathcal{T}_{H}$).
\item the map $s: H \rightarrow H^{0}$ is open (in $\mathcal{T}_{H}$).
\item
$\nu^{u}, u \in H^{0}$, is a Haar system for 
$H$.
\end{enumerate}
\end{thm}

 \begin{proof}
 The first two conditions are clearly equivalent since $r(x^{-1})=s(x)$ and 
 $x \rightarrow x^{-1}$ is a homeomorphism. The third condition
 implies the first from  Lemma \ref{5:10} or 
 Proposition  2.4, Chapter 2 of \cite{Ren:LNM}. 
 
 It remains for us to prove that the first condition implies the third.
 First, it is clear from the fact that the two topologies agree
 on all sets $H^{u}=G^{u}$ that the measures are well-defined,
 have the desired support and are left-invariant. It remains for us to verify 
 the continuity property.
 
 Let $h$ be any continuous function of compact support on $H$.
 Let $U = \{ x \in H \mid h(x) \neq 0\}$.
 Its closure $\bar{U}$ is compact in $H$ and
  so is $K =  r(\bar{U}) \cup s(\bar{U})$. We apply Lemma \ref{5:20}
  to find $g$ in $C_{c}(G)$ such that 
  $g(x) = h(x)$ for all $x$ with $r(x)$ or $s(x)$ in $K$.
  
  We consider the function $I(u) = \int_{H^{u}} h(x) d\nu^{u}(x)$ defined
  on $H^{0}$ and show it is continuous. We appeal to the same
  general topology result as in the last proof.
   First, on the set $r(\bar{U})$, it agrees with the function
  $J(u) = \int_{G^{u}}  g(x) d\nu^{u}(x)$. As $\nu^{u}$ is a Haar system,
   this function is continuous on $G^{0}$ in the topology $\mathcal{T}_{G}$ 
   and so its restriction to 
   $r(\bar{U})$ is continuous in the topology $\mathcal{T}_{H}$. 
   Secondly, on the set $H^{0} \setminus r(U)$, it is clearly $0$, 
   which is continuous in any topology.  These two sets cover $H^{0}$. 
   The first is closed as we observed above that it is compact.
   The second is closed since $U$ is open and our hypothesis is that $r$ is an 
   open map.
 \end{proof}

We consider the left regular representation of $G$, $\lambda_{G}$.
For each unit, $u$ in $G^{0}$,  we define the measure $\nu_{u}$ on 
$G_{u}$ by $\nu_{u}(E) = \nu^{u}(E^{-1})$, for any Borel set
$E$ in $G_{u}$. We let $L^{2}(G_{u}, \nu_{u})$ be the corresponding 
Hilbert space and define (with a slight abuse of notation)
$L^{2}(G, \nu) = \oplus_{u \in G^{0}} L^{2}(G_{u}, \nu_{u})$.
It is worth noting that the elements can be seen as functions on $G$.
For each $f$ in $C_{c}(G)$ and $u$ in $G^{0}$, we define the
 operator $\lambda_{u}(f)$
on $L^{2}(G_{u}, \nu_{u})$
\[
\left( \lambda_{u}(f) \xi \right)(x) = \int_{y \in G_{u}}
  f(xy^{-1}) \xi(y) \sigma(xy^{-1},y) d\nu_{u}(y),
  \]
  for $f$ in $C_{c}(G)$, $\xi$ in $L^{2}(G, \nu)$ and $x$ in  $G_{u}$.
  We also let $\lambda_{G}(f)$ be $\oplus_{u \in G^{0}} \lambda_{u}(f) $.
  
We  let $\lambda_{H}$ denote the left regular representation
of $H$ on $L^{2}(H, \nu)$. From the fact that, for every unit $u$,
 $G_{u} = H_{u}$ and assuming they
have the same topology $G$ and in $H$, the Hilbert space
$L^{2}(H, \nu)$ is a closed subspace of $L^{2}(G, \nu)$. We use this inclusion
implicitly.  Furthermore, we can regard the operators 
$\lambda_{H}(f), f \in C_{c}(H),$
and also those in the operator-norm closure of these as being defined
on $L^{2}(G, \nu)$ by setting them to be zero on the orthogonal complement
of $L^{2}(H, \nu)$.

\begin{thm}
\label{5:50}
Let $(G, \mathcal{T}_{G})$ be a locally compact, 
Hausdorff topological \newline 
 groupoid
with Haar system, $\nu^{u}, u \in G^{0}$,
 and $2$-cocycle $\sigma$.
Let
$H \subseteq G$ be  a subgroupoid with a topology, $\mathcal{T}_{H}$,
 in which it is also
locally compact and Hausdorff. Suppose that
\begin{enumerate}
\item 
$H^{0} \subseteq G^{0}$ is $G$-invariant,
\item 
the topology $\mathcal{T}_{H}$ is finer than the relative
topology of $\mathcal{T}_{G}$ on $H$, 
\item 
For every $u$ in $H^{0}$, the topologies 
$\mathcal{T}_{G}$ and $\mathcal{T}_{H}$ agree on 
$H^{u} = G^{u}$ \newline 
and $G_{u} = H_{u}$.
\item 
$\nu^{u}, u \in H^{0}$ is a Haar system for 
$H$.
\end{enumerate}

Then the following hold.
\begin{enumerate}
\item
For each $f$ in $C_{c}(G)$ and $g$ in $C_{c}(H)$, the functions
\begin{eqnarray*}
(fg)(x) & = &  \int_{H^{u}} f(y) g(y^{-1}x) \sigma(y, y^{-1}x)  d\nu^{u}(y) \\
(gf)(x) & = &  \int_{H^{u}} g(y) f(y^{-1}x) \sigma(y, y^{-1}x) d\nu^{u}(y)
\end{eqnarray*}
are well-defined and  in $C_{c}(H)$. 
\item For each $f, g$ as above we have 
\begin{eqnarray*}
\lambda_{G}(f) \lambda_{H}(g)   & =  & \lambda_{H}(fg),  \\
\lambda_{H}(g) \lambda_{G}(f) & =  & \lambda_{H}(gf), 
\end{eqnarray*}
\item 
This defines a  map sending $f$ in $C_{c}(G)$ to $\rho(f)$ in $M(C_{c}(H))$
which
extends to a $*$-homomorphism from $C^{*}_{r}(G, \sigma)$ to \newline
$M(C^{*}_{r}(H, \sigma))$ which is injective
 if and only if $H^{0}$ is dense
in $G^{0}$.
\end{enumerate}
\end{thm}

We will not give the proof which is really quite straightforward.

It is observed in Remark (iii) on page 59 and Proposition 1.14
 of \cite{Ren:LNM} 
that there is a natural way that $C_{c}(H^{0})$ acts as 
multipliers of $C_{c}(H)$:
\[
(ef)(x) = e(r(x))f(x), (fe)(x) = f(x)e(s(x)),
\]
for each $f$ in $C_{c}(H)$, $e$ in $C_{c}(H^{0})$ and $x$ in $H$.
In our case here, this representation of $C_{c}(H)$ interacts with
$C_{c}(G)$ in a particularly nice way.
Observe that the two formulae above make sense equally well if
 $e$ is in $C_{c}(H^{0})$ and $f$ is in $C_{c}(G)$, the results being
 functions on $H$ and it is easy to see that they are both continuous
 and have compact support.
 
 There remains one technical issue in this construction. Suppose that 
 $f$ is a continuous function of compact support on $H$. Let us suppose 
 for the moment that $H$ is closed in $G$ (which is rarely the case).
 Then the Tietze
 Extension Theorem guarantees the existence  of a 
 continuous function $\tilde{f}$ on $G$ such that $\tilde{f}|_{H} = f$.
 It is a more subtle question to ask if $\tilde{f}$ 
 may be chosen so that the norm, $\Vert \tilde{f} \Vert_{r}$,
 can be controlled in some way by  $\Vert f\Vert_{r}$, 
 independently of $f$.

 This is  true if we
 replace the reduced $C^{*}$-norm by the uniform norm. 
The proof is quite standard but it will be helpful to examine it, as
 we will do in a moment.

As we noted $H$ itself is usually not closed in $G$ and it is necessary for
us  to restrict our attention
to  subsets of $H$ which \emph{are}
 closed in $G$, which usually rules out $H$ itself.

\begin{defn}
\label{5:60}
Let $(G, \mathcal{T}_{G})$ be a locally compact, 
Hausdorff topological \newline 
groupoid
with Haar system, $\nu^{u}, u \in G^{0}$,
 and $2$-cocycle $\sigma$.
Let
$H \subseteq G$ be  a subgroupoid with a topology, $\mathcal{T}_{H}$,
 in which it is also
locally compact and Hausdorff. 

For any $C \geq 1$, we
 say that a set $X \subseteq H$ which is closed in $G$ (and hence also in $H$)
has the \emph{$C$-extension property}
if,
 for any $f$ in $C_{c}(H)$ with support in $X$, there exists 
 $\tilde{f}$ in $C_{c}(G)$ such that $\tilde{f}|_{X} = f|_{X}$
 and 
 \[
 \Vert \tilde{f} \Vert_{r} \leq C \Vert f \Vert_{r}.
 \]
\end{defn} 

Let us just check that the property holds (with $C=1$)
if we use the uniform norm
instead of the  reduced $C^{*}$-norm.

Let $X$ be a subset of $H$ which is closed in $G$. 
Let $f$ be in $C_{c}(H)$ with support in $X$. Consider $X \cup \{ \infty\}$
as a closed subset of $G \cup \{ \infty\}$, the one-point
 compactification of $G$. Extending $f|_{X}$ to be zero
  at $\infty$, this function is continuous on $X$,
   with the topology from $G$ and 
 we may apply the Tietze Extension Theorem to find $\tilde{f}$, a 
 continuous function of $G \cup \{ \infty\}$ which agrees with $f$ on $X$
 and is zero at $\infty$. It is a simple matter to check that this 
 can be modified so that $\tilde{f}$ is actually compactly supported.
 
 It remains to worry about the norm of the extension.
Define the function
$h(t)$ to be $1$ for $t$ in the interval $[0, \Vert f \Vert^{2}]$, and 
$\Vert f \Vert^{-1}t^{-1/2}$ for $t > \Vert f \Vert^{2}$. It is
an easy exercise to check that  $\tilde{f} h(  \tilde{f}^{*} \tilde{f})$
satisfies all the desired properties. Notice that
 $h(  \tilde{f}^{*} \tilde{f})$ exists in the
  unitization $C(G \cup \{ \infty \})$.

It is worth asking why this same argument does not suffice for
general groupoids. The answer is that $h(  \tilde{f}^{*} \tilde{f})$ 
exists in the unitization of $C^{*}_{r}(G)$, but not 
necessarily in the unitization of $C_{c}(G)$ and the element
$\tilde{f} h(  \tilde{f}^{*} \tilde{f})$ exists in $C^{*}_{r}(G)$, 
but not necessarily in $C_{c}(G)$.

Now we could replace the function $h$ above by some polynomial which 
approximates our given $h$ and our final element would indeed lie in 
$C_{c}(G)$, but then we cannot be sure of the condition 
$\tilde{f}|_{X} = f|_{X}$.

 It follows
 then that in the case that $H$ and $G$
are co-trivial, $G= G^{0}, H=H^{0}$, where the two norms agree, this holds.
In generality, it seem to be a subtle issue, although it does hold in many
cases of interest.

Let us just observe the following positive result in a very special case.

 \begin{prop}
 \label{5:70}
 Let $G,H$ be as in Definition \ref{5:60} and assume that 
 $H$ is closed in $G$. Then $G - H$ is also a locally 
 compact groupoid with Haar system. Assume we have a short exact sequence
 \[
 0 \rightarrow C_{r}^{*}(G - H) \rightarrow C_{r}^{*}(G) 
 \rightarrow C_{r}^{*}(H) \rightarrow 0.
 \]
 (See \cite{Ren:LNM}, \cite{AD:exact} for further discussion.)
 Then $H$ has the $C$-extension property, for any $C > 1$.
 \end{prop}

\begin{proof}
Let the quotient map from $C^{*}_{r}(G)$ to $C^{*}_{r}(H)$
be  denoted by $\rho$.
Let $f$ be in $C_{c}(H)$. It is a consequence of the 
Tietze Extension Theorem that we may find $g$ in $C_{c}(G)$ 
with $\rho(g) = g|_{H} = f$. It follows that 
\[
\Vert f \Vert_{r} = \inf\{ \Vert f + b \Vert_{r} \mid b \in 
C_{r}^{*}(G \setminus H) \}.
\]
The desired conclusion follows for any $C > 1$  from the fact
that $C_{c}(G \setminus H)$ is dense in $C_{r}^{*}(G \setminus H)$.
\end{proof}

I do not know if the converse holds (the $C$-extension property
implies exactness), but this does suggest
that the property is linked with amenability/exactness
 in some way.

Our next task is to provide
 a fairly general way of constructing 
groupoids $ H \subseteq G$ from $G$. This follows ideas of 
Muhly, Renault and Williams \cite{MRW:grpeq}.

Let $(G, \mathcal{T}_{G})$ be a locally compact, 
Hausdorff topological groupoid
with Haar system, $\nu^{u}, u \in G^{0}$,
 and $2$-cocycle $\sigma$.
Suppose that $Y $ is a closed subset of $G^{0}$. 
We can form
\[
G_{Y}^{Y} = \{ g \in G \mid r(g), s(g) \in Y \}
\] 
which is obviously a closed subgroupoid of $G$ with unit
space $Y$. It also acts on the left 
of 
\[
G^{Y} = \{ g \in G \mid r(g) \in Y \}.
\]
Observe  that this action is \emph{free} in the sense that 
$g \cdot x = x$ for $g $ in $G^{Y}_{Y}$ and $x$ in $G^{Y}$
with $s(g) = r(x)$ 
only if $g$ is a unit. It is
 also proper in the sense that 
the map sending $(g,x)$ in $G_{Y}^{Y} \times G^{Y}$ with $s(g) = r(x)$
to $(gx, x)$ in $G^{Y} \times G^{Y}$ is a proper map.

We define 
\[
\tilde{H} = \{ (x, y) \in G^{Y} \times G^{Y} \mid r(x)=r(y) \}
\]
which is equipped with an action of $G^{Y}_{Y}$ by 
$g(x,y) = (gx, gy)$, for $(x,y)$ in $H$ and
$g$ in $G_{Y}^{Y}$ with $s(g) = r(x)=r(y)$. This action is 
also free and proper.

\begin{thm}
\label{5:100}
Let $(G, \mathcal{T}_{G})$ be a locally compact, 
Hausdorff topological \newline 
groupoid
with Haar system, $\nu^{u}, u \in G^{0}$, and 
 a $2$-cocycle, $\sigma$.
Suppose that $Y $ is a closed subset of $G^{0}$
such that 
\[
r, s: G^{Y}_{Y} \rightarrow Y
\]
are open. If we define 
\[
H = \{ x^{-1}y \mid (x,y) \in \tilde{H} \},
\]
and endow it with the quotient topology from the map
sending 
$(x,y) $ in $\tilde{H}$ to $x^{-1}y$, then it satisfies 
the hypotheses of Theorem \ref{5:50}. In addition, 
$G^{Y}$ is a $G^{Y}_{Y}-H$-equivalence bimodule.
\end{thm}

We will not give a proof of this result. Most of it follows the
techniques of \cite{MRW:grpeq} or is quite straightforward.
We would like to examine a special case where 
$G$ is a 
transformation groupoid.

Suppose that $X$ is a locally compact Hausdorff space and 
$\Gamma$ is a locally compact Hausdorff topological group
which acts on $X$ by homeomorphisms on the right: that is there is a map
sending $(x, \gamma)$ in $X \times \Gamma$ to $x\gamma$ in $X$
which is continuous and satisfies $(x\gamma)\gamma' = x(\gamma\gamma')$,
for all $x$ in $X$, $\gamma, \gamma'$ in $\Gamma$.

 Let  $G$ be the associated 
transformation groupoid:
\[
G = X \rtimes \Gamma,
\]
which is simply $X \times \Gamma$ as a set 
with product given by 
$(x_{1}, \gamma_{1}) (x_{2}, \gamma_{2}) = ( x_{1}, \gamma_{1}\gamma_{2})$
 if 
$x_{1} \cdot \gamma_{1} = x_{2}$ and inverse
$(x, \gamma)^{-1} = (x \cdot \gamma, \gamma^{-1})$. The unit space is
 $G^{0} = X \times \{ e_{\Gamma} \}$. For notational 
 convenience, we will usually write this as simply $X$. It is given the 
product topology. It has a Haar system by transferring the Haar measure
from $\Gamma$ to $G^{(x,e)} = \{ x \} \times \Gamma$ in the obvious way, 
for any $x$ in $X$.

We remark that in this example, the reduced groupoid $C^{*}$-algebra, $C^{*}_{r}(G)$, 
coincides with the reduced crossed product algebra, $C_{0}(X) \rtimes_{r} \Gamma$.

\begin{defn}
\label{5:105}
Suppose that $X$ is a locally compact Hausdorff space and 
$\Gamma$ is a locally compact, Hausdorff topological group
which acts on $X$ by homeomorphisms
If $Y$ is a closed subset of $X$, we say that $Y$ is 
\emph{$\Gamma$-semi-invariant}  if, for every $\gamma$ in $\Gamma$
either $Y \gamma = Y$ or $Y \gamma \cap Y$ is empty.
In this case, we
denote by $\Gamma_{Y}$ the set of $\gamma$ for
 which the first condition holds.
\end{defn}

Notice that this includes the case, $Y \gamma \cap Y = \emptyset$, for all
$\gamma \neq e$.
Observe that $\Gamma_{Y}$ 
 is clearly a subgroup of $\Gamma$. Also notice that for any $\gamma$ 
in $\Gamma$,  the set $Y\gamma$ depends only on the right coset
$\Gamma_{Y}\gamma$.

The following is an easy consequence of the definitions and we will not
give a proof.

\begin{thm}
\label{5:110}
Suppose that $X$ is a locally compact Hausdorff space and 
$\Gamma$ is a locally compact Hausdorf topological group
which acts on $X$ by homeomorphisms. Suppose that $Y$ is a closed 
$\Gamma$-semi-invariant subset such that $\Gamma_{Y} \backslash \Gamma$
is discrete. Then $\Gamma_{Y}$ is open in $\Gamma$.
Moreover,  $G^{Y}_{Y}  \cong Y \rtimes \Gamma_{Y}$, 
where the latter  is regarded as a transformation groupoid and 
the maps $r, s: G^{Y}_{Y} \rightarrow Y$ are open. If $H$ is groupoid
given in \ref{5:100},  then the unit space of $H$ is 
\[
H^{0} = \cup_{\Gamma_{Y} \gamma \in \Gamma_{Y} \backslash \Gamma} Y \gamma
\]
and has the inductive limit topology. With this identification, 
$H$ is isomorphic to 
the transformation group
\[
\left( \cup_{\Gamma_{Y} \gamma \in \Gamma_{Y} \backslash \Gamma} 
Y \gamma \right) \rtimes \Gamma.
\]
\end{thm}

These examples seem to be more accessible for the extension property
 mentioned earlier. While the following deals with a
  couple of specific closed subsets of $H$,
   it would seem the techniques of proof could be applied more generally
   and we will do so in future sections.
   
\begin{thm}
\label{5:120}
Let  $X$ be a locally compact Hausdorff space and 
$\Gamma$ be a
 locally compact Hausdorff topological group
acting  on $X$ by homeomorphisms. Suppose that $Y$ is a closed
$\Gamma$-semi-invariant subset of $X$. Suppose that
 $\Gamma_{Y}\backslash \Gamma$ is discrete. Let $H$ be as
  in  \ref{5:100}. 
  
 If there is a short exact sequence
 \[
0 \rightarrow C_{0}(X - Y) \rtimes_{r} \Gamma_{Y} 
  \rightarrow C_{0}(X ) \rtimes_{r} \Gamma_{Y} 
   \rightarrow C_{0}( Y) \rtimes_{r} \Gamma_{Y} 
   \rightarrow 0
 \]
 then  the closed sets $G_{Y}, G^{Y} \subseteq H$ have 
  the extension property
  of Definition \ref{5:60} .
\end{thm}

We will not give a proof. First, we will not need the result. Secondly, the proof 
is very similar to the one given for the last item in Theorem
\ref{6:250}.

\section{Application to subgroupoids}
\label{6}

In this section, we consider a groupoid $G$ and open subgroupoid $G'$
and give conditions under which our excision theorem can 
be applied to the reduced groupoid $C^{*}$-algebra, $B= C^{*}_{r}(G)$, 
and a $C^{*}$-subalgebra, $B' = C^{*}_{r}(G')$.
 
We assume that $G$ is a locally compact, Hausdorff, second countable
 groupoid
with Haar system $\nu^{u}, u \in G^{0},$ and $2$-cocycle $\sigma$.
We suppose that $G^{0} \subseteq G' \subseteq G$ is a 
subgroupoid (using the same 
algebraic operations) and is open in $G$. This means
 that the unit space of $G'$
coincides with that of $G$.
The following is an easy result and we omit the proof. 

\begin{thm}
\label{6:10}
Let $G$ be a locally compact, Hausdorff topological groupoid
with Haar system, $\nu^{u}, u \in G^{0}$ and 
suppose
$G^{0} \subseteq G' \subseteq G$ is an open  subgroupoid. 
\begin{enumerate}
\item 
The system of measures, 
$\nu^{u}\vert_{G^{u} \cap G'}, u \in G^{0}$, is a Haar system 
for $G'$.
\item 
The inclusion $C_{c}(G') \subseteq C_{c}(G)$ obtained by extending
functions to be zero on $G - G'$ extends to an 
inclusion of $C^{*}$-algebras, 
$C_{r}^{*}(G', \sigma) \subseteq C_{r}^{*}(G, \sigma)$.
\end{enumerate}
\end{thm}

Let us define a notational convention: 
if $A \subseteq G$ is any subset of a groupoid $G$, 
$A^{2} = A \times A \cap G^{2}$ and $A^{n}, n \geq 3$ 
is defined in an analogous way.

We state the  following result for convenience.
 The proof is trivial and we omit it.

\begin{lemma}
\label{6:20}
For an open subgroupoid $G^{0} \subseteq G' \subseteq G$, we have 
\begin{enumerate}
\item 
 $\Delta = G - G'$ is closed in $G$,
 \item   $ \Delta = \Delta^{-1}$,
 \item $\Delta G', G'\Delta \subseteq \Delta$,
\item 
$\Delta^{2} = \left( \Delta \times \Delta \right) \cap G^{2}$
 is closed in $G^2$, 
\item 
$\Delta^{2} \cap \mu^{-1}(G')$
is open in $\Delta^{2}$. 
\end{enumerate}
\end{lemma}

We now want to construct a new pair of groupoids, $H' \subseteq H$. 
If $u$ is any unit
of $G$, the set $s(G^{u})$ is $G$-invariant and the restriction
of $G$ to this set is a transitive groupoid (see 1.1 of \cite{Ren:LNM}).
Of course, this restriction  may be disastrous, topologically.
Our new groupoids will be the restrictions of
$G$ and $G'$ to all sets $s(G^{u})$ where they differ. In other words, 
all $s(G^{u})$ where $u $ is in $r(\Delta)$. At least intuitively, 
any relative theory of $G' \subseteq G$ should be the same as
that for $H' \subseteq H$. The first difficulty lies in the issue 
of putting  suitable topologies on $H'$ and $H$. In fact, there
is a canonical way to do this, but we will need certain 
technical conditions to proceed further.

To facilitate this, let us name some of the topologies
involved. We let $\mathcal{S}$ be the topology on $G$ and 
$\mathcal{S}_{A}$ be the relative topology on any subset 
$A $ of $G$.

\begin{defn}
\label{6:30}
We say the inclusion  $G^{0} \subseteq G' \subseteq G$ is
 \emph{regular} if 
the map $r: \Delta \rightarrow r(\Delta)$ is open, when the image is
given the quotient topology. We let $\mathcal{R}$ be the quotient
topology on $r(\Delta)$.
\end{defn}

\begin{lemma}
\label{6:35}
Let $G^{0} \subseteq G' \subseteq G$ be a open
subgroupoid.   If the inclusion is regular, then 
 $\Delta^{2} \cap \mu^{-1}(G')$
is closed in $\Delta^{2}$.
\end{lemma}

\begin{proof}
Suppose that $(x_{i}, y_{i}), i \geq 1$ is a sequence in 
$\Delta^{2} \cap \mu^{-1}(G')$ converging to
$(x,y)$, which is in $\Delta^{2}$, as $\Delta$ is closed.
Suppose that $xy$ is in $\Delta$. We may choose a decreasing 
sequence
of open subsets in $G$, $U_{n}, n \geq 1$, which intersect
to $xy$.
Then $U_{n} \cap \Delta$ is an open set in $\Delta$ and, as 
our inclusion is regular, $r(U_{n} \cap \Delta)$ is open in 
$\mathcal{R}$. As $x_{i} \in \Delta, i \geq 1$ is converging to $x$, 
$r(x_{i}), i \geq 1$ in the topology $\mathcal{R}$. So 
for each $n \geq 1$, we  may find $x_{i_{n}}$ such that 
$r(x_{i_{n}})$ is in $r(U_{n} \cap \Delta)$. So we may find 
$z_{n}$ in $U_{n} \cap \Delta$ with $r(z_{n}) = r(x_{i_{n}})$.
The sequence $z_{n}^{-1} x_{i_{n}} y_{i_{n}}$ is in $\Delta G' = \Delta$
and converges to $(xy)^{-1}xy$ which is in $G^{0} \subseteq G'$. 
This contradicts $G'$ being open.
\end{proof}

\begin{defn}
\label{6:38}
Assume that $G^{0} \subseteq G' \subseteq G$ is a open
subgroupoid and that the inclusion is regular. Define 
\[
H' = \mu(\Delta^{2} \cap \mu^{-1}(G')) \subseteq G'.
\]
We endow $H'$ with the quotient topology from the map
\[
\mu: \Delta^{2} \cap  \mu^{-1}(G') \rightarrow H'
\]
which we denote by $\mathcal{T}'$.
\end{defn}

We observe the following for future purposes.
 
\begin{lemma}
\label{6:39}
With $H'$ as defined in Definition \ref{6:38}, we 
have
\begin{enumerate}
\item $x$ in $G'$ is in $H'$ if and only if 
$r(x)$ is in $r(\Delta)$,
\item $H'$ is a subgroupoid of $G'$, 
\item 
the unit space of $H'$ is $r(\Delta)$.
\end{enumerate}
 Moreover, the topology
on the unit space given in Definition \ref{6:38}
agrees with the quotient topology of 
$r: \Delta \rightarrow r(\Delta)$.
That is, we have $\mathcal{T}'_{(H')^{0}} = \mathcal{R}$. 
\end{lemma}

\begin{proof}
For the first part, it is clear that
\[
r(H') = r \circ \mu( \Delta^{2} \cap \mu^{-1}(G') ) 
\subseteq r \circ \mu(\Delta^{2}) \subseteq r(\Delta).
\]
On the other hand, if $x$ is in $\Delta$,, then $(x,x^{-1})$ is in 
$\Delta^{2}$ and also $\mu^{-1}(G')$ and 
$r(x) = \mu(x,x^{-1})$.
The second and third statements follow easily from the first.

The final part relies on the
 following basic topological fact: if $X$ is a topological space, 
 $f:X \rightarrow Y$ is a surjection and $Y$ is given the quotient topology, 
 then for any $Z \subseteq Y$, the relative topology from $Y$ agrees
 with the quotient topology from $f: f^{-1}(Z) \rightarrow Z$.
 In our case, we use $X = \Delta ^{2} \cap \mu^{-1}(G')$,  $f= \mu$,
 $H' = Y$ and $Z = (H')^{0}$ along with the observation that the map 
 sending $x$ in $\Delta$ to $(x, x^{-1})$ is a homeomorphism from $\Delta$ to 
 $\mu^{-1}((H')^{0}$.
\end{proof}

\begin{thm}
\label{6:40}
Assume that $G^{0} \subseteq G' \subseteq G$ is a open
subgroupoid and that the inclusion is regular. 
Let $H', \mathcal{T}'$ be as in Definition \ref{6:38}.
\begin{enumerate}
\item The topology $\mathcal{T}'$ on $H'$ is finer
than $\mathcal{S}_{G'}$, the relative topology from $G$.
\item 
$H'$ is topological groupoid.
\item  
$H'$ is locally compact.
\item  
$H'$ is Hausdorff.
\item For each $u$ in $(H')^{0} = H' \cap G^{0}$, we have 
\[
G^{u} \cap H' = G^{u} \cap G', \hspace{.5cm} G_{u} \cap H' = G_{u} \cap G'
\]
and the relative topologies from  $\mathcal{T}'$and $\mathcal{S}$  are the
same on each of these
sets.
\item 
The system of measures,
 $\nu^{u}\vert_{G^{u} \cap G'}, u \in (H')^{0}$, is a Haar system 
for $H'$.
\end{enumerate}
\end{thm}

\begin{proof}
The first statement follows immediately from the fact
that the map $\mu$ is continuous and the definition
of the quotient topology.

Let us consider the set $\Delta^{4} \cap \mu^{-1}(G') \times \mu^{-1}(G')$
and the map $(\mu^{2} \times id)(w,x,y,z) = (wxy,z)$ defined on this set. 
First, we observe that $wx$ is in $G'$ and so $wxy$ is again in $\Delta$.
We also note that $wxyz= (wx)(yz)$ is in $G'$ and so 
we have 
\[
\mu^{2} \times id : \Delta^{4} \cap \mu^{-1}(G') \times \mu^{-1}(G')
\rightarrow \Delta^{2} \cap \mu^{-1}(G').
\]
Moreover, we have a commutative diagram:

\xymatrix{ \Delta^{4} \cap \mu^{-1}(G') \times \mu^{-1}(G')
 \ar^{\mu \times \mu}[d] \ar_{\mu^{2} \times id}[r]  & 
  \Delta^{2} \cap \mu^{-1}(G') \ar^{\mu }[d] \\
  (H' )^{2} \ar^{\mu_{H'}}[r] &  H'}

\noindent To verify $\mu_{H'}$ is continuous, we must take an open 
set $U$ in $H'$ and see that $\mu_{H'}^{-1}(U)$ is open in $(H')^{2}$. 
To see that, we must see
that $(\mu \times \mu)^{-1}(\mu_{H'})^{-1}(U))$ is open. 
From the commutativity of the diagram, we have 
\[
(\mu \times \mu)^{-1}(\mu_{H'})^{-1}(U)) =
 (\mu^{2} \times id)^{-1}(\mu^{-1}(U)).
\]
The set $\mu^{-1}(U)$ is open due to the definition of the 
topology  on $H'$ and  $(\mu^{2} \times id)^{-1}(\mu^{-1}(U))$
 is open because $\mu^{2} \times id$ is 
  continuous. The continuity of the inverse
 in the new topology is obvious
from the fact that $\mu(gh)^{-1} = \mu(h^{-1}, g^{-1})$.

Let us prove  $H'$ is locally compact.
Let $h$ be in $H'$, so it is in $G'$ we may write
$h=gg'$ with $g, g'$ in $\Delta$.
As $\mu$ is open and $G'$ is open, we may find 
open sets $g \subseteq U$, $g' \subseteq U'$ , 
each with compact closure,
such that $\mu( U \times U' \cap G^{2}) $ is 
an open set in $G'$. It follows fairly easily that
$\mu( U \times U' \cap \Delta^{2}) $ is an open set
in $H'$. Let us verify its closure is compact. Let
$h_{n}, n \geq 1$ be any sequence. It follows that, for 
every $n$,
we can find $g_{n}$ in $U \cap \Delta $ and $g_{n}'$ in 
$U' \cap \Delta$ with $g_{n} g'_{n} = h_{n}$. 
As $U$ and $u'$ have compact closures, we may pass to a 
subsequence such that $(g_{n_{i}}, g'_{n_{i}})$ is converging.
From Lemma \ref{6:35}, the limit
of this subsequence must also lie in $\mu^{-1}(G') \cap \Delta^{2}$.
It follows then that $h_{n}$ has a subsequence converging in $H'$.

The topology on $H'$ is Hausdorff since it
 is finer than the usual topology, which is Hausdorff.

 For the fifth part, the containment 
 $G^{u} \cap H' \subseteq G^{u} \cap G'$ 
 is obvious since $H' \subseteq G'$. For the reverse
  containment, let
 $g$ be in $G^{u} \cap G'$. The fact that $u$ is 
 in $H'$ means that
 $u = hh^{-1}$, for some $ h$ in $\Delta$.
  We have 
 $g = ug = (hkh^{-1})g = h(h^{-1}g)$, which is in $H'$ since $h^{-1}g$ 
 is in $\Delta G' \subseteq \Delta$.
 
 As we know that the topology for $H'$ is finer
  than that for $G'$, to see they
 are equal it suffices for us to take a sequence
 $g_{n}$ in $(G')^{u}$ converging to
 $g$ in the topology of $G'$ and show it also
  converges in the topology for $H'$.
 We simply write $g_{n} = ug_{n} = h (h^{-1} g_{n})$, 
 with $h$ as above. 
 It suffices now to observe that $h,  h^{-1}g_{n}$ 
 and $ h^{-1}g$
  are in $\Delta$, and $(h,  h^{-1}g_{n})$
   converges to $(h, h^{-1}g)$ in 
  $G^{2}$. This implies their images under
   $\mu$ converge in the quotient topology
  as desired.

   For the last part, by using Theorem \ref{5:30}, it suffices to 
  prove that $r_{H'}:H' \rightarrow (H')^{0}$ is open. 
 Let $x$ be in $H'$ and $U$ be set in $\mathcal{T}'$
  containing $x$.
  This means $x$ is in $G'$ while
 $x = yz$, with $y, z$ is $\Delta$. It also means that 
 $\mu^{-1}(U) \cap \Delta^2 \cap \mu^{-1}(G')= V \cap \Delta^2 
 \cap \mu^{-1}(G')$, for 
 some open set $V$ in $G^2$. 
 We have $(x,y)$ is in $V \cap \mu^{-1}(G')$
 and so we may find open sets $W,Z$ in $G$ with $s(W) = r(Z)$ 
 such that 
 $W \times Z \cap G^2 \subseteq V \cap \mu^{-1}(G')$. 
 It follows 
 from regularity and Lemma \ref{6:20} 
 that $r(W \cap \Delta)$ is an open subset
 of $(H')^{0}$. We claim that it is contained in 
 $r(U)$. If $u = r(w)$ for some $w$ in $W \cap \Delta$, then 
 $r(w) = r(wz)$, for some $(w,z)$ in
  $G^2 \cap W \times Z \subseteq V$
 The fact that $w$ is in $\Delta$ while $wz$ is in $G'$ means that 
 $z$ is in $\Delta$ also. Hence $(w,z)$ is in 
 $\Delta^2 \cap V = \mu^{-1}(U) \cap \Delta^2 \cap \mu^{-1}(G')$,
 $wz$ is in $U$ and $u= r(w) = r(wz) \in r(U)$, as desired. 
\end{proof}

\begin{thm}
\label{6:50}
Assume that $G^{0} \subseteq G' \subseteq G$ is a open
subgroupoid and that the inclusion is regular. 
 With $H'$ as in Definition \ref{6:38}, we define 
\[
H = H' \cup \Delta.
\]
We endow $H$ with the disjoint union topology; that is, 
both $H'$ and $\Delta$ are clopen. We denote
 this topology by $\mathcal{T}$, 
so that $\mathcal{T}_{H'} = \mathcal{T}'$ and 
$\mathcal{T}_{\Delta} = \mathcal{S}_{\Delta}$.
\begin{enumerate}
\item The topology  $\mathcal{T}$ on $H$ is finer
than the relative topology from $G$.
\item  
$H$ is topological groupoid.
\item 
$H$ is locally compact.
\item 
$H$ is  Hausdorff
\item For each $u$ in $H^{0} = H \cap G^{0}$, we have 
\[
G^{u} \cap H = G^{u} , \hspace{.5cm} G_{u} \cap H = G_{u} 
\]
and  the relative topologies from $H$ and $G$  are the
same on each of these
sets.
\item 
The the systems of measures,
 $\nu^{u}, u \in H^{0}$, is a Haar system 
for $H$.
\end{enumerate}
\end{thm}

\begin{proof}
The first, third and fourth parts follow easily from their
counterparts in Theorem \ref{6:40}. For the second part, the
continuity of inverses is clear. As $\Delta$ is clopen, it suffices
to check the continuity of products on $H' \times H', H' \times \Delta, 
\Delta \times H'$ and $\Delta \times \Delta$ separately.
The first is done. The second and third both have range in $\Delta$
and the continuity follows from that of the product in 
$G$, together with the fact the 
the topology on $H'$ is finer than that of $G$. For the last case, 
we use the fact that Lemma \ref{6:35} implies that 
\[
\Delta^{2} = \Delta^{2} \cap \mu^{-1}(G') \cup 
( \Delta^{2} - \Delta^{2} \cap \mu^{-1}(G') )
\]
are clopen. On the first set, continuity follows from 
Theorem \ref{6:40}, the second just 
 uses the continuity in $G$.
 
 For the proof of the fifth part, we can write 
 $G^{u} = (G')^{u} \cup ( \Delta \cap G^{u})$ and the result follows from 
 Theorem \ref{6:40}.
 
 The last part follows from the last part of Theorem 
 \ref{6:40} and Lemmas \ref{6:35} and \ref{6:39}.
\end{proof}

Let us remark that the groupoids $(G, \mathcal{S})$ and 
$(H, \mathcal{T})$ satisfy all hypotheses of Theorem \ref{5:50}.

  We turn now to the regular representations of $C^{*}_{r}(G, \sigma)$
  and $C^{*}_{r}(H, \sigma)$ on $L^{2}(G, \nu)$.
  We let $p$ be the projection operator on 
  $L^{2}(G, \nu)$ whose range is $L^{2}(G', \nu)$. That is, for 
  $\xi$ in $L^{2}(G, \nu)$ and $g$ in $G$, $p\xi(g) = \xi(g)$ for 
  $g$ in $G'$ and  $p\xi(g) = 0$  for $g$ in $\Delta$.

We let $E$ denote the $C^{*}$-algebra  
$\mathcal{B}(L^{2}(G, \nu))$
 and  define $\delta: A + B \rightarrow E$ by 
\[
\delta(x) = i[x, p] = i(xp - px),
\]
for $x$ in $A + B$.
It is clear that $\delta$ is a completely contractive, $*$-derivation.

We let $\chi_{\Delta}$ be the function on $G$ which is $1$ on $\Delta$
and $0$ on $G'$. It is neither continuous, nor compactly
supported, but will be useful, as follows.

\begin{lemma}
\label{6:120}
\begin{enumerate}
\item 
For any function $a$ in $C_{c}(H)$, the function $\chi_{\Delta}a$ (meaning the 
pointwise product) is in $C_{c}(H)$ and $ (1 - \chi_{\Delta})a$ is in 
$C_{c}(H')$.    
\item 
For any function $a$ in $C_{c}(H)$ , we have 
\begin{eqnarray*}
pa(1-p) & =  & p \left(  \chi_{\Delta}a \right)  \\
(1-p)ap & =  & \left(  \chi_{\Delta}a \right)p   \\
\delta(a) & = & \delta(\chi_{\Delta}a) \\
pap & =  & p\left( (1 -\chi_{\Delta})a \right) \\
pap & =  & \left( (1 -\chi_{\Delta})a \right)p .
\end{eqnarray*}
\item 
For any function $b$ in $C_{c}(G)$, 
   the function $\chi_{\Delta}b$ (meaning the 
pointwise product) is in $C_{c}(H)$.   
\item 
For any function $b$ in $C_{c}(G)$,  
  we have 
\begin{eqnarray*}
pb(1-p) & =  & p \left(  \chi_{\Delta}b \right)  \\
(1-p)bp & =  & \left(  \chi_{\Delta}b \right)p  \\
\delta(b) & = & \delta(\chi_{\Delta}b).
\end{eqnarray*} 
\end{enumerate}
\end{lemma}

\begin{proof}
The first part follows from the fact that 
$\Delta$ is both closed and open in $H$ and $H'$ is the complement of
$\Delta $ in $H$.

For the first equation of the second part, let 
$\xi$ be a function in $L^{2}(G, \nu)$.  From the definition
of 
$p$,  $(pa(1-p)\xi)(g) = \left(p (\chi_{\Delta})a)  \xi\right)(g) =0$, 
 if $g$ is in $\Delta$. Now, let 
$g$ be in $G'$ and
we compute
\begin{eqnarray*}
(pa(1-p)\xi)(g) & =  & (ap\xi)(g)  \\
 &  =  &  \int_{h \in G^{r(g)}} a(gh^{-1}) 
 ((1-p) \xi)(h)\sigma(gh^{-1}, h)
    d\nu_{s(g)}(h).
\end{eqnarray*}
Now $((1-p)\xi)(h)$ is non-zero 
only if $h$ is in $\Delta$. As $g$ is 
in $G'$, this means that $gh^{-1}$ must be in $\Delta$ as well.
For such values of $gh^{-1}$, $a(gh^{-1}) = \left( \chi_{\Delta}a\right)(gh^{-1}) $ and 
we have 
\begin{eqnarray*}
(pa(1-p)\xi)(g) & = & 
 \int_{h \in G_{s(g)}} a(gh^{-1}) ((1-p) \xi)(h) \\
     &   & \hspace{2cm}  \sigma(gh^{-1}, h)
  d\nu_{s(g)}(h) \\
 & = & 
 \int_{h \in G_{s(g)} \cap \Delta} a(gh^{-1}) ((1-p) \xi)(h) \\
   &  &   \hspace{2cm}  \sigma(gh^{-1}, h)
  d\nu_{s(g)}(h) \\
   &  =  & \int_{h \in G_{s(g)} \cap \Delta }
    (\chi_{\Delta} a)(gh^{-1}) \xi
   (h)\sigma(gh^{-1},h) d\nu_{s(g)}(h) \\
    &  =  & \int_{h \in G_{s(g)} }
    (\chi_{\Delta} a)(gh^{-1}) \xi
   (h) \sigma(gh^{-1}, h) d\nu_{s(g)}(h) \\
     &  =  &  \left( (\chi_{\Delta} a) \xi \right)(g).
     \end{eqnarray*}

The second equation is obtained by taking adjoints of both 
sides of the first
(with $a^{*}$ replacing $a$). For the third equation, we compute, 
using the first two equations, 
\begin{eqnarray*}
\delta( \chi_{\Delta}a ) & = &i\left(  \chi_{\Delta}a  p -  p \chi_{\Delta}a
\right) \\
  & =  & i \left((1-p)ap  - pa(1-p) \right) \\
    & = &i \left( ap-pa \right) \\
  & = & \delta(a).
  \end{eqnarray*}
  
  For the fifth equation, it follows from the first 
  that 
  \[
  pap = pa  - p \left(  \chi_{\Delta}a \right) 
  = p(\left( (1- \chi_{\Delta})a \right).
  \]
  as desired. The sixth is done similarly.
  
The third part follows from the fact
that $\Delta$ has the same topology in $H$ as it does in $G$, and 
that $\Delta$ is a clopen subset of $H$. So the function 
$\chi_{\Delta}b$ is a continuous function of compact support on $H$.

The proofs of the last part are exactly the same as the first 
three parts of the second part 
and we omit the details.
\end{proof}

  \begin{lemma}
  \label{6:130}
  If $b$ is in $C_{c}(G, \sigma)$, then 
  $\delta(b) \in \delta(C_{c}(H, \sigma))$.
  \end{lemma}
  
\begin{proof}
This follows at once from part 3 and 
the last equation in part 4 of Lemma \ref{6:120}.
\end{proof}

   Finally, we need a reasonable extension of the notion
of finite index from groups to groupoids to obtain condition C1
of Theorem \ref{3:70}.

\begin{defn}
\label{6:140}
Let $G^{0} \subseteq G' \subseteq G$ be an open subgroupoid.
We say it has \emph{finite index} if there is a 
constant $K \geq 1$ such that, for any $u$ in $G^{0}$, 
there is a finite subset $F \subseteq G_{u}$ with 
$\# F \leq K$ 
and $G_{u} =  G'F$.
\end{defn}

We remark that if this holds for $G$, then it also holds for $H$
as it is a purely algebraic condition.

 \begin{lemma}
 \label{6:145}
 Let $H^{0} \subseteq H' \subseteq H$ be an open subgroupoid
 of finite index.  If $K$ is as in the definition and $p$ 
 is the projection
 of $L^{2}(H, \nu)$ onto $L^{2}(H', \nu)$, 
 \ref{6:140}, then for every $a$ in $C_{r}^{*}(H, \sigma)$ then 
 $\Vert a \Vert \leq K \Vert pa \Vert$.
 \end{lemma}
 
 \begin{proof}
 Let us fix a unit $u$ in $H^{0}$ and restrict our attention to
 $L^{2}(G_{u}, \nu_{u})$. Let $F \subseteq H^{u}$ be a 
 subset of at most $K$ elements such 
 that $F H' = H^{u}$. We may assume that this set is minimal, so that
 $f_{1} f_{2}^{-1}$ is not in $H'$, for $f_{1} \neq f_{2}$ in $F$. 
 For each $f$ in $F$, the map sending $h$ in $H_{u}$ to $hf^{-1}$
 in $H_{r(f)}$ induces a
 unitary operator $ U_{f}: L^{2}(H_{r(f)}, \nu_{r(f)}) 
 \rightarrow L^{2}(H_{u}, \nu_{u})$. This unitary operator 
 conjugates the part of the left regular
  representation of $C_{r}^{*}(H, \sigma)$ on 
  $L^{2}(H_{u}, \nu_{u})$
 onto that on $L^{2}(H_{r(f)}, \nu_{r(f)})$.
 For convenience, we denote these two Hilbert spaces by
 $\mathcal{H}_{u}$ and  $\mathcal{H}_{r(f)}$, respectively. 
   Define $p_{f} = U_{f} p U_{f}^{*}$.
 It is an easy computation to check that the condition that
 $ H'F = H_{u}$ implies that 
 $\sum_{f \in F} p_{f}$ is the identity operator on 
 $L^{2}(H_{u}, \nu_{u})$. 
  Hence we have 
 \begin{eqnarray*}
 \Vert a|_{\mathcal{H}_{u}}\Vert & = &  \Vert \sum_{f \in F} 
 p_{f}a|_{\mathcal{H}_{u}} \Vert  \\
  &  \leq  &  \sum_{f \in F} \Vert p_{f}a|_{\mathcal{H}_{u}} \Vert \\
  & = & \sum_{f \in F} \Vert U_{f} p U_{f}^{*} a|_{\mathcal{H}_{u}} \Vert \\
   & = & \sum_{f \in F} \Vert  p U_{f}^{*} a U_{f}|_{\mathcal{H}_{s(f)}} \Vert \\
   & = & \sum_{f \in F} \Vert  p a|_{\mathcal{H}_{s(f)}} \Vert \\
   & \leq  &  K \Vert p a \Vert.
 \end{eqnarray*} 
 Taking the supremum over all $u$ completes the proof.
 \end{proof}

   \begin{cor}
  \label{6:150}
 If $G^{0} \subseteq G' \subseteq G$ is an open subgroupoid
 and the inclusion is  regular and finite index, 
  then
  the $*$-algebra $C_{c}(H)$ satisfies 
  condition C1 of Theorem \ref{3:70}. 
  \end{cor} 
  
  \begin{proof}
  Let $a$ be in $C_{c}(H)$ and define $a' = (1-\chi_{\Delta})a$
  which is in $C_{c}(H')$ by Lemma \ref{6:120}. 
We have 
\[
\delta(a) = i(ap-pa) = i\left( (1-p)ap - pa(1-p) \right).
\]
The first of these terms maps $pL^{2}(G, \nu)$ to  $(1-p)L^{2}(G, \nu)$
and the second the other way. It follows then that 
\begin{eqnarray*}
K \Vert \delta(a) \Vert & = & 
K \max\{ \Vert (1-p)ap \Vert, \Vert  pa(1-p)\Vert \} \\
    & = &  \max\{ K \Vert (\chi_{\Delta}a)p \Vert,
     K \Vert  p(\chi_{\Delta}a) \Vert \} \\
  & \geq  &  \Vert \chi_{\Delta}a \Vert  \\
   & =  &  \Vert a - a' \Vert. 
   \end{eqnarray*}
   \end{proof}

    \begin{lemma}
  \label{6:170}
  If $G^{0} \subseteq G' \subseteq G$ is an open subgroupoid
 and the inclusion is  regular, then
  the $*$-algebra $C_{c}(H)$ satisfies
   condition C2 of Theorem \ref{3:70}.
  \end{lemma}

  \begin{proof}
  Let $a_{1}, \ldots, a_{M}$  be in $C_{c}(H)$. 
  Define $a= (a_{1}, \ldots, a_{M})$ 
   Let $U = \{ x \in H \mid a(x) \neq 0 \}$ so that 
   $\overline{U}$ is compact in $H$. It follows that
   $r(\overline{U}) \cup s(\overline{U})$ is compact in $H^{0}$. 
We can find $e: H^{0} \rightarrow [0,1]$ which is continuous, identically
one on    $r(\overline{U}) \cup s(\overline{U})$ and has support contained 
in the compact set $K \subseteq H^{0}$.

   We apply Lemma \ref{5:20} with $K, a$ as above and $Y= \Delta$ to each 
  to obtain $b = (b_{1}, \ldots, b_{M})$.
   
   Then we have 
   \[
   \delta(b_{m}) = \chi_{\delta}b_{m} = \chi_{\delta}a_{m} = \delta(a_{m}), 
   \]
   from parts 2 and 4 of Lemma \ref{6:120}, for $1 \leq m \leq M$.

   We claim that $a_{m}= ea_{m}$, for all $ 1 \leq m \leq M$.
   For any $x$ in $H$, we have $(ea_{m})(x) = e(r(x))a_{m}(x)$. If $x$ is 
   not in $U$, then $a_{m}(x) = 0 = e(r(x))a_{m}(x)$. On the other hand, if
   $x$ is in $U$, then $r(x) $ is in $r(supp(a)) $ and so 
   $e(r(x))=1$. 
   
   Finally, we claim that $ea_{m}(x) = eb_{m}(x)$, for all $x$ in $H$. 
   If $r(x)$ is in $K$, this follows from the fact that 
   $a_{m}(x) = b_{m}(x)$ for such $x$. On the other hand, 
   if $r(x) $ is not in $K$, then $e_{m}(x) = 0$.
   
   The proof that $a_{m} = a_{m}e = b_{m}e$ is done in a similar way.  
  \end{proof}
  
  What remains at this point is to verify that $ C^{*}_{r}(H, \sigma) \cap \ker(\delta) = C^{*}_{r}(H', \sigma)$ and 
   $ C^{*}_{r}(G, \sigma) \cap \ker(\delta)  =   C^{*}_{r}(G', \sigma) $.
   The first is relatively simple.

  \begin{thm}
  \label{6:180}
  If $G^{0} \subseteq G' \subseteq G$ is an open subgroupoid
 and the inclusion is  regular and finite index, 
  then 
\[
   C^{*}_{r}(H, \sigma) \cap \ker(\delta)  =   C^{*}_{r}(H', \sigma). 
\]
  \end{thm}
  
  \begin{proof}
  The containment $ C^{*}_{r}(H, \sigma) \cap \ker(\delta)
   \supseteq C^{*}_{r}(H', \sigma)$
  follows at once from part 2 of  Lemma \ref{6:120} and the continuity of
  $\delta$. For the reverse inclusion, let $a$ be any element 
  of $ C^{*}_{r}(H, \sigma) \cap \ker(\delta)$. 
   We may
   select a sequence 
  $a_{n}, n \geq 1$ in $C_{c}(H)$.
  As $\delta(a) = 0$, 
  we know $\delta(a_{n}), n \geq 1$ converges to $0$. By Corollary 
  \ref{6:150}, for each $n$, we may find $a_{n}'$ in $C_{c}(H')$
  such that $\Vert a_{n} - a_{n}'\Vert \leq K \Vert \delta(a_{n}) \Vert$, 
  so $a_{n}'$ also converges to $a$.
  \end{proof}
  
  The second equality is more subtle and we prove it only under 
  some additional hypotheses. It would follow if we knew that 
  $\mathcal{A} = C_{c}(G)$ satisfied condition C1 of 
  \ref{3:70}, using exactly the same argument as above.
  
 \begin{lemma}
 \label{6:148}
   Assume that $G$ is a locally compact, Hausdorff 
   groupoid with Haar system
  $\nu^{u}, u \in G^{0}$, $\sigma$ is a $2$-cocycle on $G$
   and that $G^{0} \subseteq G' \subset G$ is 
  a regular, open subgroupoid of finite index with 
  $K$  as in the definition.
  If $a$ is in $C_{c}(\Delta) \subseteq C_{c}(H)$, then
  \[
  \Vert a \Vert \leq K \Vert \delta(a) \Vert.
  \]
 \end{lemma}

 \begin{proof}
 It follows from part 2 of Lemma \ref{6:120} that 
 $pa = p (\chi_{\Delta }a) = p a (1-p)$ and that 
 $ap = (\chi_{\Delta} a) p = (1-p)ap$. We have 
 \[
 \delta(a) = iap - ipa = (1-p)iap - pia(1-p).
 \]
 Thus $\delta(a)$ is the sum of two operators, one 
 which is zero on $(1-p)\mathcal{H}$ and range in 
 $(1-p)\mathcal{H}$ and  the other which is zero on 
 $p\mathcal{H}$ and range in $p\mathcal{H}$. It follows that 
 \begin{eqnarray*}
 \Vert \delta(a) \Vert & = & 
 \max\{ \Vert (1-p)iap \Vert, \Vert pia(1-p) \Vert \} \\
     & = & \max\{ \Vert ap \Vert, \Vert pa \Vert \} \\
     & = & \max\{ \Vert pa^{*} \Vert, \Vert pa \Vert \} \\
      &  \leq & \max\{ K\Vert a^{*} \Vert,K \Vert a \Vert \} \\
        &  =  &  K \Vert a \Vert.
        \end{eqnarray*}
 \end{proof}
 
    \begin{thm}
  \label{6:240}
   Assume that $G$ is a locally compact, Hausdorff 
  \newline 
   groupoid with Haar system
  $\nu^{u}, u \in G^{0}$, $\sigma$ is a $2$-cocycle on $G$
   and that $G^{0} \subseteq G' \subset G$ is 
  a regular, open subgroupoid of finite index.
   If the closed set $\Delta = G-G'$ has the $C$-extension
   property of Definition \ref{5:60} for some $C \geq 1$, then
\[
   C^{*}_{r}(G, \sigma) \cap \ker(\delta)  =   C^{*}_{r}(G', \sigma). 
\]
  \end{thm}
  
    \begin{proof}
  The containment 
  $ C^{*}_{r}(G, \sigma) \cap \ker(\delta) \supseteq C^{*}_{r}(G', \sigma)$
  follows at once from part 4  of  Lemma \ref{6:120} and the continuity of
  $\delta$. For the reverse inclusion, let $b$ be any element 
  of $ C^{*}_{r}(G, \sigma) \cap \ker(\delta)$. 
   We may
   select a sequence 
  $b_{n}, n \geq 1$ in $C_{c}(G)$. 
  As $\delta(b) = 0$, 
  we know $\delta(b_{n}), n \geq 1$ converges to $0$. From part 2 of Lemma
  \ref{6:130} above, $\chi_{\Delta}b_{n}$ lies in 
  $C_{c}(\Delta) \subseteq C_{c}(H)$ and, from part 4 of the same Lemma, 
 $\delta(\chi_{\Delta}b_{n}) = \delta(b_{n})$.
  As $\Delta$ satisfies the $C$-extension property, 
   for every $n$, we may find $c_{n}$ in 
  $C_{c}(G)$ with 
  \[
  c_{n}|_{\Delta } = \chi_{\Delta}b_{n}|_{\Delta} = b_{n}|_{\Delta}
  \]
  and 
  \[
  \Vert c_{n} \Vert \leq C \Vert  \chi_{\Delta}b_{n} \Vert 
  \leq C K \Vert \delta( \chi_{\Delta}b_{n}) \Vert 
  \leq CK \Vert \delta(b_{n}) \Vert.
  \]
  It follows that $\Vert c_{n} \Vert$ tends to zero as $n$ tends to infinity
  so $b_{n} - c_{n}$ also converges to $b$. These functions 
  are in $C_{c}(G)$ and vanish on $\Delta$, so they can be approximated by 
  functions in $C_{c}(G')$. 
  \end{proof}
  
  We finish this section with  a specific construction which generalizes
  the \emph{orbit splitting} groupoids, first introduced in
  \cite{Put:ZCan}. The group of integers still plays
  an important role. This has seen a number of applications.
  
  Suppose that the locally compact Hausdorf space $X$ has an action
  by homeomorphisms of the locally compact Hausdorff group $\Gamma$. 
  Also suppose that $Y \subseteq X$ is semi-invariant in the sense
  of Definition \ref{5:105}. We further suppose that $\Gamma_{Y}$
  is a normal subgroup of $\Gamma$ and that the quotient group
  $\Gamma/ \Gamma_{Y}$ is isomorphic to $\Z$. Let us denote the quotient
  map from $\Gamma$ to $\Z$ by $\zeta$.
  
  We notice that
  Theorem \ref{5:110} applies immediately to this situation. But we can go
  further to find an open subgroupoid.
  
  \begin{thm}
  \label{6:250}
  Let $\Gamma, X, Y, \zeta$ be as above and let $G = X \rtimes \Gamma$ 
  be the associated transformation groupoid. Define
  \begin{eqnarray*}
  \Delta & =  & \{ (y\gamma_{1}, \gamma_{1}^{-1}\gamma_{2}),  
  (y \gamma_{2},  \gamma_{2}^{-1}\gamma_{1})\mid 
  y \in Y, \gamma_{1}, \gamma_{2} \in \Gamma,  \\
  &  &  \zeta(\gamma_{1})\leq 0, \zeta(\gamma_{2}) > 0 \}
  \end{eqnarray*}
  Then we have the following.
 \begin{enumerate}
 \item  
  $G' = G - \Delta$ is an open subgroupoid.
  \item  
  The inclusion $G^{0} \subseteq G' \subseteq G$ is regular and finite index.
  \item If $H$ and $H'$ are the groupoids of Definition \ref{6:38} 
  and Theorem \ref{6:50}, then 
  \[
  H \cong (Y \rtimes \Gamma_{Y}) \times (\Z \rtimes \Z)
  \]
  (the latter is the co-trivial groupoid)
  while 
   \[
  H' \cong (Y \rtimes \Gamma_{Y}) \times 
  \left( \Z^{-} \rtimes \Z^{-} \cup \Z^{+} \rtimes \Z^{+} \right)
  \]
  where $\Z^{+} = \{ 1, 2, 3, \ldots \}$ and $\Z^{-} = \{ 0, -1, -2, \ldots \}$.
  \item 
  If there is a short exact sequence
   \[
0 \rightarrow C_{0}(X - Y) \rtimes_{r} \Gamma_{Y} 
  \rightarrow C_{0}(X ) \rtimes_{r} \Gamma_{Y} 
   \rightarrow C_{0}( Y) \rtimes_{r} \Gamma_{Y} 
   \rightarrow 0
  \]
  then the closed set $\Delta$ has the $C$-extension property 
  for any $C > 1$.
  \end{enumerate}
  \end{thm}
  
  \begin{proof}
   Let us choose $a$ in $\Gamma$ with $\zeta(a) = 1$.
   
   First observe that $Y \gamma = Y a^{\zeta(\gamma)}$, for every 
   $\gamma $ in $\Gamma$. 
  It follows that  we may write
  \[
  \Delta =\left( \cup_{ i \leq 0, j   > 0}
   Y a^{i} \times \zeta^{-1}\{ i+ j \} \right) \cup 
 \left(  \cup_{ i > 0, j  \leq 0}
   Y a^{i} \times \zeta^{-1}\{ j-i\}  \right).
   \]
   The sets on the right are pairwise disjoint 
   and each is clopen in $G$. 
   It follows that $\Delta$ is closed so
  $G'$ is open. It is a simple computation to check that
  $G'$ is a subgroupoid.
  
  For the second point, it is easy to check from the equation above
  that
  \[
  r(\Delta) = \cup_{i \in \Z} Ya^{i},
  \]
  and that each of the sets $Ya^{i}$ is clopen in the quotient topology.
  In addition, the restriction of $r$ to a set of the
   form $Ya^{i} \times \gamma$
   is actually a 
  homeomorphism to $Ya^{i}$ and, in particular, $r$ is open so the
  inclusion is regular.
  It also follows that the map sending $ya^{i}$ to $(y,i)$ 
  is a homeomorphism between $r(\Delta)$ and $Y \times \Z$.
  
  We check that the inclusion is finite index.
  If $(x,e)$ is any unit in $G$ and $x \notin r(\Delta)$, then we may
  use $F = \{ (x,e) \}$ in Definition \ref{6:140}. On the other hand, 
  if $(x,e) = (ya^{i}, e)$, for some $y$ in $Y$ and integer $i$, letting 
  $F = \{ (ya^{i}, e), (ya^{-i+1}, a^{2i-1}) \}$ satisfies 
  $G_{u} = G'_{u}F$. (The essential point being that $i \leq 0$ if and only if 
  $-i+1 > 0$.)
  
  As for the descriptions of $H'$ and $H$, recall that $(ya^{i}, \gamma)$
  and $(y'a^{i'},\gamma')$ are composable if and only if  $ya^{i}\gamma= y'a^{i'}$.
  If this occurs then $i + \zeta(\gamma) = i'$ and $y' = a^{i}\gamma a^{-i'}$.
  We know also that if $i \leq 0$, then $i' =i + \zeta(\gamma) > 0$ while 
  if $i > 0$, then $i' =i + \zeta(\gamma) \leq 0$. It follows then that 
  \[
  \Delta^{2} \cap \mu^{-1}(G') = \cup_{(i,\gamma, \gamma')} 
  \cup_{y \in Y} \left( (y a^{i}, \gamma), (ya^{i}\gamma, \gamma') \right)
  \]
  where the union is over $y$ in $Y$ and triples $(i, \gamma, \gamma')$ with 
  either $i \leq 0, i + \zeta(\gamma) > 0, i + \zeta(\gamma) + \zeta(\gamma') \leq 0$
  or $i > 0, i + \zeta(\gamma) \leq 0, i + \zeta(\gamma) + \zeta(\gamma') > 0$.
  If we just take the union over $Y$, these sets are
   pairwise disjoint, clopen and each is homeomorphic to $Y$.
   
 As a set, we can write 
 \[
 H'= \{ (ya^{i}, \gamma) \mid y \in Y, 
  i, i+\zeta(\gamma) \leq 0, \text {or }
 i, i+\zeta(\gamma) > 0 \}
 \]
 and the isomorphism of part 3 sends $(ya^{i}, \gamma) $
  to \newline
 $\left( 
 (y, a^{-i} \gamma a^{i + \zeta(\gamma)}, (i, i + \zeta(\gamma) ) \right) $.
  We omit the topological details. This is extended to $H$ by mapping
  $(y a^{i}, \gamma)$ in $\Delta$ to 
  $\left( 
 (y, a^{-i} \gamma a^{i + \zeta(\gamma)}, (i, i + \zeta(\gamma) ) \right) $.
 
  We turn our attention to the last part. 
  Observe that, for any $n \geq 2$, the sets 
  \[
  \Delta_{n} = \cup_{i=1-n}^{0} \cup_{j=1}^{n}  \left(
  Y a^{i} \times \zeta^{-1}\{ j-i \} \cup Ya^{j} \zeta^{-1}\times \{i-j\} \right)
  \]
  are all in $\Delta$. Moreover, they are increasing with $n$
   and each is clopen (in $\Delta$ and hence also in $H$).
   Therefore, it suffices for us to consider a continuous function $f$
   with compact support in $\Delta_{n}$, for some fixed $n$.
   Let us denote $Y_{n} = \cup_{i=1-n}^{n} Y a^{i}$, which we identify
   with $r(\Delta_{n}) = s(\Delta_{n})$.
   
   In fact, $\Delta_{n}$ is a clopen subset of the groupoid
   $G_{Y_{n}}^{Y_{n}} = H_{Y_{n}}^{Y_{n}}$ and by 
   simply restricting the function
   we can regard $f$ is being in $C_{c}( G_{Y_{n}}^{Y_{n}})$ 
   and its norm there
   coincides with its norm in $C_{r}^{*}(H)$.
   
   As the sets $Ya^{i}, i \in \Z$ are pairwise disjoint, we may 
   choose an open set $Y \subseteq U \subseteq X$ such that
   the sets $U a^{l}, 1-3n \leq l \leq 3n$ are pairwise
   disjoint also.
   
   The set $U$ will not necessarily be  $\Gamma_{Y}$-invariant. In addition,
   $U \cap U^{a^{i}}$ may be non-empty for some $i \neq 0$. However, we may 
   consider the reductions $(X \times \Gamma)_{U}^{U}$ and 
   $(X \times \Gamma_{Y})_{U}^{U}$. Their associated $C^{*}$-algebras 
   are hereditary subalgebras of $C_{0}(X) \rtimes_{r} \Gamma$ and 
   $  C_{0}(X) \rtimes_{r} \Gamma_{Y}$, respectively.

   We define $V = \cup_{i=1-n}^{n} U a^{i}$ and, identifying $V$ 
   with $V \times \{ e \}$ in $G^{0}$,
   \[
   G(V) = G_{V}^{V} \cap X \times \zeta^{-1}\{ 1-2n, \ldots, 2n-1 \}.
   \]
   This is an open subset of the groupoid $G_{V}^{V}$.
   We claim that it is also a groupoid. Suppose that $(x, \gamma)$ is in 
   $G(V)$. Then for some $ 1-n \leq i \leq n$, $x$ is in $Ua^{i}$ 
   and $1-2n \leq \zeta(\gamma) \leq 2n-1$. In addition 
   $s(x,\gamma) = (x \gamma, e)$ is 
   in $V$. On the other hand, $x \gamma$ is in $Ua^{i + \zeta(\gamma)}$ 
   and $ 1-3n \leq i + \zeta(\gamma) \leq 3n-1$. As these sets
   are pairwise disjoint and only in $V$ if 
   $ 1-2n \leq i + \zeta(\gamma) \leq 2n-1$, we know 
   that this inequality must hold.
   
   The left regular representation of $C^{*}(G)$ also extends
    to a representation
   of the bounded Borel functions on $G^{0}$, as well as
    to a unitary
   representation of the group $\Gamma$ \cite{Ren:LNM}. 
   It is a simple matter to check that 
   \[
   \chi_{Ua^{i}}a^{i}  = a^{i} \chi_{U}
   \]
   and, for any $g$ in $C_{c}(G(V))$,
   $1-n \leq i, j \leq n$, 
   \[
   a^{-j}\chi_{U a^{j}} g \chi_{Ua^{i}}a^{i}
   \]
   is in $C_{c}((X \times \Gamma_{Y})_{U}^{U})$. This defines an isomorphism
   (we will not write it explicitly) between  $C^{*}_{r}(G(V))$ and
   $M_{2n}( C^{*}_{r}((X \times \Gamma_{Y})_{U}^{U}))$
    (the matrix entries are indexed by the 
   set $\{ 1-n, \ldots, n \}$). For simplicity, let us denote 
   $X \times \Gamma_{Y}$ by $L$.

   The set $Y_{n}$ is a closed invariant subset of 
   the unit space of $G(V)$ and we have a 
   quotient map from $C^{*}_{r}(G(V))$ to $C^{*}_{r}(G_{Y_{n}}^{Y_{n}})$. 
   This yields the following commutative diagram
   
  \hspace{3cm}
   \xymatrix{ C^{*}_{r}(G(V))  \ar[r] \ar[d] & 
     C^{*}_{r}(G_{Y_{n}}^{Y_{n}}) \ar[d] \\
    M_{2n}( C^{*}_{r}(G_{U}^{U}) ) \ar[r]  &   M_{2n}( C^{*}_{r}(G_{Y}^{Y}) )}
    
    and the vertical map on the right is also an isomorphism. 
   
   We now invoke our hypothesis to extend this diagram to
   
      \xymatrix{  C^{*}_{r}(G(V))  \ar[r] \ar[d] & 
     C^{*}_{r}(G_{Y_{n}}^{Y_{n}}) \ar[d] \ar[r] & 0 \\
  M_{2n}( C^{*}_{r}(L_{U}^{U}) ) \ar[r]  
  \ar[d]   &   M_{2n}( C^{*}_{r}(L_{Y}^{Y}) )  \ar[d] 
    \ar[r] & 0  \\ 
 M_{2n}(C_{0}(X ) \rtimes_{r} \Gamma_{Y} ) \ar[r]  &
  M_{2n}( C_{0}( Y) \rtimes_{r} \Gamma_{Y})  \ar[r] & 0.
    }

   We now take our function $f$ lying in $C_{c}(H)$, whose support is in
   $\Delta_{n}$ and assume that $f$ is non-zero. 
   We can regard this as an element of $M_{2n}( C_{c}( Y \rtimes \Gamma_{Y}) $
   and we may lift it an element $\tilde{f}$ in
    $M_{2n}( C_{c}( X \rtimes \Gamma_{Y})$. From our hypothesis, we know that 
    \[
    \Vert f \Vert_{r} = \inf\{ \Vert \tilde{f} + g \Vert_{r} \mid g \in 
    C_{c}((X-Y) \times \Gamma_{Y}) \},
    \]
    and so if $C > 1$, we may choose $g$ so that 
    $\Vert \tilde{f} + g   \Vert_{r} \leq C \Vert f \Vert_{r}$.
    It remains to get this function back into $M_{2n}( C^{*}_{r}(L_{U}^{U}) )$. 
    As we noted above, this is a  hereditary subalgebra of 
    $M_{2n}( C_{c}( X \rtimes \Gamma_{Y}))$ and we can multiply on both sides
    by a function bounded between $0$ and $1$, identically
     $1$ in $Y_{n}$ and supported
    in $V$.  
   \end{proof}
  
  \begin{cor}
  \label{6:260}
  With $X, \Gamma, Y, \zeta$ as above, 
  if there is a short exact sequence
    \[
0 \rightarrow C_{0}(X - Y) \rtimes_{r} \Gamma_{Y} 
  \rightarrow C_{0}(X ) \rtimes_{r} \Gamma_{Y} 
   \rightarrow C_{0}( Y) \rtimes_{r} \Gamma_{Y} 
   \rightarrow 0
  \]
  then 
  \[
  K_{*}(C^{*}_{r}(G');C_{0}(X) \rtimes_{r} \Gamma) 
  \cong K_{*}(C_{0}(Y) \rtimes_{r} \Gamma_{Y}).
  \]
  \end{cor}

 \section{Application to factor groupoids}
\label{7}
In this section, we again
have two  groupoids, $G, G'$. We assume that 
each is locally compact and Hausdorff. In addition, we assume that
$G$ has a Haar system, $\lambda_{G}^{u}, u \in G^{0}$ and 
$G'$ has a Haar system, $\lambda_{G'}^{u}, u \in (G')^{0}$. Finally, we assume that
$\sigma$ is a $2$-cocycle on $G$ and that 
$\sigma'$ is a $2$-cocycle on $G'$

We assume that there is a map $\pi: G \rightarrow G'$ 
satisfying the following conditions, which we refer
 to as our standing hypotheses  on $\pi:G \rightarrow G'$:
\begin{enumerate}
\item 
$\pi$ is continuous,
\item 
$\pi$ is proper,
\item 
$\pi$ is a morphism of groupoids,
\item for every $u$ in $G$ and $\pi|_{G^{u}}: G^{u} \rightarrow (G')^{\pi(u)}$
is a homeomorphism and 
\item for every $u$ in $G$ and Borel set $E$ in $G^{u}$, we have 
\[
\lambda^{u}_{G}(E) = \lambda^{\pi(u)}_{G'}(\pi(E)).
\]
\item 
for all $(s,t)$ in $G^{2}$, we have $\sigma(s,t) = \sigma'(\pi(s), \pi(t))$.
\end{enumerate}
We will assume this holds throughout this section.

\begin{ex}
\label{7:10}
As an example, consider the case where $\Gamma$ is a 
locally compact, Hausdorff topological group 
acting continuously on  compact Hausdorff spaces, $X$ and $X'$. If 
$\pi: X \rightarrow X'$ is a continuous, $\Gamma$-invariant surjection then,   
$\pi \times id_{\Gamma}: X \rtimes \Gamma \rightarrow X' \rtimes \Gamma$ 
satisfies the desired conditions.
\end{ex}

The proof of the following result  is straightforward and we omit it. 

\begin{thm}
\label{7:40}
 The map sending $b$ in $C_{c}(G')$ to 
$b \circ \pi$ in $C_{c}(G)$ extends to an injective $*$-homomorphism, also
denoted $\pi$, from $C^{*}_{r}(G',\sigma)$ to  $C^{*}_{r}(G,\sigma)$.
\end{thm}

It will be convenient (though probably not necessary) for us to
assume that $d_{G}$ is a metric on $G$ giving rise to its topology.
By simply replacing $d_{G}$ by the function
\[
\max\{ d_{G}(x,y), d_{G}(x^{-1}, y^{-1}), d_{G}(r(x),r(y)), d_{G}(s(x),s(y))  \}
\]
 we may assume that the map $x \rightarrow x^{-1}$ is an isometry
 and that 
  $r, s$ are contractions.
  
We will let $G(x, r)$ denote the ball centred at $x \in G$ of 
radius $r > 0$. In addition, if $A \subseteq G$ is any subset, we let 
\[
G(A, r)= \{ x \in X \mid \text{ for some } y \in A, d(x,y) < r \}
\]
for $r > 0$.

We will then use the same notation, $d_{G}$, for the Hausdorff 
  metric on the compact subsets
of $G$:
\[
d_{G}(E, F) = \inf \{ \epsilon > 0 \mid 
F \subseteq  G(E, \epsilon), 
 E \subseteq G(F, \epsilon) \}, 
\]
for $E, F \subseteq G$ compact. We will  use the notation
\[
diam_{G}(E) = \sup \{ d_{G}(x,y) \mid x, y \in E \},
\]
for $E \subseteq G$ compact.
We will also assume for convenience that $G'$ is second countable.

The following is an obvious consequence of the third and fourth conditions, 
but we find it convenient to state explicitly.

\begin{lemma}
\label{7:20}
Let $\pi: G \rightarrow G'$ be as above. For every $x$ in $G$, we have 
\[
\# \pi^{-1}\{ x \} = \# \pi^{-1}\{ r(x) \} = \# \pi^{-1}\{ s(x) \}.
\]
\end{lemma}

The following technical result will also be useful later.

\begin{lemma}
\label{7:30}
 Suppose the sequence $x_{k}'$ converges to
$x'$ in $G'$. Then
\[
\lim_{k} diam_{G}(\pi^{-1}\{ x_{k}'\}) = 0,
\]
if and only if
\[
\lim_{k} diam_{G}(\pi^{-1}\{ r(x_{k}') \}) = 0,
\]
\end{lemma}

\begin{proof}
Let $X' $ consist of the sequence $x_{k}'$, along with its limit point, $x'$. 
This is evidently a  compact subset of $G'$. Hence, $r(X')$ is also 
compact in $G'$, while $X = \pi^{-1}(X')$ and $r(X)$ are compact in $G$ as $\pi$ is proper.
 
 As $X$  is compact, $r$ is uniformly continuous on $X$ and the 'only if'
 direction follows at once. Conversely, suppose
 the first condition fails: then there exist subsequences $x_{k}, y_{k}$ in $G$ with
 $\pi(x_{k}) = x'_{k} = \pi(y_{k})$, $\lim_{k} x_{k} = x \neq y =\lim_{k} y_{k}$
 and $\lim_{k} r(x_{k})  = \lim_{k} r(y_{k})$.
Then we have 
\[
r(x) = r( \lim_{k} x_{k} ) = \lim_{k} r(x_{k})  =  \lim_{k} r(y_{k}) = r(\lim_{k} y_{k} ) =r(y).
\]
 On the other hand, we also have
 \[
\pi(x) = \pi( \lim_{k} x_{k} ) = \lim_{k} \pi(x_{k}) = x'  =\lim_{k} \pi(y_{k}) = \pi(\lim_{k} y_{k} ) = \pi(y).
\]
Thus $x, y$ are in $G^{r(x)}$ with $\pi(x)=\pi(y)$. By the fourth condition of
 our standing hypotheses, this means $x=y$, 
a contradiction.
\end{proof}

As a final preliminary topological result, we have the following.

\begin{lemma}
\label{7:45}
 Given $x'$ in $G'$ and $\epsilon > 0$, there
is an open set $x' \in U' \subseteq G'$ such that 
$\pi^{-1}(U') \subseteq G( \pi^{-1}\{ x' \}, \epsilon)$.
\end{lemma}

\begin{proof}
If the conclusion is false, then we may find a 
sequence, $x'_{k}, k \geq 1$,
converging to $x'$ and a sequence $x_{k}, k \geq 1$ with 
$\pi(x_{k})=x'_{k}$, for all $k \geq 1$ such that 
$x_{k}$ is not in $G(\pi^{-1}\{ x' \}, \epsilon)$.
The sequence $x'_{k}, k \geq 1$ along with $x'$ forms
a compact set in $G'$. As $\pi$ is proper, its preimage is 
aslo compact in $G$. So we may find a subsequence $x_{k_{l}}, l \geq $
converging to some $x$ in $G$. As $G(\pi^{-1}\{ x' \}, \epsilon)$ is open, 
$x$ is also not in this set. But by continuity of $\pi$, 
$\pi(x) = \lim_{l} \pi(x_{k_{l}}) = x'$. This is a contradiction.
\end{proof}

We want to focus our attention on the parts of $G$ and $G'$ where they
are actually different; that is, where $\pi$ is not injective.

\begin{defn}
\label{7:50}
Let $\pi: G \rightarrow G'$ satisfy the standing hypotheses. 
We define
\[
H' = \{ x' \in G' \mid \# \pi^{-1}\{ x' \} > 1 \}
\]
and $H = \pi^{-1}(H')$.
\end{defn}

We endow $H'$ with the metric
\[
d_{H'}(x', y') = d_{G}(\pi^{-1}\{ x'\}, \pi^{-1}\{ y'\}), 
\]
for $x', y'$ in $H'$, and $H$ with the metric
\[
d_{H}(x,y) = d_{G}(x, y) + d_{H'}(\pi(x), \pi(y)),
\]
for $x, y$ in $H$. 

To obtain our excision result, we will need a  hypothesis
on our map $\pi$. It implies the continuity of
the fibres of $\pi$, but in a weak sense.

\begin{defn}
\label{7:60}
Let $\pi: G \rightarrow G'$ satisfy the standing hypotheses. 
We say that $\pi$ is \emph{regular} if, for every $x'$ in $H'$  
and 
$\epsilon > 0$, there is an open set $x' \in U' \subseteq G'$ 
such that if $y'$ is in $U'$, then either
\[
d_{G}(\pi^{-1}\{ x'\}, \pi^{-1}\{ y'\}) < \epsilon
\]
or 
\[
diam_{G}(\pi^{-1}\{ y'\}) < \epsilon.
\]
In view of Lemma \ref{7:45}, we may also assume the conclusion
there also holds for $U'$.
We remark that if $\epsilon < 3^{-1} diam_{G}(\pi^{-1}\{ x'\}) $, 
the two conditions are mutually exclusive.
\end{defn}

\begin{prop}
\label{7:65}
The topologies on $H'$ and $H$ from the metrics $d_{H'}$ and $d_{H}$ are finer, 
respectively, than the relative topologies from $G'$ and $G$.
\end{prop}

\begin{proof}
Let $s'$ be in $H'$ and let $U$ be an open set in $G'$ containing it.
Choose any $s$ in $\pi^{-1}\{ s' \}$. As $\pi$ is continuous, there
is $\epsilon > 0$ such that $\pi(G(s,\epsilon) ) \subseteq U$. 
We claim that $H'(s', \epsilon) \subseteq U \cap H'$. If $t'$ is in
$H'(s', \epsilon) $ then it is clearly in $H'$. Moreover, we know
$d_{G}(\pi^{-1}\{ s' \}, \pi^{-1}\{ t' \}) < \epsilon$. It follows
that there is $t$ with $\pi(t) = t'$ and $d_{G}(s,t)  < \epsilon $. 
It follows that $t' = \pi(t)$ is in $U$. Hence the topology
from $d_{H'}$ finer than the relative topology from $G'$.

The fact that, for all $x, y$ in $H$,
 \[
 d_{G}(x,y) \leq d_{H'}(\pi(x), \pi(y)) + d_{G}(x,y) = d_{H}(x,y),
 \] 
 immediately implies the desired conclusion for $H$.
\end{proof}

\begin{thm}
\label{7:70}
Suppose that $\pi: G \rightarrow G'$  is regular. 
 Then $H'$ and $H$, with the metrics $d_{H'}$ and $d_{H}$,
  are locally compact, Hausdorff topological
  groupoids and $\pi: H \rightarrow H'$ is an open,  continuous, proper
  morphism of groupoids.
\end{thm}

\begin{proof}
We begin with $H'$. It follows from
Lemma \ref{7:20} that either $x', r(x')$ and $
 s(x')$ are all in $H'$, 
or none are. It follows that
$H'$ is a subgroupoid of $G'$.

 Next, we show that $H'$ is locally compact. Let $x'$ be in $H'$.
 Let $\epsilon = 2^{-1} diam_{G} \pi^{-1}\{ x' \}$. Observe that 
 if $y'$ is any element of $H'$ within $\epsilon$ of $x'$ then 
 $diam_{G}(\pi^{-1}\{ y \}) > \epsilon$. Select 
 $x' \in U \subseteq G'$ open, as in Definition \ref{7:60}.
  As $G'$ is locally 
 compact, we may assume that $U$ has compact closure.
 Finally, select $\epsilon > \epsilon' >0$ such that 
 $\pi(G(\pi^{-1}\{ x' \}, \epsilon') \subseteq U$. We claim that 
 the $d_{H'}$-ball around $x'$ of radius $\epsilon'$ has compact 
 closure. Let $x_{k}'$ be any sequence in this ball. The fact that
 $d_{H'}(x', x_{k}') < \epsilon'$ means that $\pi^{-1}\{ x_{k}'\}$ is 
 contained in $G(\pi^{-1}\{ x' \}, \epsilon')$ and so $x_{k}'$ is in $U$.
 As $U$ has compact closure, we may pass 
 to a subsequence which converges to $y'$ 
 in $G'$. We claim this sequence also converges to $y'$ in $d_{H'}$. 
 Let $\delta > 0$ be given. Without loss of generality, assume 
 $\delta < \epsilon$. We now apply our regularity 
 hypothesis to $y'$ to find an open set $y' \in W \subseteq G'$ such that,
  for all $z'$ in $W$, either
 \[
d_{G}(\pi^{-1}\{ y'\}, \pi^{-1}\{ z'\}) < \delta
\]
or 
\[
diam_{G}(\pi^{-1}\{ z'\}) < \delta.
\]
As our subsequence of $x_{k}'$ converges to $y'$ in the usual topology, 
we may find $K$ such that $x_{k}'$ is in $W$, for all $k \geq K$.
Since our subsequence is taken from the $d_{H'}$-ball, we know that
$diam_{G}(\pi^{-1}\{ x_{k}'\}) \geq \epsilon > \delta$. This eliminates the second
possibility above and hence, we have $d_{H'}(x_{k}', y') < \delta$, for
$k \geq K$.

Every metric space is Hausdorff.

As $d_{G}$ is preserved under inverses, so is $d_{H'}$. We must now check that
the product on $H'$ is continuous in $d_{H}$.
 Suppose that $(x_{k}', y_{k}')$ is a sequence
converging to $(x',y')$ in $(H')^{2}$. We again let $X'$
 be this sequence, together with its limit
point. As $G^{2}$ is closed in $G \times G$,
 $X = (\pi \times \pi)^{-1}(X') \cap G^{2}$ is 
compact in $G^{2}$. The product map on $X$ is continuous
 and it follows that 
it is continuous on the compact subsets of $X$, equipped
 with the Hausdorff metric. 
The continuity of the product on $H'$ follows from this.

We now turn our attention to $H$. As 
$\pi \circ r = r \circ \pi, \pi \circ r = r \circ \pi$, 
an element
$x$ in $G$ is in $H$ if and only if $r(x)$ is in $H$, 
if and only if $s(x)$ is in $H$.
It follows that $H$ is  a groupoid. 

We observe first that. for any $x, y$ in $G$, we have 
\[
d_{H'}(\pi(x), \pi(y)) \leq d_{H'}(\pi(x), \pi(y)) 
  + d_{G}(x, y) = d_{H}(x, y) 
  \]
  so the map $\pi$ is contractive and hence continuous.

We show that $H$ is locally compact in $d_{H}$. Let $x$ be in $H$. We may 
find $\epsilon > 0$ such that $H'(\pi(x), \epsilon)$ has compact closure.
We claim the same is true of $H(x, \epsilon)$. Let $x_{k}$ be any 
sequence in $H(x, \epsilon)$. It follows from the
 definition of $d_{H}$ that 
 that $\pi(x_{k})$ is in 
 $H'(\pi(x), \epsilon)$. Hence we may pass to a subsequence such that
 $\pi(x_{k})$ is converging to $x'$. We use the same trick again: let $X$
  denote the pre-image of of the subsequence and its limit point under 
  $\pi$, which is compact and contains $x_{k}$. Now we can further
  extract a subsequence where the $x_{k}$'s converge in the usual topology
  of $G$, also. It then follows
  from the definition of $d_{H}$ that this subsequence is also converging
  in $d_{H}$. Again, a metric space is Hausdorff. 
  
  To check that the product in $H$ is continuous in the metric
  $d_{H}$, it suffices to observe  the definition of the
   metric $d_{H}$ and the facts that the product is continuous in $d_{G}$, 
   $\pi$ is continuous and the product in $H'$ is continuous in $d_{H'}$.
   The inverse is isometric in $d_{H}$: this 
   follows from the fact that taking inverses 
   is isometric in both $d_{G}$ and $d_{H'}$.
   
   Let us show that $\pi$ is continuous. Let $x$ be in $H$ and 
   $\epsilon > 0$. Without loss of generality, assume that
   $\epsilon < 3^{-1} diam_{G}(\pi(x))$.
   There is an open set $U'$ in $G'$ containing $\pi(x)$
   such that, for all $x'$ in $U'$, $diam_{G}(x') < \epsilon$ or
   $d_{G}(\pi^{-1}\{ x' \}, \pi^{-1}\{ \pi(x) \}) < \epsilon$. As the 
   topology on $H'$ from the metric $d_{H'} $ is finer than the 
   relative topology from $G'$, we may find $\epsilon/2 > \delta > 0$ 
   such that 
   $H'(\pi(x), \delta) \subseteq U'$. Now suppose that $x'$ is in 
   $H'(\pi(x), \delta)$. It follows that 
    $d_{G}(\pi^{-1}\{ x' \}, \pi^{-1}\{ \pi(x) \})
     < \delta < \epsilon/2$, implying that we may find
    $y$ in $H$ with $\pi(y) = x'$ and $d_{G}(y,x) < \delta$. 
    It follows that $d_{H}(x,y) < \epsilon$ and this completes the proof.
   
   The last thing for us to check is that $\pi$ is
    proper. Let $K '\subseteq H'$ be 
   any subset which is compact in $d_{H'}$. Then it 
   is compact in the topology of
   $H$ as well and hence $K = \pi^{-1}(K')$ is compact in $G$. Now let
   $x_{k}, k \geq 1$ be any sequence in $K$.
    It follows that there is a subsequence
   which is converging in $d_{G}$. There is a
    further subsequence such that 
   $\pi(x_{k})$ is converging in $d_{H'}$. This
    subsequence converges from the 
   definition of $d_{H}$.
\end{proof}

\begin{thm}
\label{7:80}
Suppose that $\pi: G \rightarrow G'$  is regular. 
\begin{enumerate}
\item  The groupoids $G'$ and $H'$ satisfy the 
hypothesis of Theorem \ref{5:50}.
 \item  The groupoids $G$ and $H$ satisfy the 
hypothesis of Theorem \ref{5:50}.
\item For every $u$ in $H^{0}$, the map
$\pi : H^{u} \rightarrow (H')^{\pi(u)}$ is a homeomorphism.
 \end{enumerate}
\end{thm}

\begin{proof}
We begin with $G'$ and $H'$. We have already seen $G'$-invariance 
as a consequence of Lemma \ref{7:10}. The second condition is the conclusion 
of Proposition \ref{7:65}.

Let us now fix $u'$ in $H'$ and verify that the relative 
topology from $G'$ on 
$(G')^{u'} = (H')^{u'}$ is finer than the topology of $d_{H'}$. Fix $s'$ in 
$H'$ with $r(s') = u'$ and let $\epsilon > 0$. We want to find an open set
$U$ in $(G')^{u'}$ such that $U  \subseteq H'(s', \epsilon)$.
It is a fairly easy result in topology that there exists an open set $V$
containing $s'$ such that, for every $t'$ in $V$,
 $\pi^{-1}\{ t' \} \subseteq G(\pi^{-1}\{ s' \}, \epsilon)$.
The collection of sets $G(s, \epsilon/2 ) , s \in \pi^{-1}\{ s' \}$ forms
an open 
cover of $\pi^{-1}\{ s' \}$.
As  $\pi$ is proper, we extract a finite subcover corresponding
to points $s_{1}, \ldots, s_{K}$ in $\pi^{-1}\{ s' \}$. 
As $\pi$ is assumed to 
be a homeomorphism when restricted to each $G^{r(s_{i})}$, 
each of the sets
$\pi(G(s_{k}, \epsilon/2) \cap G^{r(s_{k}})$ is an open 
subset of $(G')^{u'}$.
We define
\[
U = V \cap \left( \cap_{k=1}^{K} \pi(G(s_{k}, \epsilon/2) 
\cap G^{r(s_{k})}) \right),
\] 
which is an open set in $(G')^{u'}$ containing $s'$.  We claim that
$U$ is contained in the $d_{H'}$-ball of radius $\epsilon$ around $s'$.
 Let $t'$ be in $U $. As $U \subseteq V$, we 
 have $\pi^{-1}\{ t' \} \subseteq G(\pi^{-1}\{ s' \}, \epsilon)$. 
 For the other inclusion, for any $s$ in $\pi^{-1}\{ s' \}$, we know that
 $s$ is in $G(s_{k}, \epsilon/2)$, for some $k$. We also know that
 $t'$ is in $\pi(G(s_{k}, \epsilon/2) \cap G^{r(s_{k})})$, so we may find
 $t$ in $G(s_{k}, \epsilon/2)$ with $\pi(t) = t'$. From the triangle 
 inequality, we have $d_{G}(t, s) < \epsilon$. 
 
 Finally, to check that fourth property, 
  we must verify that $r: H' \rightarrow (H')^{0}$
 is open in $d_{H'}$. 
 Let $x'$ be in $H'$ and $\epsilon_{0} > 0$. As $G'$ is
  locally compact, we may find 
 an open set $x' \in U_{0} \subseteq G'$ with compact closure.
  It follows that 
$\pi^{-1}(\overline{U_{0}})$ is also compact and contains
 $\pi^{-1}\{ x' \}$. 
As $r_{G}$ is uniformly continuous on the compact set
 $\pi^{-1}(\overline{U_{0}})$, we may 
find $\epsilon_{0} > \epsilon > 0$ such that 
for any  $  y, z \in \pi^{-1}(\overline{U_{0}})$ 
 with $d_{G}(y, z) < \epsilon$, 
 it follows  that  
$d_{G}(r_{G}(y), r_{G}(z)) < 3^{-1} diam_{G}( \pi^{-1}\{ r_{G'}(x') \})$.
Now, we may use our hypothesis of regularity to find an 
open set $x' \in U \subseteq U_{0}$
such that, for all $y'$ in $U$, we have either 
$d_{G}(\pi^{-1}\{ y' \}, \pi^{-1}\{ x' \}) < \epsilon$
  or $diam_{G}(\pi^{-1}\{ y' \}) < \epsilon$.
  We use the fact that $r_{G'}$ is an open map to find an 
  open set $r_{G'}(x') \in V \cap (G')^{u} \subseteq r_{G'}(U)$. 
  Finally, we choose 
  $ 0 < \delta  < 3^{-1}diam_{G}( \pi^{-1}\{ r_{G'}(x') \}) $ such 
  that $G(\pi^{-1}\{ r_{G'}(x')\}, \delta) \subseteq \pi^{-1}(V)$.
  We claim that 
  \[
  H'(r_{G'}(x'), \delta) \cap (H')^{u} \subseteq r_{G'}(H'(x', \epsilon)) 
  \subseteq r_{G'}(H'(x', \epsilon_{0})).
  \] 
  Let $u'$ be in the leftmost set. From our choice of
  $\delta$, $\pi^{-1}\{ u'\}$ is contained in  $\pi^{-1}(V)$, 
  so $u'$ is in $V \cap (G')^{u}$, which, in turn, 
  is contained in $r_{G'}(U)$.
  Hence, we know that $u' = r_{G'}(y')$, with $y'$ in
   $U$. It follows from the
  choice of $U$ that either
   $d_{G}(\pi^{-1}\{ y' \}, \pi^{-1}\{ x' \}) < \epsilon$
  or $diam_{G}(\pi^{-1}\{ y' \}) < \epsilon$. In the 
  former case, it follows 
  from the definition of $d_{H'}$
   that $y'$ is in $H'(x', \epsilon)$ and we are done.
  In the latter case, as $y'$ is in $U \subseteq U_{0}$, we have 
  $diam_{G}(r_{G}(\pi^{-1}\{ y' \}) ) < 
  3^{-1} diam_{G}( \pi^{-1}\{ r_{G'}(x') \})$.
  As $r_{G}(\pi^{-1}\{ y' \})  = \pi^{-1}\{ r_{G'}(y') \} 
  = \pi^{-1}\{ u' \}$,
   this contradicts
  the hypothesis that 
  \[
  d_{H'}(u', r_{G'}(x')) < \delta < 3^{-1}diam_{G}( \pi^{-1}\{ r_{G'}(x') \}).
  \]

 We turn to $G$ and $H$. The first two  conditions follow from Lemma 
 \ref{7:10} and Proposition \ref{7:65}. For the third part, let 
 $\mathcal{S}$ denote the usual topology on
 $G'$. Fix a unit $u$ in $H^{0}$. We know that the relative topology
  $ \mathcal{S}|_{H'}$ and the metric topology from $d_{H'}$ agree on 
  $(H')^{\pi(u)}$. On the other hand, $\pi$ induces a 
  homeomorphism between $H^{u}$, with the metric topology from
  $d_{G}$  and   $(H')^{\pi(u)}$ with the  relative topology
  $ \mathcal{S}|_{H'}$. Together, these imply that
  $\pi$ is a homeomorphism between the metric spaces $(H^{u}, d_{G})$
  and $((H')^{\pi(u)}, d_{H'})$. It follows immediately
  that the metrics $d_{G}$ and $d_{H}$ induce the same topology on
  $H^{u}$. At the same time, we see that
   $\pi: (H^{u}, d_{H}) \rightarrow ((H')^{\pi(u)}, d_{H'})$
    is a homeomorphism.
   
   Finally, to check the last condition of Theorem \ref{5:50}, we 
   must see that $r: H \rightarrow H^{0}$ is open.  If this fails,
    then there is an $x$ in $G$, an $\epsilon > 0$ and a sequence
    $u_{k}, k \geq 1$ in $H^{0} $ converging 
    to $r(x)$ in $d_{H}$, while $u_{k} \notin H(x, \epsilon), k \geq 1$.
    By continuity, $\pi(u_{k}), k \geq 1$ converges
     to $\pi(r(x))$ in $d_{H'}$.
   We know that  $r_{G'}: H' \rightarrow (H')^{0}$, so we may 
   find $y_{k}',  k \geq 1$ in $H$ with 
   $r_{G}(y_{k}') = \pi(u_{k}), k \geq 1$ converging 
   to $\pi(x)$. 
     From our standing hypotheses and Lemma  \ref{7:20},
      for every $k \geq 1$, we may 
     find $y_{k}$ in $H$ with $\pi(y_{k})  = y_{k}', k \geq 1$
     and $r_{G}(y_{k}) = u_{k}, k \geq 1$. The sequence $y_{k}, k \geq 1$
     lies in $\pi^{-1}\{ y_{k}', \pi(x) \mid k \geq 1 \}$, which 
     is compact since $\pi$ is proper (from $d_{H}$ to $d_{H'}$).
    By passing  to a subsequence, we may assume
     $y_{k}, k \geq 1$ converging to some $y$.
     On  one hand, we have 
     \[
     \pi(y) = \lim_{k} \pi(y_{k}) =\lim_{k} y'_{k} = \pi(x)
     \]
     and on the other, 
     \[
     r_{G}(y) = \lim_{k} r_{G}(y_{k}) = \lim_{k} u_{k} = r_{G}(x).
     \]
    From our standing hypotheses, this implies that $x=y$. So for some 
    $k$ sufficiently large,
    $y_{k}$ is in $H(x, \epsilon)$ and $u_{k}= r_{G}(y_{k})$, a 
    contradiction. This completes the proof.      
\end{proof}

Our next task is to find a sequence of 
approximants to $G$, which we will use in 
constructing our subalgebra $\mathcal{A}$ within $C_{c}(G)$.
For each $n \geq 1$, we define an equivalence
 relation, $\sim_{n}$, on $G$ 
as follows. For $x$ in $G$, we set 
\[
[x]_{n} = \left\{ \begin{array}{cl} \{ x \}, 
& diam_{G}(\pi^{-1}\{ \pi(x) \}) > n^{-1}  \\ 
                   \pi^{-1}\{ x \},        
                   & diam_{G}(\pi^{-1}\{ \pi(x) \}) \leq n^{-1} 
                   \end{array} \right.
                   \]
 Observe that each $\sim_{n}$-equivalence class is compact.
      We let $G_{n}$ be the quotient space $G/ \sim_{n}$ which 
   we equip with the quotient topology.
We let $q_{n}: G \rightarrow G_{n}$ denote    
quotient map and since $[x]_{n} \supseteq [x]_{n+1} $, we
let $p_{n}: G_{n} \rightarrow G_{n-1}$ be the obvious 
quotient map, for $n \geq 2$ and 
$q'_{n}: G_{n} \rightarrow G'$ be the map sending
 $[x]_{n} $ to $\pi(x)$. All of these maps are
clearly continuous. 

It will also be convenient for us to define
\[
H'_{n} =   \{ x \in G' \mid 
diam_{G}(\pi^{-1}\{ x \}) > n^{-1}  \} \subseteq H',
\]
and 
\[
H_{n} =   \pi^{-1}(H_{n}') \subseteq H,
\] 
for all $n \geq 1$.

\begin{lemma}
\label{7:90}
\begin{enumerate}
\item Each space $G_{n}$ is locally compact.
\item Each space $G_{n}$ is Hausdorff.
\item The space $G$ is the inverse limit of 
\[
G_{1} \stackrel{p_{2}}{\leftarrow} G_{2} \stackrel{p_{3}}{\leftarrow} \cdots
\]
\item 
Each set $H_{n}$ is open in $H$.
\item The closure of $H_{n}$ in $G$, is contained in $H$.
\item $H_{1} \subseteq H_{2} \subseteq \cdots$ and the union is $H$.
\end{enumerate}
\end{lemma}       

\begin{proof}
The first part follows easily from the 
following observation: if $U' \subseteq G'$
is open with compact closure, then by the continuity of
$q'_{n}$, $(q_{n}')^{-1}(U')$ is open and is contained in 
$q_{n}( \pi^{-1}(\overline{U'}))$, which is compact.

The second part follows quite easily from the following fact:
if $U \subseteq G$ is an open set, then 
\[
U_{n} = \{ [z]_{n} \mid  [z]_{n} \subseteq U \}
\] 
is open in $G_{n}$. 
To prove this, we need to show that

\[
(q_{n})^{-1}(U_{n}) = \{ z \in G \mid  [z]_{n} \subseteq U \}
\]
is open in $G$.

Let $z$ be in $(q_{n})^{-1}(U_{n})$. As $[z]_{n} $
is compact, we may find $\epsilon > 0$ such that 
\[
G([z]_{n}, \epsilon) \subseteq U.
\]
 We consider two cases. First, 
suppose that $[z]_{n} = \pi^{-1}\{\pi(z)\}$. We appeal to Lemma 
\ref{7:45} to find an open set $\pi(z) \in U' \subseteq G'$ such that
\[
 \pi^{-1}(U') \subseteq G( \pi^{-1}\{\pi(z)\}, \epsilon) \subseteq U.
 \]
 Then $z \in \pi^{-1}(U')$ is open in $G$ and, 
 if $x$ is in $\pi^{-1}(U')$, 
 then we have 
 \[
 [x]_{n} \subseteq \pi^{-1}\{ \pi(x) \} \subseteq G( \pi^{-1}\{\pi(z)\}, \epsilon) \subseteq U
 \]
 so  $\pi^{-1}(U') \subseteq (q_{n})^{-1}(U_{n})$.
 
 The second case is $[z]_{n} = \{ z \}$. It follows from the definition
 of $[z]_{n}$ that $diam_{G}(\pi^{-1}\{ \pi(z) \}) > n^{-1}$.
 Now, we find an open set  $\pi(z) \in U' \subseteq G'$ satisfying the regularity
 condition at $\pi(z)$ for 
 \[
 \epsilon'< 2^{-1}( diam_{G}(\pi^{-1}\{ \pi(z)\}) - n^{-1}).
 \] 
 We also require that $\epsilon' < 2^{-1} \epsilon$.
 Then $z \in \pi^{-1}(U') \cap G(z, \epsilon')$, 
 which is open in $G$. We claim this set is contained in $(q_{n})^{-1}(U_{n})$.
 Suppose $x$ is in $G(z, \epsilon')$ with $\pi(x)$ in $U'$. We need to show that 
 $[x]_{n} \subseteq U$. The case $[x]_{n} = \{ x\}$ is easy as $x$ is in 
$ G(z, \epsilon') \subseteq G(z, \epsilon) \subseteq U$.  In the case 
$[x]_{n} = \pi^{-1}\{ \pi(x) \}$, we then know that 
$diam_{G}(\pi^{-1}\{ \pi(x) \}) \leq n^{-1}$. From this, we see that 
\begin{eqnarray*}
d_{G}( \pi^{-1}\{ \pi(x) \}, \pi^{-1}\{ \pi(z) \} ) &
\geq  &  2^{-1}\left( diam_{G}(\pi^{-1}\{ \pi(z)\})  \right. \\
  &  &  \left. -
 diam_{G}(\pi^{-1}\{ \pi(x)\}) \right)   \\
   &  >   & \epsilon'.
 \end{eqnarray*}
 It follows from the regularity condition and the fact that $\pi(x)$ is in
 $U'$ that
 $diam_{G}(\pi^{-1}\{ \pi(x) \}) \leq \epsilon'$. 
 Then we have 
 \[
 d_{G}(\pi^{-1}\{ \pi(x) \}, z) \leq diam_{G}(\pi^{-1}\{ \pi(x)\}) + d_{G}(x,z) < 2 \epsilon' < \epsilon.
 \]
 The desired conclusion follows.

   We now prove that $H_{n}$ is open in $H$. In fact, this follows if
   we show that 
   $H_{n}'$ is open in $H'$. Let $x$ be in $H'_{n}$. Choose 
   $\epsilon > 0$ so that $diam_{G}( \pi^{-1}\{ x \}) > n^{-1} + 3 \epsilon$.
   If $y$ is in $H'$ with $d_{h}(x,y) < \epsilon$, then 
   $d_{G}( \pi^{-1}\{ x \}, \pi^{-1}\{ y \} ) < \epsilon$ and it follows
   that 
   \[
   diam_{G}(\pi^{-1}\{ y\}) \geq diam_{G}( \pi^{-1}\{ x \}) - 2 \epsilon > n^{-1}.
   \]
   Hence $y$ is in $H_{n}'$. 
   
   Now suppose that $x$ in $G$ is in the closure of $H_{n}$.
   We will show that $\pi(x)$ is in $H'$.
   Every neighbourhood of $\pi(x)$ contains a point $x'$ with 
  \newline 
   $diam_{G}(\pi^{-1}\{ x' \}) > n^{-1}$.
   On the other hand, if $\pi(x)$ is not in $H'$,
    then $\pi^{-1}\{ x \}$ is a single point. In this case, 
    Lemma \ref{7:45} implies that there is an open set $U'$ containing
    $\pi(x)$ such that $diam(\pi^{-1}(U'))$ is 
    arbitrarily small. This would be a contradiction.

   The last statement is obvious.
\end{proof}

We now begin the task of establishing the conditions
 needed to apply Theorem \ref{3:70}. 
Our ultimate goal will be to see 
$B = C^{*}_{r}(G), B \cap \ker(\delta)= C^{*}_{r}(G'),
A = C^{*}_{r}(H)$ and $A \cap \ker(\delta) = C^{*}_{r}(H')$.
 The first part is to find a Hilbert space
with actions of both $C^{*}_{r}(G)$ 
 and $C^{*}_{r}(H)$. We assume throughout that $\pi$ satisfies the 
 standing hypothesis and is regular.

We define 
\[
G \times_{\pi} G= \{ (x,y) \in G \times G \mid \pi(x) = \pi(y) \},
\]
the fibred product of $G$ with itself 
over $\pi$. This is actually a groupoid in 
an obvious fashion.  We will not need 
that fact, but it does influence our notation.
It receives the relative topology from the product. 
We denote the map sending $(x,y)$ in 
$G\times_{\pi} G$ to $\pi(x)=\pi(y)$ in $G'$ by $\pi$.
For $u, v$ in $G^{0}$ with $\pi(u) = \pi(v)$, we let 
$(G \times_{\pi} G)^{(u,v)}  = \{ (x, y) \in G^{u} \times G^{v}
 \mid \pi(x) = \pi(y) \}$. There is an analogous definition for
$ (G \times_{\pi} G)_{(u,v)}$. 
It follows from the standing hypothesis that 
$\pi|_{(G \times_{\pi} G)^{(u,v)}}$ is a 
homeomorphism to $(G')^{\pi(u)}$. We let $\nu^{(u,v)}$
 be the measure
$\nu^{\pi(u)}$ pulled back by this map. There
 is an analogous definition
of $\nu_{(u,v)}$. We represent
$C_{c}(G)$ on the Hilbert space 
$L^{2}((G \times_{\pi} G)_{(u,v)},\nu_{(u,v)})$ by
\[
( b \xi) (x,y) = \int_{(r(w), r(z)) = (r(x), r(y))}
  b(w) \xi(w^{-1}x, z^{-1}y) d\nu^{(r(x), r(y))}(w, z),
  \]
  for $b$ in $C_{c}(G)$, $\xi$ in 
  $L^{2}((G \times_{\pi} G)^{(u,v)},\nu^{(u,v)})$ and 
  $(x,y)$ in $(G \times_{\pi} G)_{(u,v)}$ It is obvious that
  this is unitarily equivalent to the left regular
  representation of $C_{c}(G)$ on 
  $L^{2}(G_{u},\nu_{u})$ and hence extends to 
  all of $C^{*}_{r}(G)$. Let us also remark at this point that if
  $x, y, w,z$ are in $G$ with $r(w) = r(x), r(y)=r(z), 
  \pi(w)=\pi(z), \pi(x)=\pi(y)$, then 
  \[
  \sigma(w, w^{-1}x) = \sigma'(\pi(w), \pi(w)^{-1}\pi(x)) =\sigma(z, z^{-1}y)
  \] 
  as a consequence
  of the last part of our standard hypothesis.
  
  We define our Hilbert space 
  \[
  \mathcal{H} = \oplus_{(u,v) \in G^{0} \times_{\pi} G^{0} }
  L^{2}((G \times_{\pi} G)_{(u,v)},\nu_{(u,v)}).
  \]
  
  We also observe that all of the preceding discussion 
  applies equally well to the groupoid $H$, provided
  $u,v$ are in $H$, which
   is implied by the condition $u \neq v$.
If $u,v$ are not in $H$, we simply represent $C_{c}(H)$ on 
  $L^{2}((G \times_{\pi} G)_{(u,v)},\nu_{(u,v)})$ as the zero representation.
  
  Observe that the operator $(F\xi)(x,y) = \xi(y,x)$ is a
   unitary from $L^{2}((G \times_{\pi} G)_{(u,v)},\nu_{(u,v)})$
  to $L^{2}((G \times_{\pi} G)_{(v,u)},\nu_{(v,u)})$. In the case $u=v$, 
  it is the identity.  We 
  denote its extension to $\mathcal{H}$ by $F$ also. 
  Observe that $F^{2}=I, F=F^{*}$. We 
  we define $\delta(a) = i[a, F] = i(aF - Fa)$, for 
  all $a$ in $C^{*}_{r}(G) + C^{*}_{r}(H)$.
  
  It is probably worth recording the following fact. Its 
  proof is trivial and we omit it.
  
  \begin{lemma}
  \label{7:95}
  Let $(u, v)$ be in $G \times_{\pi} G$, $\xi$ be in 
   $L^{2}((G \times_{\pi} G)_{(v,u)},\nu_{(v,u)})$ and 
   $b$ be in either $C_{c}(G)$ or $C_{c}(H)$. We have 
   \begin{eqnarray*}
   \delta(b) \xi(x, y) & =  & 
i \int
  \left( b(w) - b(z)\right) \xi(z^{-1}y, w^{-1}x)  \\
   &  &  \hspace{2cm} \sigma(w, w^{-1}x)
  d\nu^{(r(x), r(y))}(w, z),
  \end{eqnarray*}
  for $(x,y)$ in $(G \times_{\pi} G)_{(v,u)}$, where the integral is over
  $(w,z) \in (G \times_{\pi} G)^{(r(x), r(y))} $.
  \end{lemma}

  \begin{lemma}
  \label{7:100}
  If $b$ is in $C_{c}(G_{n}) \subseteq C_{c}(G)$, for
  some $n \geq 1$, then 
  $\delta(b) \in \delta(C_{c}(H))$.
  \end{lemma}
  
\begin{proof}
We regard $b$ as a function on $G$, constant
on the $\sim_{n}$-equivalence classes. Let $K$ be a compact set in $G$
such that $b$ is zero off $K$ and define
\[
X = \{ x \in H \cap \pi^{-1}(\pi(K))  \mid 
diam_{G} \pi^{-1}\{ \pi(x) \} \geq n^{-1} \}.
\]
 First, we claim that $X$ is compact in $H$.
If $x_{k}, k \geq 1$ is any sequence in $X$,
 then, as $K$ is compact and also $\pi^{-1}(\pi(K))$, it 
has a subsequence which converges to some $x$ in $K$ in the metric $d_{G}$.
It also follows that $diam_{G} \pi^{-1}\{ \pi(x_{k}) \} \geq n^{-1}$, for
every $k$ and so $diam_{G} \pi^{-1}\{ \pi(x) \} \geq n^{-1}$, as well.
This implies that $x$ is in $H$.  
It also follows from the fact that $\pi$ is regular that 
$\pi^{-1} \{ \pi(x_{k}) \}$ converges to 
$\pi^{-1} \{ \pi(x) \}$ and hence $x_{k}, k \geq 1$
 converges to $x$ in $d_{H}$. 
 
 Let $U'$ be an open set in $G'$, which
  contains $\pi(X)$ and 
 its closure is  
 compact. We 
 may find  $e: H' \rightarrow [0,1]$ in $C_{c}(H)$ such that 
 $e(x') =1$, for $x' \in \pi(X)$ and $e(x') = 0$, for $x' \notin U'$.
 Then for any $x$ in $H$, we define
 $a(x) = b(x)e(\pi(x))$. This function clearly has compact support. To 
 see it is continuous, it suffices to observe that
 it is non-zero only on $\pi^{-1}(\overline{U'})$, which is 
 compact in $H$ and as the 
 inclusion map is continuous, it is also compact
 in $G$ and the two relative topologies agree there.

It remains to prove $\delta(b)=\delta(a)$. 
From the formula provided by Lemma \ref{7:95}, it suffices for us to prove 
that $b(z)- b(w) = a(z) - a(w)$, for all $(w,z)$ in $G \times_{\pi} G$
(where we interpret $b$ to be zero off of $H$). If 
 $w,z$ are not in $H$, then $w=z$ and the conclusion holds.
 Next, let us suppose that $diam_{G} \pi^{-1}\{ \pi(w) \} \leq n^{-1}$. 
 As $b$ is in $C_{c}(G_{n})$, it follows that $b(w) = b(z)$ and 
 then $a(w) = b(w) e(\pi(w)) = b(z) e(\pi(z)) = a(z)$ and the conclusion holds.
 Now, let us assume that $diam_{G} \pi^{-1}\{ \pi(w) \} > n^{-1}$. 
 If $w$ is not in $\pi^{-1}(\pi(K))$, then $w$ is not in $K$ so  $b(w)=0$. 
 But this also means that $z$ is not in
  $\pi^{-1}(\pi(K))$, so $b(z)=0$, as well.
 Again, we have $a(w)=a(z)=0$. Finally, we are left with the case
 that $w$ and $z$ are both in $X$. It follows that
 \begin{eqnarray*}
 a(z) - a(w) & = &  b(z)e(\pi(z)) - b(w) e(\pi(w))  \\
    &  =    &  b(z) \cdot 1 - b(w) \cdot 1  \\
    &   =  &  b(z) - b(w),
 \end{eqnarray*}
 as $e(x')=1$, for $x'$ in $\pi(X)$.
\end{proof}  

The hypothesis of the next result is not particularly strong. We know that
the  space $H'$ has been given a topology in which the fibres of the map
$\pi$ vary continuously. This hypothesis ensures the existence on measures
on these fibres, also varying continuous.
  
  \begin{prop}
  \label{7:110}
  We say that $\pi$ is \emph{measure regular} if
  there is a continuous function 
   $\mu: (H')^{0} \rightarrow M(H^{0})$, the set 
   of Borel probability 
   measures on $H^{0}$ with the weak-* topology
  such that
  \begin{enumerate}
  \item for any $u'$ in $(H')^{0}$,
  the support of $\mu(u')$ is contained in $\pi^{-1}\{ u' \}$, 
  \item for any $x'$ in $H'$ and Borel subset $\phi$ of 
$\pi^{-1}\{ x' \}$, 
\[
\mu(r_{H'}(x'))(r_{H}(\phi))
 = \mu(s_{H'}(x'))(s_{H}(\phi)),
 \]
 \end{enumerate}
In this case, the $*$-algebra $C_{c}(H)$ satisfies condition
   C1 of Theorem \ref{3:70}. Moreover, we have 
  $\delta(C^{*}_{r}(G, \sigma)) \subseteq \delta(C^{*}_{r}(H, \sigma))$.
  \end{prop}
  
  \begin{proof}
  For simplicity, we will ignore the cocycle.
 Let $a$ be any element of $C_{c}(H)$ and define
  \[
  a'(x') = \int_{\pi^{-1}\{ r_{H'}(x') \}} a(x)
   d\mu( r_{H'}(x'))(x),
   \]
   for any $x'$ in $H'$.
  
 It is clearly  in $C_{c}(H')$. 
 Fix a pair of units in $H^{0}$, $(u,v)$ with 
 $\pi(u) = \pi(v)$. Let 
 $W_{v}$ denote the canonical unitary between 
 $L^{2}(G_{u}, \nu_{u})$ and  $L^{2}((G \times_{\pi} G)_{(u, v)}, 
   \nu_{(u,v)}$ induced by the projection
   onto the first factor.
 
 It follows from 
 Lemma \ref{7:95} that
 \[
\int_{v}  W_{v}^{*} ( a - FaF)W_{v} d\mu(u) (v)  = a- a'.
\]
The conclusion then follows from the fact that $\mu(u)$ is 
a probability measure and 
\[
\Vert a - FaF \Vert_{r} = \Vert \delta(a) \Vert_{r}.
\]
  
   For the last statement, we know that 
   $\delta(C_{c}(G))$ is contained in \newline 
   $\delta(C_{c}(H))$. As $\delta$
   is continuous, $\delta(C^{*}_{r}(G, \sigma))$ is contained in 
   the closure of $\delta(C_{c}(H))$ which is 
   $\delta(C^{*}_{r}(H, \sigma))$ as a consequence of C1.
  \end{proof}

  \begin{prop}
  \label{7:115}
  If $\pi$ is regular and there is a continuous morphism of groupoids
   $\mu: H' \rightarrow H$
  such that $\pi \circ \mu = id_{H'}$, then
 $\pi$ is measure regular.
  \end{prop}
  
  \begin{proof}
  We define a function, also denoted by $\mu$,  from $(H')^{0}$ to
  $M(H^{0})$ by setting  $\mu(u') $ to be point mass at $\mu(u')$, 
  for $u'$ in $(H')^{0}$. The hypotheses of Proposition \ref{7:110}
  are obviously satisfied. 
  \end{proof}
   
A main 
  case of interest is when these fibres are actually finite. 
   One should consider the hypotheses
  in this case to be analogous to those of the the
  condition of finite index \ref{6:140} in the 
  subgroupoid case.

   \begin{prop}
  \label{7:120}
 If  $\pi: G \rightarrow G'$ is
 regular and
  there is a positive integer $N \geq 2$ such that, 
  for each $x'$ in $H'$, $\# \pi^{-1} \{ x' \} =N$,
  then $\pi$ is measure regular.
  \end{prop} 
  
  \begin{proof}
  We define a function, also denoted by $\mu$,  from $(H')^{0}$ to
  $M(H^{0})$ by setting  
  \[
  \mu(u') = N^{-1} \sum_{u \in \pi^{-1}\{ u' \}} \mu_{u},
  \]
  where $\mu_{u}$ denotes point mass at $u$,
  for $u'$ in $(H')^{0}$. The hypotheses of Proposition \ref{7:110}
  are obviously satisfied. 
  \end{proof}
  
    \begin{lemma}
  \label{7:130}
  Suppose that $\pi: G \rightarrow G'$ is regular.
  The $*$-algebra $C_{c}(H)$ satisfies condition C2 of Theorem \ref{3:70}.
  \end{lemma}
  
  \begin{proof}
  Let $a_{1}, \ldots, a_{I}$ be in $C_{c}(H)$. We may find a 
  compact set $K \subseteq H$ such that all are zero off of $K$. 
  By the last two parts of Lemma \ref{7:90}
  that there is some $n$ with $K \subseteq H_{n}$. We may apply
  the Tietze extension Theorem to each function, $a_{i}$, restricted 
  to the  closure of $H_{n}$ in $G$ to find $b_{i}$ in $C_{c}(G_{n})$
  such that $b_{i}|_{H_{n}} = a_{i}|_{H_{n}}$. 
  We also regard these functions as being in $C_{c}(G)$ as well.
  
As $H_{n}$ is open, $r(H_{n}) \cup s(H_{n})$ is open in $H^{0}$ 
and contains $r(K) \cup s(K)$ so  
  we may find
  $e$, a function in
  $C_{c}(H^{0})$, which is identically $1$ on  $r(K) \cup s(K)$. 
  and zero outside $r(H_{n}) \cup s(H_{n})$. It is now routine to check these
  satisfy the properties in C2.
  \end{proof}

  \begin{cor}
  \label{7:135}
 If $\pi: G \rightarrow G'$ is a factor map satisfying the standing 
 hypotheses and is regular and measure regular, then 
 \[
 K_{*}(\ker(\delta) \cap C^{*}_{r}(G, \sigma);C^{*}_{r}(G, \sigma)) \cong  
 K_{*}(C^{*}_{r}(H', \sigma);C^{*}_{r}(H, \sigma)).
 \]
  \end{cor}

  We would now like to replace $\ker(\delta) \cap C^{*}_{r}(G, \sigma)$ with 
  $ C^{*}_{r}(G', \sigma)$.

  \begin{thm}
  \label{7:140}
  Assume  that $\pi$ is both regular and measure regular.
  If the closure of the sets $H_{n} \subseteq H, n \geq 1$ in $G$, 
  denoted $Cl(H_{n})$, 
   satisfy
   the $C$-extension property, for some $C \geq 1$, then we have 
 \[
  C^{*}_{r}(G, \sigma) \cap \ker(\delta)=  C^{*}_{r}(G', \sigma).
  \]
  \end{thm}
  
  \begin{proof}
  The containment $\supseteq$ is clear from  
  Lemma \ref{7:95} and the fact that $\delta$ is continuous.
  For the converse, if $b$ is in $ C^{*}_{r}(G, \sigma)$
  and $\delta(b)=0$, then we may find a sequence $b_{n}, n \geq 1$
  in $C_{c}(G)$ converging to $b$. In view of part 3 
  of Lemma \ref{7:90} (and after doing some re-indexing), we may 
  assume that $b_{n}$ is in $G_{n}$, for all $n$. it follows from Lemma 
  \ref{7:100} that there exist $a_{n}, n \geq 1$  in $C_{c}(H)$ 
  with $\delta(b_{n}) = \delta(a_{n})$. 
By  Proposition \ref{7:110}, we know that
 condition C1 of Theorem \ref{3:70} holds. By adding an element of 
 $C_{c}(H')$, we may assume that  
 \[
 \Vert a_{n} \Vert_{r} \leq K \Vert \delta(a_{n}) \Vert_{r} 
 = K \Vert \delta(b_{n}) \Vert_{r}.
 \]

  First, the condition that $\delta(b_{n})= \delta(a_{n})$ implies that
   $b_{n}|_{H} = a_{n}$.  Let $K_{n}$ be a compact subset of $H$ containing the 
   support of $a_{n}$. The sets $H_{k}, k \geq 1$, are an open 
  cover of $H$, so we may find $k_{n} \geq n$ such that 
  $H_{k_{n}}$ contains $K_{n} \cup  s(K_{n})$.
  In addition, the function sending $x$ in $H$ to 
  $diam_{G}( \pi^{-1}\{ \pi(x) \})$
  is continuous and positive, hence it is bounded below on 
  any compact set. It follows then that we may also choose $k_{n}$ so that 
  $diam_{G}( \pi^{-1}\{ \pi(x) \}) > k_{n}^{-1}$, 
  for all $x$ in  $K_{n} \cup s(K_{n})$.
   
   We use the extension property for
  $Cl(H_{k_{n}})$ to find $c$ in $C_{c}(G)$ such that 
  $c_{n}|_{H_{k_{n}}} = a_{n}|_{k_{n}}$ and 
  $\Vert c_{n} \Vert_{r} \leq C \Vert a_{n} \Vert_{r}$.

  As $b_{n} - c_{n}$ is in $C_{c}(G - G_{k_{n}})$, we may find a compact set 
  $L_{n} \subseteq H- H_{k_{n}}$ such 
  that the support of $\delta(b_{n} - c_{n})$ is in $L_{n}$.
  We claim that $\pi( s(L_{n}))$ is disjoint from 
  $\pi( s(K_{n}))$. Recall that $k_{n}$ was chosen so that
    $diam_{G}( \pi^{-1}\{ \pi(x) \}) > k_{n}^{-1}$, for all $x$ in 
     $ s(K_{n})$. If $x$ is in $L_{n}$, then 
     $diam_{G}(\pi^{-1}\{ \pi(x) \}) \leq l_{n}^{-1} < k_{n}^{-1}$.
     The fact that  $s$ is a  contraction implies that the 
     same conclusion holds for $s(x)$. This establishes the claim.
     
      We let $h_{n}: (G')^{0} \rightarrow [0,1]$ be continuous, 
      compactly supported function   
  which is  identically one on  $\pi( s(K_{n}))$ and identically zero
  on   $\pi(s(L_{n}))$. We next claim that $b_{n} - c_{n}h_{n}$ is in 
  $C_{c}(G')$. It suffices to check that 
  \[
  0 = \delta(b_{n} -c_{n}h_{n})  = \delta(b_{n}) - \delta(c_{n})h_{n},
  \]
  as $h_{n}$ is in $C_{c}((G')^{0})$. The function on the right is
  clearly supported in $K_{n} \cup L_{n}$. As $h_{n} =1$ on $\pi(s(K_{n}))$, 
  we have 
  \[
 \left( \delta(b_{n}) - \delta(c_{n})h_{n} \right)|_{K_{n}}
  = \delta(b_{n})|_{K_{n}} - \delta(c_{n})|_{K_{n}}
  = \delta(b_{n}-c_{n})|_{K_{n}} = 0
  \]
  as $K_{n} \subseteq H_{k_{n}}$ where $b_{n}$ and $c_{n}$ agree.
  On the other hand, we also know that $h_{n}$ is zero on $\pi(s (L_{n}))$ so
    \[
 \left( \delta(b_{n}) - \delta(c_{n})h_{n} \right)|_{L_{n}} 
 = \delta(b_{n})|_{L_{n}} - 0
  = 0
  \]
     since $L_{n} \subseteq H_{l_{n}}-H_{k_{n}}$.
     
  Finally, we have 
  \[
  \Vert c_{n} h_{n} \Vert_{r} \leq  \Vert c_{n}  \Vert_{r}
   \leq C \Vert a_{n} \Vert_{r} \leq 
   C K \Vert \delta(b_{n}) \Vert_{r}
  \]
  which tends to zero. Hence $b_{n} - c_{n}h_{n}$ 
  is in $C_{c}(G')$ and also converges to $b$.  
  \end{proof}
  
We will finish by giving some special examples of factor maps where
all of our hypotheses are satisfied.

 We  assume that $X$
 is a locally compact metric space with an action by the locally compact
metric 
group $\Gamma$ by homeomorphisms. We suppose that $ Y$ is a 
$\Gamma$-semi-invariant set (in the sense of \ref{5:105})
and that $\Gamma_{Y} \backslash \Gamma$
is discrete. Let us also suppose that, for any $\epsilon > 0$, the set
\[
\{ \Gamma_{Y} \gamma \in \Gamma_{Y} \backslash \Gamma
 \mid diam_{X}(Y\gamma) > \epsilon \}.
\]
is finite. 

We define the space $X'$ as the quotient of $X$ which identifies each set 
$Y \gamma, \gamma \in \Gamma$ to a single point. We let $\pi$ denote
the quotient map, $\pi: X \rightarrow X'$. From the hypotheses on the set $Y$
and on $\Gamma_{Y}$, $X'$ is locally compact and Hausdorff. 
There is an obvious
action of $\Gamma$ on $X'$ by homeomorphisms and we have a 
factor map, which we also denote by $\pi$ from $G = X \rtimes \Gamma$ 
to $G' =X' \rtimes \Gamma$.

\begin{thm}
\label{7:160}
Let $X, \Gamma, Y$ be as above and satisfy the hypotheses there.
Let $H, H'$ be the groupoids of of \ref{7:50}. 
The following hold.
\begin{enumerate}
\item
We have 
\[
C^{*}_{r}(H') \cong 
C_{0}(\Gamma_{Y} \backslash \Gamma) \rtimes_{r} \Gamma
\]
  and hence
 is Morita equivalent to $C^{*}_{r}(\Gamma_{Y})$ while 
 \[
 C^{*}_{r}(H) \cong 
 C_{0}( \cup_{\Gamma_{Y}\gamma \in \Gamma_{Y} \backslash \Gamma } Y\gamma) \rtimes_{r}
 \Gamma
 \]
 and hence is Morita equivalent to $C_{0}(Y) \rtimes_{r} \Gamma_{Y}$.
\item 
The factor map $\pi$ is regular \ref{7:60}.
\item If the action of $\Gamma_{Y}$ on $Y$ admits an 
invariant probability measure, then
$\pi$ is measure regular.
\item 
If there is a short exact sequence
   \[
0 \rightarrow C_{0}(X - Y) \rtimes_{r} \Gamma_{Y} 
  \rightarrow C_{0}(X ) \rtimes_{r} \Gamma_{Y} 
   \rightarrow C_{0}( Y) \rtimes_{r} \Gamma_{Y} 
   \rightarrow 0
  \]
then each of the closed sets $Cl(H_{n}) \subseteq H$ has the $C$-extension
 property, for any $C \geq 1$.
\end{enumerate}
\end{thm}

\begin{proof}
The descriptions of $H' \cong \Gamma_{Y} \backslash \Gamma \rtimes \Gamma$ and 
$H = C_{0}(Y \times \Gamma_{Y} \backslash \Gamma) \rtimes \Gamma$ are immediate
from the definitions. The other parts of the first part follow from
results of \cite{Ren:LNM}.

To prove the factor map is regular, we use the sum metric 
\[
d_{G}((x,\gamma), (x', \gamma')) = d_{X}(x,x') + d_{\Gamma}(\gamma,\gamma')
\]
where $d_{X}, d_{\Gamma}$ are metrics on $X$ and $\Gamma$, respectively.
Let us verify regularity. Let $\epsilon > 0$ and $(\pi(Y\gamma_{1}), \gamma_{2})$
in $H'$ be given. 
There are only finite many sets of the form $Y \gamma$ with diameter less than 
$\epsilon$, so we may choose 
$\pi(Y\gamma_{1}) \subseteq U$,  an open set in $X'$ which is disjoint 
from the images of these other than $\pi(Y\gamma_{1})$,
 which is a finite set in $X'$. Letting
 $U' = U \times \Gamma_{Y}\gamma_{2}$, there is only point in 
 $U'$ with diameter of the pre-image greater than $\epsilon$, which 
 is the single
 point $\pi(Y\gamma_{1}) \times \Gamma_{Y}\gamma_{2}$ and the desired
 conclusion is trivial.
 
Let $\mu$ be a $\Gamma_{Y}$-invariant measure on $Y$. It is a simple 
matter to check that 
\[
\mu_(Y \gamma)(E) = \mu(E \gamma^{-1}), 
\]
for $\gamma \in \Gamma, E \subseteq Y \gamma$, is a well-defined 
function from $(H')^{0}$ to $M(H^{0})$ satisfying he conditions of 
Proposition \ref{7:110}, so $\pi$ is measure regular.

The proof of the last statement is very similar to that
of Theorem \ref{6:250} and we omit the details.
\end{proof}

\begin{cor}
\label{7:170}
Let $X, \Gamma, Y$ be as above and satisfy the hypotheses there.
Assume that the action of $\Gamma_{Y}$ on $Y$ admits an 
invariant probability measure and that
there is a short exact sequence
   \[
0 \rightarrow C_{0}(X - Y) \rtimes_{r} \Gamma_{Y} 
  \rightarrow C_{0}(X ) \rtimes_{r} \Gamma_{Y} 
   \rightarrow C_{0}( Y) \rtimes_{r} \Gamma_{Y}, 
   \rightarrow 0 
  \]
  then
  \[
  K_{*}(C_{0}(X') \rtimes_{r} \Gamma;C_{0}(X) \rtimes_{r} \Gamma ) 
  \cong K_{*}(C^{*}_{r}(\Gamma_{Y}); C_{0}(Y) \rtimes_{r} \Gamma_{Y}).
  \]
\end{cor}

\bibliographystyle{amsplain}

\begin{thebibliography}{99}

\bibitem{AD:exact} C.  Anantharaman-Delaroche, \textit{ Exact groupoids}, \newline
 https://arxiv.org/pdf/1605.05117.pdf, 2016.

\bibitem{Bl:book} B. Blackadar, \textit{K-theory for operator algebras},
 MSRI Publications 5,
Springer-Verlag, New York, 1986.

\bibitem{DPS:NonH}   R.J. Deeley, I.F. Putnam and K.R. Strung, 
\textit{Non-homogeneous extensions of Cantor minimal systems},
preprint.

\bibitem{DPS:2} R.J. Deeley, I.F. Putnam and K.R. Strung, 
\textit{Constructions in minimal amenable dynamics and applications to the classification of C*-algebras},  preprint.

\bibitem{Ha:relK} M. J. Haslehurst, \textit{ Relative K-theory for 
$C^{*}$-algebras}, in preparation.

\bibitem{Ha:fac} M. J. Haslehurst, \textit{ Some examples of factor groupoids}, 
in preparation.  


\bibitem{Ka:book}
M. Karoubi, \textit{$K$-theory: An Introduction},
Springer-Verlag, \newline 
Berlin-Heidelberg-New York, 1978.



\bibitem{LT:BD} Lindsey and R. Trevi\~{n}o, 
\textit{Infinite type flat surface models of ergodic systems}, 
Disc. and Cont.
Dyn. Sys. A  36 (2016), 5509-5553.

\bibitem{MRW:grpeq} P.S. Muhly, J.N. Renault 
and D.P. Williams, \textit{Equivalence a
nd isomorphism for groupoid 
$C^{*}$-algebras},
 J. Oper. Th. 17 (1987), 3-22.


\bibitem{Ped:book}  G.K. Pedersen, 
\textit{$C^{*}$-algebras 
and their automorphism groups}, London Mathematical Society
Monographs 14, Academic Press, London 1979.

\bibitem{Ped:AN} G.K. Pedersen, \textit{Analysis NOW},
 Graduate Texts in Mathematics, Vol.
118, Springer-Verlag, Berlin-Heidelberg-New York, 1988.


\bibitem{Put:ZCan}  I.F. Putnam, \textit{The $C^{*}$-algebras 
associated 
with minimal \newline
 homeomorphisms of the Cantor set}, Pacific Journal of 
Mathematics 136(1989), 329-353.


\bibitem{Put:exc}  I.F. Putnam, \textit{An excision theorem for 
the K-theory of  
$C^{*}$-algebras}, Journal of Operator Theory 38(1997), 
151-171.
 
 
 \bibitem{Put:grpd}  I.F. Putnam, \textit{On the K-theory of 
$C^{*}$-algebras of 
principal groupoids}, Rocky Mountain Journal of Mathematics
28(1998), 1483-1518.
 
 \bibitem{Put:K} I.F. Putnam, \textit{Some classifiable groupoid C*-algebras 
 with prescribed K-theory}, Mathematische Annalen, 370 (2018), 1361-1387.
 
\bibitem{Ren:LNM} J. Renault, \textit{A Groupoid Approach to 
$C^{*}$-algebras}, Lecture Notes in
Mathematics 793, Springer, Berlin, 1980.
 
 
\bibitem{RLL:book} M. R\o rdam, F. Larsen, and J. Laustsen, 
\textit{ An Introduction to K-Theory for $C^{*}$-Algebras}, 
London Mathematical Society Student Texts 49,
 Cambridge University Press, Cambridge-New York-Melbourne, 2000.
 
 
 
\bibitem{Tu:BC} J. L. Tu, \textit{La conjecture de Baum-Connes pour les \newline
feuilletages moyennable}, K-Theory 17 (1999), no. 3, 215-264.
 
\bibitem{Wil:grpbook} D. P. Williams, \textit{A Tool
 Kit for Groupoid $C^{*}$-algebras}, American Mathematical Society, 
 Providence, 2019.


\end{thebibliography}

\end{document}